\documentclass[final,reqno]{siamart1116}

\usepackage{amsfonts}
\usepackage{graphicx,epstopdf}
\usepackage{mathrsfs}
\usepackage{bm}
\usepackage{multirow}
\usepackage{color}
\usepackage{hyperref} 
\usepackage{subfigure} 

\hypersetup{colorlinks=false} %

\usepackage{xparse} % for {{}}
\usepackage{stmaryrd} % for [[]]

\newsiamthm{claim}{Claim}
\newsiamremark{remark}{Remark}
\newsiamremark{expl}{Example}
\numberwithin{theorem}{section}

\allowdisplaybreaks

\def\vec#1{{\bf #1}}

\NewDocumentCommand{\dgal}{sO{}m}{%
	\IfBooleanTF{#1}
	{\dgalext{#3}}
	{\dgalx[#2]{#3}}%
}

\NewDocumentCommand{\dgalext}{m}{%
	\sbox0{%
		\mathsurround=0pt % just for safety
		$\left\{\vphantom{#1}\right.\kern-\nulldelimiterspace$%
	}%
	\sbox2{\{}%
	\ifdim\ht0=\ht2
	\{\kern-.625\wd2 \{#1\}\kern-.625\wd2 \}%
	\else
	\left\{\kern-.7\wd0\left\{#1\right\}\kern-.7\wd0\right\}%
	\fi
}

\NewDocumentCommand{\dgalx}{om}{%
	\sbox0{\mathsurround=0pt$#1\{$}%
	\sbox2{\{}%
	\ifdim\ht0=\ht2
	\{\kern-.625\wd2 \{#2\}\kern-.625\wd2 \}%
	\else
	\mathopen{#1\{\kern-.7\wd0 #1\{}
	#2
	\mathclose{#1\}\kern-.7\wd0 #1\}}
	\fi
}

\def\jump#1{\llbracket #1 \rrbracket }

\makeatletter
\DeclareFontFamily{OMX}{MnSymbolE}{}
\DeclareSymbolFont{MnLargeSymbols}{OMX}{MnSymbolE}{m}{n}
\SetSymbolFont{MnLargeSymbols}{bold}{OMX}{MnSymbolE}{b}{n}
\DeclareFontShape{OMX}{MnSymbolE}{m}{n}{
	<-6>  MnSymbolE5
	<6-7>  MnSymbolE6
	<7-8>  MnSymbolE7
	<8-9>  MnSymbolE8
	<9-10> MnSymbolE9
	<10-12> MnSymbolE10
	<12->   MnSymbolE12
}{}
\DeclareFontShape{OMX}{MnSymbolE}{b}{n}{
	<-6>  MnSymbolE-Bold5
	<6-7>  MnSymbolE-Bold6
	<7-8>  MnSymbolE-Bold7
	<8-9>  MnSymbolE-Bold8
	<9-10> MnSymbolE-Bold9
	<10-12> MnSymbolE-Bold10
	<12->   MnSymbolE-Bold12
}{}

\let\llangle\@undefined
\let\rrangle\@undefined
\DeclareMathDelimiter{\llangle}{\mathopen}%
{MnLargeSymbols}{'164}{MnLargeSymbols}{'164}
\DeclareMathDelimiter{\rrangle}{\mathclose}%
{MnLargeSymbols}{'171}{MnLargeSymbols}{'171}
\makeatother
\title{
Entropy Symmetrization and High-Order Accurate Entropy Stable Numerical Schemes for Relativistic MHD Equations
}

\author{
  Kailiang Wu\thanks{Department of Mathematics, Southern University of Science and Technology, 
  	Shenzhen, Guangdong 518055, China  ({\tt wukl@sustech.edu.cn}).}
  \and
  Chi-Wang Shu\thanks{Division of Applied Mathematics, Brown University, Providence, RI 02912, USA ({\tt Chi-Wang\_Shu@brown.edu}). Research is supported in part by NSF grant DMS-1719410. }  
}

\headers{Entropy Stable Schemes for RMHD Equations}
{Kailiang Wu and Chi-Wang Shu}

\ifpdf
\hypersetup{ pdftitle={Entropy Stable Schemes for RMHD} }
\fi

\begin{document}

\maketitle

%% ------------------------------------------------------------------
%% ABSTRACT
%% ------------------------------------------------------------------

\begin{abstract}
		This paper presents entropy symmetrization and high-order accurate entropy stable schemes for the relativistic magnetohydrodynamic (RMHD) equations. 
It is shown that the conservative RMHD equations are not 
symmetrizable and do not admit 
a thermodynamic entropy pair. To address this issue, a symmetrizable RMHD system, equipped with a convex thermodynamic entropy pair, 
is first proposed by adding a source term into the equations, providing an analogue to 
the non-relativistic Godunov--Powell system.   
Arbitrarily high-order accurate entropy stable finite difference schemes are developed on Cartesian meshes based on the symmetrizable RMHD system. 
The crucial ingredients of these schemes include (\romannumeral1)  
affordable explicit entropy conservative fluxes which are technically derived through carefully selected parameter variables,  
(\romannumeral2) a special high-order discretization of the source term in the symmetrizable RMHD system, and 
(\romannumeral3) suitable high-order dissipative operators based on essentially non-oscillatory reconstruction to ensure the entropy stability.  
Several numerical tests demonstrate the accuracy and robustness of the proposed entropy stable schemes. 
\end{abstract}



\section{Introduction} 
\label{sec:intro}

Magnetohydrodynamics (MHDs) describe the dynamics of electrically-conducting
fluids in the presence of magnetic field and play an important role in many fields
including astrophysics, plasma physics and space physics. 
In many cases,  astrophysics and high energy physics 
often 
involve 
fluid flow at nearly speed of light so that the relativistic effect should be taken into account. 
Relativistic MHDs (RMHDs) have applications in investigating %a number of 
astrophysical scenarios from stellar to galactic scales, e.g., gamma-ray bursts,  astrophysical jets,
core collapse super-novae, 
formation
of black holes, and merging of compact binaries.

In the $d$-dimensional space, the governing equations of special RMHDs
can be written as a system of hyperbolic conservation laws
\begin{equation}\label{eq:RMHD}
\frac{{\partial {\bf U}}}{{\partial t}} 
+ \sum_{i=1}^d \frac{\partial {\bf F}_i ( {\bf U} ) }{\partial x_i}  = {\bf 0},
\end{equation}
along with an additional divergence-free condition on the magnetic field
\begin{equation}\label{eq:DivB}
\nabla \cdot {\bf B} := \sum_{i=1}^d \frac{\partial B_i}{\partial x_i}  =0,
\end{equation}
where $d=1,$ $2$ or $3$. Here we employ the geometrized unit system so that the speed of light $c$ in vacuum is equal to one. 
In \eqref{eq:RMHD}, the conservative vector $\vec U = ( D,\vec m^\top,\vec B^\top,E )^{\top}$,  and the flux in the $x_i$-direction is defined by
\begin{equation*}
\vec F_i(\vec U) = \left( D v_i,  v_i \vec m^\top  -  B_i \big( \gamma^{-2} \vec B^\top + (\vec v \cdot \vec B) \vec v^\top \big)  + p_{tot}  \vec e_i^\top, v_i \vec B^\top - B_i \vec v^\top ,m_i \right)^{\top},
\end{equation*}
with the mass density $D = \rho \gamma$, the momentum density vector $\vec m = (\rho h{\gamma^2} + |\vec B|^2) \vec v - (\vec v \cdot \vec B)\vec B$, the energy density $E=\rho h \gamma^2 - p_{tot} +|\vec B|^2$, and
the vector $\vec e_i$ denoting the $i$-th column of the unit matrix of size $3$.
Additionally, $\rho$ is the rest-mass density, $\vec v=(v_1,v_2,v_3)^\top$ denotes
the fluid velocity vector, $\gamma=1/\sqrt{1- |{\bf v}|^2}$ is the Lorentz factor, $p_{tot}$ is the total pressure containing the gas pressure $p$ and  magnetic pressure $p_m:=\frac12 \left(\gamma^{-2} |\vec B|^2 +(\vec v \cdot \vec B)^2 \right)$,
$h = 1 + e + \frac{p}{\rho}$ represents the specific enthalpy, and $e$ is the specific internal energy. 
We consider the ideal equation of state  
$p = (\Gamma-1) \rho e$ to close the system \eqref{eq:RMHD},  
where $\Gamma$ is constant and denotes the adiabatic index.

The system \eqref{eq:RMHD} involves strong
nonlinearity, making its analytic treatment quite difficult. Numerical simulation is a primary approach to explore physical laws in RMHDs. 
The numerical study of RMHDs may 
	date back to the 1990s, to the best of our knowledge. 
	For instance, the time-dependent morphological evolution of a RMHD jet
	was computed in \cite{van1996knots}, based on a divergence-free
	formation of the Maxwell equations coupled to fluids \cite{van1991maxwell}
	with the numerical method proposed in \cite{van1995two}.  
During the past few decades, various numerical schemes have been developed for the RMHD equations, including but not limited to:  
the total variation diminishing scheme \cite{Balsara2001},
adaptive mesh methods \cite{Host:2008,HeTang2012RMHD}, 
discontinuous Galerkin methods \cite{Zanotti2015,ZhaoTang2017}, 
physical-constraints-preserving schemes \cite{WuTangM3AS}, 
entropy limited approach \cite{guercilena2017entropy}, etc.  
Systematic review of numerical RMHD schemes is beyond the scope of the present paper, and we 
	refer interested readers to the review articles \cite{font2008,Marti2015}. 
Besides the
standard difficulty in solving the nonlinear hyperbolic systems,
an additional numerical challenge for the RMHD system \eqref{eq:RMHD} comes from
the divergence-free condition \eqref{eq:DivB}, which is also involved in the non-relativistic ideal MHD system. 
Numerical preservation of \eqref{eq:DivB} 
is highly non-trivial (for $d\ge 2$) but crucial for the robustness of numerical  computations. 
Numerical experiments and analysis in the non-relativistic MHD case indicated
that violating the divergence-free condition \eqref{eq:DivB} %of magnetic field
may lead to numerical instability and nonphysical solutions \cite{Brackbill1980,Balsara2004,Wu2017a}. %,Balsara2012
Various numerical techniques were proposed to reduce or control the effect of divergence error; see, e.g.,  
\cite{Evans1988,Powell1995,Toth2000,Dedner2002,Balsara2004,Li2005,Li2011,WuShu2018,WuShu2019}.

Due to the nonlinear hyperbolic nature of the RMHD equations \eqref{eq:RMHD}, solutions of \eqref{eq:RMHD} can be discontinuous with the presence of shocks or contact discontinuities. This leads to the consideration of weak solutions. However, weak solutions may not be unique. 
To select the ``physically relevant'' solution among all weak solutions, entropy conditions are usually imposed as the admissibility criterion. 
In the case of RMHD equations \eqref{eq:RMHD}, there is a natural entropy condition arising from the second law of thermodynamics which should be respected. 
It is natural to seek {\em entropy stable} numerical schemes which satisfy a discrete version of that entropy condition. Entropy stable numerical methods ensure that the entropy is conserved in smooth regions and dissipated across discontinuities. Thus, the numerics precisely follow the physics of the second law of thermodynamics and can be more robust. 
Moreover, entropy stable schemes also allow one to limit the amount of dissipation added to the schemes to guarantee the entropy stability. 
For the above reasons, developing entropy stable schemes for RMHD equations \eqref{eq:RMHD} is meaningful and highly desirable. 

Entropy stability analysis was well studied for first-order accurate schemes and scalar conservation laws 
\cite{crandall1980monotone,harten1976finite,osher1984riemann,osher1988convergence}.  
For systems of hyperbolic conservation laws, 
much attention was paid to 
 explore entropy stable 
schemes with entropy stability %often
focused on single given entropy function.  
Tadmor \cite{tadmor1987numerical,tadmor2003entropy} established the framework of entropy conservative fluxes, which conserves entropy locally, and entropy stable fluxes for second-order schemes. 
Lefloch, Mercier and Rohde \cite{lefloch2002fully} proposed a procedure to construct higher-order accurate entropy conservative fluxes. 
Fjordholm, Mishra and Tadmor \cite{fjordholm2012arbitrarily} developed a methodology for constructing high-order accurate entropy stable schemes, which combine  high-order entropy conservative fluxes and suitable numerical dissipation operators based on essentially non-oscillatory (ENO) reconstruction that satisfies the sign property \cite{fjordholm2013eno}.  
On the other hand, high-order entropy stable schemes have also been constructed via the summation-by-parts (SBP) procedure  
 \cite{fisher2013high,carpenter2014entropy,gassner2013skew}. 
Entropy stable space–time discontinuous Galerkin (DG) schemes were studied in \cite{Barth1998,Barth2006,hiltebrand2014entropy}, where the exact integration is required for the proof of entropy stability. 
More recently, a framework for designing entropy stable high-order DG methods through suitable numerical quadrature was proposed in \cite{chen2017entropy}, where the SBP operators established in  \cite{fisher2013high,carpenter2014entropy,gassner2013skew} were used and also generalized to triangles. 
There are other studies that address various aspects of entropy stability, including but not limited to \cite{bouchut1996muscl,fjordholm2016sign,biswas2018low,hesthaven2019entropy}. 
As a key ingredient in designing entropy stable schemes, the construction of affordable entropy conservative fluxes has received much attention. 
Although there is a general way to construct 
entropy conservative fluxes based on path integration \cite{tadmor1987numerical,tadmor2003entropy}, 
the resulting fluxes may not have an explicit formula and can be computationally expensive. 
Explicit entropy conservative fluxes were derived for 
the Euler equations \cite{ismail2009affordable,chandrashekar2013kinetic,ranocha2018comparison}, 
shallow water equations \cite{gassner2016well}, 
special relativistic hydrodynamics without magnetic field \cite{duan2019high}, 
%,   
and ideal non-relativistic MHD equations \cite{Chandrashekar2016,winters2016affordable}, etc. 
Different from the Euler equations, the conservative non-relativistic MHD equations 
are not symmetrizable and do not admit an entropy  \cite{Godunov1972,Barth1998,Chandrashekar2016}. 
Entropy symmetrization can be achieved by a modified system with an additional source \cite{Godunov1972,Barth1998}. 
Based on the modified formulation, several entropy stable schemes were developed for 
non-relativistic MHDs; see, e.g.,  \cite{Chandrashekar2016,derigs2016novel,winters2016affordable,LiuShuZhang2017,DERIGS2018420}.%

This paper aims to present entropy analysis and
high-order accurate entropy stable 
schemes for the RMHD equations on Cartesian meshes.  
The difficulties in this work 
mainly come from 
the highly nonlinear coupling between RMHD equations \eqref{eq:RMHD}, which leads to no explicit expression of the primitive variables $(\rho,{\bf v},p) $, entropy variables and fluxes $\vec F_i$ in terms of $\vec U$. 
The effort and findings in this work include the following: 

\begin{enumerate}
	\item We show that the conservative form \eqref{eq:RMHD} of RMHD equations is not symmetrizable and thus does not admit a thermodynamic entropy pair.  
	%arising from thermodynamics. 
	A modified RMHD system is proposed by building the divergence-free condition \eqref{eq:DivB} into 
	the equations through adding a source term, which is proportional to $\nabla \cdot {\bf B}$. 
	The modified RMHD system is symmetrizable, possesses a convex thermodynamic entropy pair, and 
	is an analogue to 
	the non-relativistic symmetrizable MHD system proposed by Godunov \cite{Godunov1972} and Powell \cite{Powell1995}. 
	The Godunov-type entropy symmetrization procedure was also briefly described in \cite{trakhinin2001stability} and 
		was first performed in \cite{ruggeri1981convex} for RMHDs.\footnote{We thank Professor Yuri Trakhinin of Sobolev Institute of Mathematics for drawing our attention to references \cite{ruggeri1981convex,trakhinin2001stability} and for his explanations of these works which we were not aware of.}
	
	\item We derive a consistent two-point entropy conservative numerical flux for the symmetrizable RMHD equations. 
	The flux has an explicit analytical formula, is easy to compute and thus is affordable. 
	The key is to carefully select a set of parameter variables which can explicitly express 
	the entropy variables and potential fluxes in simple form. 
	Due to the presence of the source term in the symmetrizable RMHD system, the standard framework \cite{tadmor2003entropy} of the entropy conservative flux is not 
	applicable here and should be modified to take the effect of the source term into account. 
	\item We construct semi-discrete high-order accurate entropy conservative  schemes and  entropy stable  schemes for %one- and two-dimensional 
	symmetrizable RMHD equations on Cartesian meshes.  
	The second-order entropy conservative schemes are built on the proposed two-point entropy conservative  flux and a central difference discretization to the source term. 
	Higher-order entropy conservative schemes are constructed by suitable linear combinations \cite{lefloch2002fully} of the proposed two-point entropy conservative flux. 
	To ensure the entropy conservative property and high-order accuracy, 
	a suitable discretization of the symmetrization source term should be employed, which is a key ingredient in these high-order schemes. 
	Arbitrarily high-order accurate  
	entropy stable schemes are obtained by adding suitable dissipation terms into the 
	entropy conservative schemes.  
\end{enumerate}
	It is noticed that high-order accurate entropy stable schemes were 
	developed in a very recent work \cite{duan2019high} for the relativistic hydrodynamics (RHDs) {\em without the magnetic field}, 
	whose governing equations are sometimes also called the relativistic Euler equations. 
	To the best of our knowledge, the present paper is the first to study   
	entropy stable schemes for the RMHD equations.  
	Note that, in the RMHD system, the magnetic field, evolved by  
	 the induction equations, 
	is also nonlinearly involved in the momentum and energy equations. 
	Due to the presence of 
	magnetic field, mathematical structures  
	 of the RMHD system  
	are more complicated than the relativistic Euler system, making 
	the present derivations of entropy stable schemes  
	more difficult. 
	Different from the relativistic Euler equations whose conservative form admits 
	a thermodynamic entropy pair \cite{duan2019high}, 
	the conservative form \eqref{eq:RMHD} of the RMHD system 
	does not admit the thermodynamic entropy pair, 
	so that entropy symmetrization of the RMHD system must be derived to 
	accommodate an entropy condition on the PDE level. 
	Moreover, the symmetrization source term added in the modified RMHD system 
	also renders additional technical challenges on the numerical level, since  
	a suitable high-order discretization of that source term must be 
	devised such that the resulting numerical schemes satisfy entropy stability. 
	These make the exploration in this work significantly different from those in \cite{duan2019high} and highly nontrivial.
	When the magnetic field is zero, the proposed entropy stable 
	schemes for the RMHDs reduce to a class of entropy stable schemes for the RHDs. 
	It is also worth mentioning that our entropy conservative fluxes are derived through 
	a set of carefully selected parameter variables, which are different from those used in 
	 \cite{duan2019high} for the RHDs, rendering the expression of our resulting fluxes %different and 
	 simpler; see Remark \ref{rem:Simple} for details.

The paper is organized as follows. After giving the entropy analysis in Sect.~\ref{sec:EntropyAnal}, we derive the explicit two-point entropy conservative fluxes in Sect.~\ref{sec:flux}. One-dimensional (1D) entropy conservative schemes and entropy stable schemes are constructed in Sect.~\ref{sec:1Dschemes}, and 
their extensions to two dimensions (2D) are presented in Sect.~\ref{sec:2Dschemes}. 
We conduct numerical tests in Sect.~\ref{sec:examples} to verify the performance of the proposed high-order accurate entropy stable schemes, before concluding the paper in Sect.~\ref{sec:conclusion}.

\section{Entropy Analysis} \label{sec:EntropyAnal}

First, let us recall the definition of an entropy function.

\begin{definition}
	A convex function ${\mathcal E} ({\bf U})$ is called an entropy for the system \eqref{eq:RMHD} if there exist entropy fluxes ${\mathcal Q}_i ({\bf U})$ such that 
	\begin{equation}\label{DefEntropyPair}
	{\mathcal Q}'_i ({\bf U}) = {\mathcal E}' ({\bf U}) {\bf F}'_i ( {\bf U} ), \qquad i=1,\dots,d,
	\end{equation}
	where the gradients ${\mathcal E}' ({\bf U})$ and ${\mathcal Q}'_i ({\bf U})$ are written as row vectors, and 
	${\bf F}'_i ( {\bf U} )$ is the Jacobian matrix. The functions $({\mathcal E},{\mathcal Q}_i)$ form an entropy pair. 
\end{definition}

If a hyperbolic system of conservation laws 
admits an entropy pair, then the smooth solutions of the system should satisfy  
\begin{equation*}
0 = {\mathcal E}' ({\bf U}) \left(  \frac{{\partial {\bf U}}}{{\partial t}} 
+ \sum_{i=1}^d \frac{\partial {\bf F}_i ( {\bf U} ) }{\partial x_i} \right) 
= {\mathcal E}' ({\bf U})  \frac{{\partial {\bf U}}}{{\partial t}} + \sum_{i=1}^d {\mathcal Q}'_i ({\bf U}) 
\frac{\partial {\bf U}} {\partial x_i} = \frac{\partial {\mathcal E}  }{\partial t} + \sum_{i=1}^d \frac{\partial {\mathcal Q}_i} {\partial x_i}.
\end{equation*}
Solutions of a nonlinear hyperbolic system can be discontinuous. 
This leads to the consideration of weak solutions, which, however, may not be unique, and 
for (non-smooth) weak solutions the above identity does not hold in general. 
The following inequality (see, e.g., \cite[page 83]{dafermos2000hyperbolic}) is usually imposed as an admissibility criterion to select the ``physically relevant'' solution among all weak solutions: 
\begin{equation}\label{EntropyIEQ}
\frac{\partial  {\mathcal E} }{\partial t} + \sum_{i=1}^d \frac{\partial {\mathcal Q}_i} {\partial x_i} \le 0,   
\end{equation}
which is interpreted in the sense of distribution and known as an entropy condition.

The existence of an entropy pair is closely related to the symmetrization of a hyperbolic system of conservation laws \cite{godlewski2013numerical}.

\begin{definition}
	The system \eqref{eq:RMHD} is said to be symmetrizable if there exists a change of variables ${\bf U} \to {\bf W}$ which symmetrizes it, that is, the equations \eqref{eq:RMHD} become 
	\begin{equation} 
	\frac{\partial {\bf U}}{ \partial {\bf W} }\frac{\partial {\bf W}} {\partial t} + \sum_{i=1}^d \frac{\partial {\bf F}_i} {\partial {\bf W}} \frac{\partial {\bf W}}{\partial x_i} = {\bf 0},
	\end{equation} 
	where the matrix $\frac{\partial {\bf U}}{ \partial {\bf W}} $ is symmetric positive definite and 
	$\frac{\partial {\bf F}_i} {\partial {\bf W}}$ is symmetric for all $i$. 
\end{definition}

\begin{lemma}\label{lem:SYM_E}
	A necessary and sufficient condition for the system \eqref{eq:RMHD} to possess a strictly convex entropy $ {\mathcal E} ({\bf U})$ is that there exists a change of dependent variables ${\bf U} \to {\bf W}$ which symmetrizes \eqref{eq:RMHD}. 
\end{lemma}

The proof of Lemma \ref{lem:SYM_E} can be found in \cite{godlewski2013numerical}.

\subsection{Entropy Function for RMHD Equations}

It is natural to ask whether the RMHD equations \eqref{eq:RMHD} admit an entropy pair or not.

Let us consider the thermodynamic entropy 
$
S=\ln ( p  \rho^{-\Gamma} ).
$
For smooth solutions, the RMHD equations \eqref{eq:RMHD} can be used to derive an equation for $\rho \gamma S$:
\begin{equation}\label{WKL11}
\frac{ \partial (\rho \gamma S)}{\partial t} + \nabla \cdot (\rho \gamma S {\bf v}) 
+ (\Gamma - 1) \frac{ \rho \gamma ( {\bf v} \cdot {\bf B} ) } {p}  \nabla \cdot {\bf B} = 0.
\end{equation}
Then under the divergence-free condition \eqref{eq:DivB}, the following quantities 
\begin{equation}\label{EntropyPair}
{\mathcal E} ({\bf U}) = - \frac{\rho \gamma S}{\Gamma -1}, \qquad {\mathcal Q}_i ({\bf U}) = - \frac{\rho \gamma S v_i }{\Gamma - 1}
\end{equation}
satisfy an additional conservation law, thus ${\mathcal E}$ may be an entropy function.

\begin{theorem}
	The entropy variables corresponding to the entropy function ${\mathcal E}$ are
	\begin{equation}\label{DefW}
	{\bf W} =  {\mathcal E}'({\bf U})^\top = \left( 
	\frac{\Gamma - S}{\Gamma - 1} + \frac{\rho }{p},~ \frac{\rho  \gamma }{p} {\bf v}^\top,~ 
	\frac{\rho  \gamma }{p} \Big( (1-|{\bf v}|^2) {\bf B}^\top + ({\bf v}\cdot {\bf B}) {\bf v}^\top \Big),~ -\frac{\rho \gamma}{p} \right)^\top.
	\end{equation}
\end{theorem}

\begin{proof}
	Since ${\mathcal E}$ cannot be explicitly expressed by ${\bf U}$, direct derivation of ${\mathcal E}'({\bf U})$ can be quite difficult. 
	Here we consider the following primitive variables %as the transitional variables
	\begin{equation}\label{priV}
	{\bf V} = \left(  \rho, {\bf v}^\top, {\bf B}^\top, p \right)^\top. 
	\end{equation}
	%where ${\bf u} = (u_1,u_2,u_3)^\top:=W {\bf v}$ satisfy $W=\sqrt{1+|{\bf u}|^2}$. 
	As ${\mathcal E}$ and $\bf U$ can be explicitly formulated in terms of ${\bf V}$, it is easy to derive that 
	\begin{align} \nonumber
	&	{\mathcal E}'({\bf V}) = \frac1{\Gamma - 1} \left( \gamma( \Gamma - S ) ,~ - \rho S \gamma^3 {\bf v}^\top,~ 0,~ 0,~ 0,~ - \frac{ \rho \gamma }{ p } \right),
	\\
	\label{PUPV}
	& \frac{\partial {\bf U}}{\partial {\bf V}} = \begin{pmatrix}
	\gamma & \rho \gamma^3 {\bf v}^\top & {\bf 0}_3^\top & 0 
	\\
	\gamma^2{\bf v} & (  \rho h \gamma^2 + |{\bf B}|^2 ) {\bf I}_3 - {\bf B} {\bf B}^\top + 2\rho h \gamma^4 {\bf v} {\bf v}^\top  &  - ({\bf v}\cdot {\bf B}) {\bf I}_3 
	+ 2 {\bf v} {\bf B}^\top - {\bf B} {\bf v}^\top & \frac{ \Gamma \gamma^2  }{\Gamma - 1} {\bf v}
	\\
	{\bf 0}_3 & {\bf O}_3 & {\bf I}_3 & {\bf 0}_3 
	\\
	\gamma^2 & (2 \rho h \gamma^4 + |{\bf B}|^2) {\bf v}^\top - ({\bf v} \cdot {\bf B}) {\bf B}^\top & ( 1 + |{\bf v}|^2 ) {\bf B}^\top - ({\bf v}\cdot {\bf B}) {\bf v}^\top & \frac{\Gamma \gamma^2}{\Gamma - 1} -1 
	\end{pmatrix},
	\end{align}
	where ${\bf 0}_3=(0,0,0)^\top$, and ${\bf O}_3$ and ${\bf I}_3$ denote the zero square matrix and the identity matrix of size $3$, respectively. 
	One can verify that the vector $\bf W$ satisfies 
	${\bf W}^\top \frac{\partial {\bf U}}{\partial {\bf V}}  = {\mathcal E}'({\bf V}), $ 
	which implies 
	$
	{\bf W}^\top =  {\mathcal E}'({\bf V}) \left( \frac{\partial {\bf U}}{\partial {\bf V}} \right)^{-1} = {\mathcal E}'({\bf V})\frac{\partial {\bf V}}{\partial {\bf U}} = {\mathcal E}'({\bf U}).
	$
	%The proof is complete. 
\end{proof}

However, the change of variable $ {\bf U} \to {\bf W} $ fails to symmetrize the RMHD equations \eqref{eq:RMHD}, and the functions $({\mathcal E}, {\mathcal Q}_i )$ defined in \eqref{EntropyPair} do not form an entropy pair for the RMHD equations \eqref{eq:RMHD}; see the following theorem.

\begin{theorem}
	For RMHD equations \eqref{eq:RMHD} and the entropy variables $\bf W$, one has 
	\begin{enumerate}
		\item the change of variable $ {\bf U} \to {\bf W} $ fails to symmetrize the RMHD equations \eqref{eq:RMHD}, i.e., the matrix $\frac{\partial {\bf F}_i }{\partial {\bf W}}$ is not symmetric in general.
		\item  the functions $({\mathcal E}, {\mathcal Q}_i )$ defined in \eqref{EntropyPair} satisfy 
		\begin{equation}\label{NotPair}
		{\mathcal Q}'_i ({\bf U}) = {\mathcal E}' ({\bf U}) {\bf F}'_i ( {\bf U} ) 
		+ \frac{\rho \gamma}{p} ( {\bf v} \cdot {\bf B} ) B_i'({\bf U}), \quad i=1,\dots,d,
		\end{equation}
		which implies that $({\mathcal E}, {\mathcal Q}_i )$ do not satisfy the condition  \eqref{DefEntropyPair}. 
	\end{enumerate}
\end{theorem} 

\begin{proof}
	We only prove the conclusions for $i=1$, as the proofs for $2\le i \le d$ are similar. 
	Let us first show that $\frac{\partial {\bf F}_i }{\partial {\bf W}}$ is not symmetric in general. 
	Because ${\bf F}_1$ cannot be formulated explicitly in terms of $\bf W$, we calculate the Jacobian matrix ${\bf F}_1({\bf W})$ with the aid of primitive variables ${\bf V}$.  
	The Jacobian matrix  $\frac{\partial {\bf F}_1}{\partial {\bf V}}$ is computed as 
	\begin{equation} \label{PF1PV}
	\frac{\partial {\bf F}_1}{\partial {\bf V}} = 
	v_1  \frac{\partial {\bf U}}{\partial {\bf V}} + 
	\begin{pmatrix}
	0 & \rho \gamma {\bf e}_1^\top  & {\bf 0}_3^\top & 0
	\\
	{\bf 0}_3 & ( \rho h \gamma^2 + |{\bf B}|^2 ) {\bf v} {\bf e}_1^\top  +  {\bf M}_1 & {\bf M}_2 & {\bf e}_1
	\\
	{\bf 0}_3 & {\bf B} {\bf e}_1^\top - B_1 {\bf I}_3 & -{\bf v} {\bf e}_1^\top & {\bf 0}_3
	\\
	0 &    ( \rho h \gamma^2 + |{\bf B}|^2  ) {\bf e}_1^\top - B_1 {\bf B}^\top + v_1 \left( ({\bf v}\cdot {\bf B}) {\bf B} - |{\bf B}|^2 {\bf v} \right)^\top  & \bm \beta_1^\top- ({\bf v}\cdot {\bf B}) {\bf e}_1^\top 
	& v_1
	\end{pmatrix},
	\end{equation}
	with ${\bf e}_1=(1,0,0)^\top$, ${\bm \beta}_1 = v_1 (1-|{\bf v}|^2) {\bf B}^\top + v_1  ({\bf v}\cdot {\bf B}) {\bf v}^\top  - B_1 {\bf v}^\top$, and 
	\begin{align*} 
	&{\bf M}_1 = 
	{\bf e}_1 \big( ( {\bf v} \cdot {\bf B} ) {\bf B}^\top - |{\bf B}|^2 {\bf v}^\top \big) + B_1 ( 2 {\bf B} {\bf v}^\top - {\bf v} {\bf B}^\top ) - ({\bf v} \cdot {\bf B} ) ( B_1 {\bf I}_3 + {\bf B} {\bf e}_1^\top )
	\\
	& {\bf M}_2 = (1-|{\bf v}|^2) ( {\bf e}_1 {\bf B}^\top  - {\bf B} {\bf e}_1^\top ) + ( {\bf v} \cdot {\bf B} ) ( {\bf e}_1 {\bf v}^\top - {\bf v} {\bf e}_1^\top )  
	-B_1 {\bf v} {\bf v}^\top - B_1 ( 1- |{\bf v}|^2 ) {\bf I}_3.
	\end{align*}
	Since ${\bf W}$ can be explicitly expressed in terms of $\bf V$, one can derive that
	\begin{equation} 
	\frac{\partial {\bf W}}{\partial {\bf V}} =  \begin{pmatrix}
	h/p & {\bf 0}_3^\top & {\bf 0}_3^\top & - \frac{\rho}{p^2} - \frac{1}{p(\Gamma - 1)}
	\\
	\frac{\gamma}{p} {\bf v} & \frac{\rho \gamma^3}{p} {\bf M}_3   & {\bf O}_{3} & - \frac{\rho \gamma }{p^2} {\bf v} 
	\\
	\frac{\gamma}{p} ( (1-|{\bf v}|^2) {\bf B}  + ( {\bf v} \cdot {\bf B} ) {\bf v}) & \frac{\rho \gamma^3}{p} {\bf M}_4 & \frac{\rho \gamma}{p}  {\bf M}_3 & - \frac{\rho \gamma}{p^2} ( (1-|{\bf v}|^2) {\bf B}  + ( {\bf v} \cdot {\bf B} ) {\bf v}) 
	\\
	-\frac{\gamma}p & -\frac{\rho \gamma^3}p {\bf v}^\top & {\bf 0}_3^\top & \frac{\rho \gamma}{p^2}
	\end{pmatrix},
	\end{equation}	
	with
	$ {\bf M}_3 = (1-|{\bf v}|^2) {\bf I}_3 + {\bf v} {\bf v}^\top,
	{\bf M}_4 = ({\bf v}\cdot {\bf B}) {\bf v} {\bf v}^\top + \gamma^{-2} \left[ ({\bf v}\cdot {\bf B}) {\bf I}_3  + {\bf v} {\bf B}^\top - {\bf B} {\bf v}^\top \right].
	$
	Then, we obtain $\frac{\partial {\bf V}}{\partial {\bf W}}$ by the inverse of the matrix $\frac{\partial {\bf W}}{\partial {\bf V}}$, i.e.
	\begin{equation}\label{PVPW}
	\frac{\partial {\bf V}}{\partial {\bf W}} =  \begin{pmatrix}
	\rho &   \left( \rho + \frac{p}{\Gamma - 1} \right) \gamma {\bf v}^\top &  {\bf 0}_3^\top & \left( \rho + \frac{p}{\Gamma - 1} \right) \gamma
	\\
	{\bf 0}_3 & \frac{p }{\rho \gamma} {\bf I}_3 & {\bf O}_3 & \frac{p }{\rho \gamma} {\bf v}
	\\
	{\bf 0}_3 & \frac{ p \gamma }{\rho } {\bf M}_5 &  \frac{ p \gamma }{\rho }  \left( {\bf I}_3 - {\bf v} {\bf v}^\top \right) 
	& \frac{ p \gamma }{\rho }  \left(  (1+|{\bf v}|^2) {\bf B} - 2 ({\bf v}\cdot {\bf B}) {\bf v} \right) 
	\\
	p & p h \gamma {\bf v}^\top & {\bf 0}_3^\top & p h \gamma
	\end{pmatrix},
	\end{equation}
	with 	
	$
	{\bf M}_5 =  2 {\bf B} {\bf v}^\top  - ( {\bf v} \cdot {\bf B} ) {\bf I}_3 - {\bf v} \left( (1-|{\bf v}|^2) {\bf B}^\top + ({\bf v} \cdot {\bf B}) {\bf v}^\top \right).
	$	
	By the chain rule $\frac{\partial {\bf F}_1}{\partial {\bf W}} = \frac{\partial {\bf F}_1}{\partial {\bf V}} \frac{\partial {\bf V}}{\partial {\bf W}}$, we get the expression of $\frac{\partial {\bf F}_1}{\partial {\bf W}}$ and find it is not symmetric in general. For example, the $(2,6)$ element of the Jacobian matrix $\frac{\partial {\bf F}_1}{\partial {\bf W}}$ is 
	$
	\frac{p \gamma B_2}{\rho} ( 1 + v_1^2 -v_2^2 - v_3^2 ),
	$
	while the $(6,2)$ element is 
	$
	\frac{p \gamma }{\rho} \left(  B_2 (1+ v_1^2 - v_3^2) - B_1 v_1 v_2 + B_3 v_2 v_3  \right).
	$

	Next, let us prove \eqref{NotPair}. Note that 
	\begin{equation*}%\label{PscrFPV}
	{\mathcal Q}'_i ({\bf V}) = \left(
	\frac{ (\Gamma - S) \gamma v_1 } { \Gamma - 1 },~ 
	- \frac{\rho \gamma^3 S}{\Gamma - 1} \left( (1-|{\bf v}|^2) {\bf e}_1^\top + v_1 {\bf v}^\top  \right),~ 0,~ 0,~ 0,~ -\frac{\rho \gamma v_1}{p(\Gamma - 1)}
	\right).
	\end{equation*}
	Using \eqref{DefW}, \eqref{PUPV} and \eqref{PF1PV}, one can derive that 
	\begin{equation*}%\label{PscrEPU_PF1PV}
	{\mathcal E}' ({\bf U})  \frac{ \partial {\bf F}_1 }{\partial {\bf V}}=  \left(
	\frac{ (\Gamma - S) \gamma v_1 } { \Gamma - 1 },~ 
	- \frac{\rho \gamma^3 S}{\Gamma - 1} \left( (1-|{\bf v}|^2) {\bf e}_1^\top + v_1 {\bf v}^\top  \right),~ -\frac{\rho \gamma}{p} ( {\bf v} \cdot {\bf B} ) ,~ 0,~ 0,~ -\frac{\rho \gamma v_1}{p(\Gamma - 1)}
	\right).
	\end{equation*} 
	It follows that 
	$$
	{\mathcal Q}'_1 ({\bf U}) - {\mathcal E}' ({\bf U}) {\bf F}'_1 ( {\bf U} ) 
	- \frac{\rho \gamma}{p} ( {\bf v} \cdot {\bf B} ) B_1'({\bf U}) 
	= \big( {\mathcal Q}'_1 ({\bf V}) - {\mathcal E}' ({\bf U}) {\bf F}'_1 ( {\bf V} ) 
	- \frac{\rho \gamma}{p} ( {\bf v} \cdot {\bf B} ) B_1'({\bf V}) \big) \frac{\partial {\bf V}}{\partial {\bf U}} =  {\bf 0}. 
	$$ 
	The proof for $i=1$ is complete. Similarly one can prove the conclusions for $2\le i \le d$. 
\end{proof}

\subsection{Modified RMHD Equations and Entropy Symmetrization}

To address the above issue, we propose a modified RMHD system
\begin{equation}\label{ModRMHD}
\frac{{\partial {\bf U}}}{{\partial t}} 
+ \sum_{i=1}^d \frac{\partial {\bf F}_i ( {\bf U} ) }{\partial x_i}   = - {\bf S} ( {\bf U} )  \nabla \cdot {\bf B}  , 
\end{equation}
where 
\begin{equation}\label{eq:source}
{\bf S} ( {\bf U} ) := \left( 0,~  ( 1-|{\bf v}|^2 ) {\bf B}^\top  + ({\bf v} \cdot {\bf B}) {\bf v}^\top,~ {\bf v}^\top,~ {\bf v} \cdot {\bf B} \right)^\top. 
\end{equation}
The system \eqref{ModRMHD} can be considered as the relativistic extension of the  Godunov--Powell system \cite{Godunov1972,Powell1995} in the ideal non-relativistic MHDs. 
The right-hand side term of \eqref{ModRMHD}  is proportional 
to $\nabla \cdot {\bf B}$. 
This means, at the continuous level, the modified form \eqref{ModRMHD} and conservative form \eqref{eq:RMHD} are equivalent 
under the condition \eqref{eq:DivB}. 
However, the ``source term'' ${\bf S} ( {\bf U} )  \nabla \cdot {\bf B}$
modifies the character of the RMHD equations, 
making 
the system \eqref{ModRMHD} symmetrizable and admit the convex entropy pair $({\mathcal E}, {\mathcal Q}_i )$, as shown below.

\begin{lemma} \label{lem:phi}
	Let 
	$
	\phi := \frac{\rho \gamma}{p} ( {\bf v} \cdot {\bf B} ).
	$ 
	In terms of the entropy variables ${\bf W}$ in \eqref{DefW},  $\phi ( {\bf W} )$ is a homogeneous function of degree one, i.e., 
	\begin{equation}\label{Phi_Hom}
	\phi'( {\bf W} ) \cdot {\bf W} = \phi( {\bf W} ).
	\end{equation}
	In addition, the gradient of  $\phi ( {\bf W} )$ with respect to $\bf W$ equals ${\bf S} ^\top $, i.e., 
	\begin{equation}\label{Gradphi}
	{\bf S} ^\top = \phi'( {\bf W} ).
	\end{equation}
\end{lemma}

\begin{proof}
	Taking the gradient of $\phi$ with respect to the primitive variables ${\bf V}$ gives
	$$
	\phi'({\bf V}) = \left( \frac{ \gamma}{p} ( {\bf v} \cdot {\bf B} ),~ 
	\frac{\rho \gamma^3}{p} \left( (1-|{\bf v}|^2) {\bf B}^\top+ ( {\bf v} \cdot {\bf B} ) {\bf v}^\top
	\right),~
	\frac{\rho \gamma}{p} {\bf v}^\top,~  - \frac{\rho \gamma}{p^2} ( {\bf v} \cdot {\bf B} )  \right),
	$$
	which together with \eqref{PVPW} imply
	$$
	\phi'({\bf W}) = \phi'({\bf V}) \frac{\partial {\bf V}}{\partial {\bf W}} 
	= \left( 0,~  ( 1-|{\bf v}|^2 ) {\bf B}^\top  + ({\bf v} \cdot {\bf B}) {\bf v}^\top,~ {\bf v}^\top,~ {\bf v} \cdot {\bf B} \right) = {\bf S}^\top. 
	$$
	Thus the identity \eqref{Gradphi} holds. 
	Based on \eqref{DefW}, \eqref{eq:source} and \eqref{Gradphi}, one has 
	$
	\phi'( {\bf W} )  \cdot {\bf W}=   {\bf S }^\top \cdot {\bf W} = \frac{\rho \gamma}{p} ( {\bf v} \cdot {\bf B} ) = \phi,
	$ 
	which gives \eqref{Phi_Hom}. % The proof is complete.
\end{proof}

\begin{theorem}
	For the entropy variables $\bf W$ in \eqref{DefW}, the change of variable $ {\bf U} \to {\bf W} $ symmetrizes the modified RMHD equations \eqref{ModRMHD}, and  the functions $({\mathcal E}, {\mathcal Q}_i )$ defined in \eqref{EntropyPair} form an entropy pair of the modified RMHD equations \eqref{ModRMHD}. 
\end{theorem} 

\begin{proof}
	Define 
	\begin{align} 
	&\varphi  := {\bf W}^\top {\bf U}  - {\mathcal E} 
	= \rho \gamma + \frac{\rho \gamma}{2p} \Big( (1-|{\bf v}|^2) |{\bf B}|^2 + ({\bf v}\cdot {\bf B} )^2 \Big)
	\\
	\label{potentialflux}
	& \psi_i   := {\bf W}^\top {\bf F}_i  - {\mathcal Q}_i + \phi B_i, \qquad i=1,\dots,d 
	\end{align}	
	which satisfy
	$
	\psi_i  =  \varphi v_i,~  i=1,\dots,d.
	$
	In terms of the variables ${\bf W}$, the gradients of $\varphi$ and $\psi_i$ satisfy the following identities
	\begin{equation}\label{WKL1}
	{\bf U} = \varphi'({\bf W})^\top, \qquad 
	%\\ \label{WKL2}
	{\bf F}_i = \psi_i'({\bf W})^\top - B_i \phi'( {\bf W} ), \quad  i=1,\dots,d.
	\end{equation}
	Substituting \eqref{WKL1} %and \eqref{WKL2} 
	into  \eqref{ModRMHD}, we can rewrite the modified RMHD equations as %in quasi-linear form 
	\begin{equation}\label{eq:QuasiLinear}
	\varphi''( {\bf W} ) \frac{\partial {\bf W}  }{ \partial t } + \sum_{i=1}^d \Big( \psi_i''({\bf W}) - B_i \phi''( {\bf W} )  \Big) \frac{\partial {\bf W}}{\partial x_i} = {\bf 0}, 
	\end{equation}
	where the Hessian matrices $\varphi''( {\bf W} )$, $\psi_i''({\bf W}) $ and $\phi''( {\bf W} )$ are all symmetric. Moreover, ${\mathcal E} ({\bf U})$ is a convex function on ${\bf U} \in {\mathcal G}$ and the matrix $\varphi''( {\bf W} )$ is positive definite. 
	Hence, the change of variables ${\bf U} \to {\bf W}$ symmetrizes the modified equations \eqref{ModRMHD}. According to Lemma \ref{lem:SYM_E}, 
	the system \eqref{ModRMHD} possesses a strictly convex entropy $ {\mathcal E} ({\bf U})$. 
	
	We now show that the functions $({\mathcal E}, {\mathcal Q}_i )$ defined in \eqref{EntropyPair} form an entropy pair of the modified RMHD system \eqref{ModRMHD}.  The Jacobian matrix of the system \eqref{ModRMHD} in $x_i$-direction is given by 
	$$
	{\bf A}_i ({\bf U}) := {\bf F}'_i ( {\bf U} ) + {\bf S} ({\bf U}) B_i'( {\bf U} ), \qquad i=1,\dots,d.
	$$
	Thanks to \eqref{NotPair}, \eqref{Phi_Hom} and \eqref{Gradphi}, we have  
	\begin{align*}
	&{\mathcal Q}'_i ({\bf U}) - {\mathcal E}' ({\bf U}) {\bf A}_i ({\bf U}) 
	= {\mathcal Q}'_i ({\bf U}) - {\mathcal E}' ({\bf U}) {\bf F}'_i ( {\bf U} ) - \left( {\bf W}^\top \cdot  \phi' ({\bf W}) \right) B_i'( {\bf U} )
	\\
	&= {\mathcal Q}'_i ({\bf U}) - {\mathcal E}' ({\bf U}) {\bf F}'_i ( {\bf U} ) - \phi({\bf W}) B_i'( {\bf U} )
	= {\mathcal Q}'_i ({\bf U}) - {\mathcal E}' ({\bf U}) {\bf F}'_i ( {\bf U} ) -  \frac{\rho \gamma}{p} ( {\bf v} \cdot {\bf B} )  B_i'( {\bf U} ) = 0.
	\end{align*}
	Thus, 
	%\begin{equation}\label{EnPair}
	$
	{\mathcal Q}'_i ({\bf U}) = {\mathcal E}' ({\bf U}) {\bf A}_i ({\bf U}),
	$
	%\end{equation} 
	and the functions $({\mathcal E}, {\mathcal Q}_i )$ form an entropy pair of %system
	\eqref{ModRMHD}. 
	%
	%The proof is complete. 
\end{proof}

It is worth noting that Ruggeri and Strumia \cite{ruggeri1981convex} 
(perhaps for the first time) 
also 
found the entropy variables and 
derived the entropy symmetrization for RMHD \cite{ruggeri1981convex}. 
%They got the symmetric quasi-linear form \eqref{eq:QuasiLinear} from a  ``nine-equation'' conservative form (see equation (1.17) 
%of \cite{ruggeri1981convex}). 
This symmetrization was clearly described 
by Trakhinin in \cite[Section 3]{trakhinin2001stability} and used for investigating the stability of shock waves in RMHD. 
We remark that the quasi-linear form \eqref{eq:QuasiLinear} was also obtained 
		by Anile and Pennisi \cite{anile1987mathematical} for studying 
the mathematical structure of test RMHDs. 
The main goal of this paper is %not symmetrization, but is 
to construct entropy-stable shock-capturing schemes, which are 
based on the modified RMHD system \eqref{ModRMHD} instead of  \eqref{eq:QuasiLinear}.   
		Another important different form of symmetrization for the RMHD system, in term of variables $(p,\gamma {\bf v}^\top, {\bf B}^\top, S )^\top$ instead of entropy variables, was derived in \cite{freistuhler2013symmetrizations} for investigating the stability of relativistic current-vortex sheets.

\begin{remark}
	We note that, in the modified RMHD system \eqref{ModRMHD}, the induction equation is given by 
	$
	\frac{\partial {\bf B}}{\partial t} + \nabla \cdot ( {\bf v} {\bf B}^\top - {\bf B} {\bf v}^\top ) + {\bf v} \nabla \cdot {\bf B} = {\bf 0}. 
	$ 
	Taking the divergence of this equation gives
	$
	\nabla \cdot \left( \frac{\partial {\bf B}}{\partial t} + \nabla \cdot ( {\bf v} {\bf B}^\top - {\bf B} {\bf v}^\top ) + {\bf v} \nabla \cdot {\bf B} \right) = \frac{\partial }{\partial t} ( \nabla \cdot {\bf B} ) 
	+ \nabla \cdot ( {\bf v}  \nabla \cdot {\bf B} ) = {\bf 0}. 
	$
	Combining the continuity equation of \eqref{ModRMHD}, it yields 
	$\frac{ \partial  }{\partial t} \left( \frac{\nabla \cdot {\bf B}}{\rho \gamma} \right) + {\bf v} \cdot \nabla \left( \frac{\nabla \cdot {\bf B}}{\rho \gamma} \right) = 0.$ 
	This implies that the quantity $\frac{\nabla \cdot {\bf B}}{\rho \gamma}$ is constant along streamlines. 
	A similar property also holds for the symmetrizable non-relativistic ideal MHD system proposed by Godunov \cite{Godunov1972} and Powell \cite{Powell1995}. 
	As first demonstrated numerically by Powell \cite{Powell1995} in his eight-wave method, 
	such a property implies that the error in divergence may be advected away by the flow, 
	and   
	 suitable discretization
	of the symmetrizable MHD system may give robust numerical methods 
	in the sense of controlled divergence error \cite{Powell1995}, entropy stability \cite{Chandrashekar2016,winters2016affordable,LiuShuZhang2017} and positivity preservation \cite{WuShu2018,WuShu2019}, etc. 
\end{remark}

	Similar to the Powell source term (cf.~\cite{Powell1995,Chandrashekar2016,winters2016affordable,LiuShuZhang2017}) for the ideal MHD system, 
the source term added in the modified RMHD system \eqref{ModRMHD} is also non-conservative but necessary to obtain a symmetrizable formulation and to accommodate an entropy condition on the PDE level. 
Hence, in order to achieve entropy stability, our schemes presented later are designed based on the modified RMHD system \eqref{ModRMHD}, which leads to additional technical challenges in discretizing the source term suitably to ensure its compatibility with the entropy stability. 
As mentioned in \cite{Chandrashekar2016,LiuShuZhang2017} for the non-relativistic MHD system, there is still a conflict between the entropy stability which requires the non-conservative source term, and the conservation property which is lost due to the source term. 
The loss of conservation property leaves the possibility that it 
may lead to incorrect resolutions for some discontinuous problems, as observed in the ideal MHD case \cite{Toth2000}.  
It will be interesting to explore whether entropy stable schemes can be constructed via the conservative formulation \eqref{eq:RMHD}.

\section{Derivation of Two-point Entropy Conservative Fluxes} \label{sec:flux}

In this section, we derive explicit two-point entropy conservative numerical fluxes, which will play an important role in constructing entropy conservative schemes and entropy stable schemes, for the RMHD equations \eqref{ModRMHD}. 
Similar to the non-relativistic MHD case \cite{Chandrashekar2016}, the standard definition \cite{tadmor2003entropy} of the entropy conservative flux is not 
applicable here and should be slightly modified, due to the presence of the term ${\bf S} ( {\bf U} )  \nabla \cdot {\bf B} $ in the symmetrizable RMHD equations \eqref{ModRMHD}. Here we adopt a definition similar to the one proposed in \cite{Chandrashekar2016} for the non-relativistic MHD equations. 

\begin{definition}
	For $i=1,2,3,$ 
	a consistent two-point numerical flux ${\bf F}_i^{\star} ( {\bf U}_L, {\bf U}_R ) $ is entropy conservative if 
	\begin{equation}\label{DefECflux}
	( {\bf W}_R - {\bf W}_L ) \cdot {\bf F}_i^{\star} ( {\bf U}_L, {\bf U}_R ) + ( \phi_R - \phi_L ) \frac{  B_{i,R} + B_{i,L}  }{2} = \psi_{i,R} - \psi_{i,L},
	\end{equation}
	where ${\bf W}$, $\phi$, $\psi_{i}$ and $B_{i}$ are the entropy variables defined in \eqref{DefW}, the function defined in Lemma \ref{lem:phi}, the potential fluxes defined in \eqref{potentialflux}, and the $x_i$ magnetic field component $B_i$, respectively. The subscripts $L$ and $R$ indicate that those quantities are corresponding to the ``left'' state ${\bf U}_L$ and the ``right'' state ${\bf U}_R$, respectively.
\end{definition}

Now, we would like to construct explicit entropy conservative fluxes ${\bf F}_i^{*} ( {\bf U}_L, {\bf U}_R ) $ satisfying the condition \eqref{DefECflux}. 
For notational convenience, we employ 
$$
\jump {a} = a_R - a_L, \qquad \dgal{ a } = ( a_R + a_L )/2
$$
to denote, respectively, the jump and the arithmetic mean of a quantity. In addition, we also need 
the logarithmic mean 
\begin{equation}\label{logmean}
\dgal{ a }^{\ln} = ({ a_R - a_L })/({ \ln a_R - \ln a_L }),
\end{equation}
which was 
first introduced in \cite{ismail2009affordable}. %, including a stable numerical implementation.
Then, one has following identities 
\begin{align}\label{WKLIQ1}
& \jump{\ln a } = \jump{a} / \dgal{a}^{\ln}, \qquad \jump{ \sqrt{a} } = \jump{a} /( 2 \dgal{\sqrt{a}} ),
\\ \label{WKLIQ2}
& \jump {ab} = \dgal{a} \jump{b} + \dgal{b} \jump{a}, \qquad \jump{a^2} = 2 \dgal{a}\jump{a},
\end{align} 
which will be frequently used in the following derivation.

Let us introduce the following set of variables 
\begin{equation}\label{z}
{\bf z} := ( \rho, {\bf u}^\top, {\bf H}^\top, \beta  )^\top,
\end{equation}
with $
{\bf u} := \gamma{\bf v},$  ${\bf H} = \gamma^{-1} {\bf B}$ and $\beta = \rho /p$.  
Define  
$$
\widehat {\bf u} := ( \dgal{u_1}, \dgal{u_2}, \dgal{u_3}  )^\top, \quad 
\widehat {\bf H} := ( \dgal{H_1}, \dgal{H_2}, \dgal{H_3}  )^\top, \quad 
\bm \mu := ( \dgal{ \beta u_1}, \dgal{ \beta u_2}, \dgal{\beta u_3} )^\top.$$  
An explicit two-point entropy conservative flux for $i=1$ is given below. 

\begin{theorem} \label{thm:ECflux}
	Let $\widehat e  := 1+  (\dgal{\beta}^{\ln})^{-1}/(\Gamma - 1)$, $
	\vartheta  
	:=  \dgal { {\bf u} \cdot {\bf H} } 
	$, 
	$\widehat \Theta := \dgal{\beta}  \left(  \dgal{\beta} + 
	\widehat {\bf u} \cdot {\bm \mu} \right) > 0 $, 
	$\Theta := \widehat \Theta \left( |\widehat { \bf u }|^2 - \dgal{\gamma}^2 \right) < 0$, $\widehat{p}_{tot}  :=   {  \dgal {\rho} } / { \dgal{\beta} }  + \dgal{p_m}$, 
	and 
	\begin{align*}
	\sigma & := 2 \dgal{\beta u_1} ( \widehat {\bf u} \cdot \widehat {\bf H} ) 
	\big(  {\bm \mu} \cdot \widehat {\bf H} - \vartheta \dgal {\beta} \big)
	- ( \dgal {\beta} + \widehat {\bf u} \cdot \bm \mu  ) \big( 
	\dgal { u_1 }  ( \dgal {\rho}  + \dgal {\rho}^{\ln} \dgal {\beta}  \widehat e )   +  | \widehat {\bf H} |^2 
	\dgal{\beta u_1}
	\big), 
	\\ 
	\begin{split}
	{\bf \Xi} 
	& := \sigma \widehat{\bf u} +  \left( \dgal{\gamma}^2 - |\widehat {\bf u}|^2 \right) \dgal{\beta u_1} 
	\left( 
	\big((   \widehat {\bf H} + \vartheta \widehat {\bf u}  ) \cdot {\bm \mu} \big) \widehat {\bf H}
	+ \vartheta \big( \vartheta \dgal {\beta} - {\bm \mu} \cdot \widehat {\bf H} \big) \widehat {\bf u}
	\right)   
	\\
	& \quad +  \dgal { B_1 }  \dgal {\beta}^2 \dgal {\gamma}
	\left( 
	2  ( \widehat {\bf u} \cdot \widehat {\bf H} )  \widehat {\bf u} 
	+  \big( \dgal{\gamma}^2 - |\widehat {\bf u}|^2 \big) 
	\big( \widehat{\bf H} - \vartheta \widehat {\bf u}  \big) 
	\right) 
	\\
	& \quad
	+  \dgal {\beta} \dgal { B_1 } 
	( \widehat {\bf u} \cdot \bm \mu +  \dgal {\beta}   ) \big( 
	\vartheta \dgal {\gamma} - \dgal {\gamma {\bf u} \cdot {\bf H} } 
	\big) \widehat {\bf u},
	\end{split}
	\\
	{\bf \Pi} 
	& := \dgal{\beta u_1} \left( \dgal{\beta} + 
	\widehat {\bf u} \cdot {\bm \mu} \right) \widehat {\bf H} 
	+ \left( \dgal{\beta u_1}  ( \vartheta \dgal{\beta} - {\bm \mu} \cdot \widehat{\bf H} ) - 
	\dgal{\gamma} \dgal{ B_1 } \dgal{\beta}^2  \right) \widehat {\bf u},
	\\
	\xi & := ( \dgal{\beta} + 
	\widehat {\bf u} \cdot {\bm \mu} ) \left( \dgal{\gamma} \dgal {\gamma {\bf u} \cdot {\bf H} } -  \vartheta |\widehat {\bf u} |^2 \right) - 2 \dgal{\beta}  \dgal{\gamma}^2 ( \widehat {\bf u} \cdot \widehat {\bf H} ).
	\end{align*}
	Then the numerical flux given by 
	\begin{equation}\label{ECflux1}
	{\bf F}_1^{\star} ( {\bf U}_L, {\bf U}_R ) = \left( \dgal {\rho}^{\ln} \dgal{u_1},~ 
	\Theta^{-1} {\bf \Xi}^\top + \widehat {p}_{tot} {\bf e}_1^\top,~ 
	\widehat \Theta^{-1} {\bf \Pi}^\top,~  \Theta^{-1} ( \sigma \dgal{\gamma} - \dgal{\beta} \dgal{B_1} \xi )
	\right)^\top 
	\end{equation}
	satisfies \eqref{DefECflux} for $i=1$.
\end{theorem}

\begin{proof}
	Let ${\bf F}_1^{\star} ( {\bf U}_L, {\bf U}_R ) =: ( f_1, f_2, \dots, f_8 )^\top $ and ${\bf W}=(W_1,\cdots, W_8)^\top$. Then the condition \eqref{DefECflux} for $i=1$ can be rewritten as 
	\begin{equation}\label{DefECflux2}
	\sum_{j=1}^8  \jump{ W_j }  f_j   =  \jump{ \psi_1 } - \dgal{B_1} \jump{\phi}.
	\end{equation}
	To determine the unknown components of ${\bf F}_1^{\star}$, we would like to expand each jump term in \eqref{DefECflux2} into linear combination of the jumps of certain parameter variables. This will give us a linear algebraic system of eight equations for the unknown components $( f_1, f_2, \dots, f_8 )$. There are many options to 
	choose different sets of parameter variables, which may result in different fluxes. Here we take   
	${\bf z} = ( \rho, {\bf u}^\top, {\bf H}^\top, \beta  )^\top$ as the parameter variables, to make the resulting formulation of 
	${\bf F}_1^{\star}$ simple. 
	
	In terms of the parameter variables ${\bf z} = ( \rho, {\bf u}^\top, {\bf H}^\top, \beta  )^\top$, the entropy variables $\bf W$ in \eqref{DefW} can be explicitly expressed as 
	\begin{align*}
	& W_1 = \frac{\Gamma - S}{\Gamma - 1} + \frac{\rho }{p} = \frac{\Gamma }{\Gamma - 1} + \beta + \frac1{\Gamma - 1} \ln  \beta  + \ln \rho, ~~ W_8 = - \beta \gamma = - \beta \sqrt{ 1 + |{\bf u}|^2 },
	\\
	& W_{i+1}  = \beta \gamma v_i = \beta u_i, ~~ W_{i+4}  = \beta  \gamma \left( ( 1-|{\bf v}|^2 ) B_i + ( {\bf v} \cdot {\bf B} ) v_i \right) = \beta \big( H_i + ( {\bf u} \cdot {\bf H} ) u_i \big), \quad  1\le i \le 3,
	\end{align*}
	and $\psi_1$ and $\phi$ can be expressed as 
	\begin{align*}
	& \psi_1 = \rho \gamma v_1 + \frac{ \beta \gamma }{2 } \Big( (1-|{\bf v}|^2) |{\bf B}|^2 + ({\bf v}\cdot {\bf B} )^2 \Big) v_1 
	= \rho u_1 + \beta u_1 \frac{  |{\bf H}|^2 + ({\bf u} \cdot {\bf H})^2 } 2,
	\\
	&\phi = \beta \gamma ( {\bf v} \cdot {\bf B} ) = \beta ({  {\bf u} \cdot {\bf H} }){ \sqrt{ 1+ |{\bf u}|^2 } }. 
	\end{align*}
	Then, using the identities \eqref{WKLIQ1}--\eqref{WKLIQ2}, we rewrite the jump terms involved in \eqref{DefECflux2} as 
	\begin{align*}
	&  \jump{ W_1 } = \jump {\beta} + \frac 1  { (\Gamma - 1 ) \dgal{\beta}^{\ln} } \jump {\beta} + \frac{ \jump{ \rho } }{ \dgal{\rho}^{\ln} } = \widehat e  \jump {\beta} + \frac{ \jump{ \rho } }{ \dgal{\rho}^{\ln} },
	\\
	&   \jump { W_{i+1} } = \dgal{\beta} \jump{ u_i } + \dgal{u_i} \jump{\beta}, \qquad 1\le i \le 3,
	\\
	&
	\jump { W_{i+4} } = \dgal{\beta} \jump{ H_i } + \dgal{H_i} \jump{\beta} + \dgal{\beta u_i} \jump { {\bf u} \cdot {\bf H} } 
	+  \vartheta   (  \dgal{\beta} \jump{ u_i } + \dgal{u_i} \jump{\beta} ), \quad 1\le i \le 3,
	\\
	& \jump{ W_8 } = -\dgal{\gamma} \jump{\beta} - ( \dgal{\beta}/\dgal{\gamma} ) \left \llbracket{ \frac{|\bf u|^2}2 } \right \rrbracket,
	\\
	& \jump{ \psi_1 } = 
	\dgal  {\rho} \jump { u_1 } + \dgal{ u_1 } \jump { \rho } +  \dgal{ p_m }   (  \dgal{\beta} \jump{ u_1 } + \dgal{u_1} \jump{\beta} ) 
	+ 
	\dgal{\beta u_1} \left(  \left \llbracket{ \frac{|\bf H|^2}2 } \right \rrbracket 
	+ \vartheta  \jump{ {\bf u} \cdot {\bf H} }  \right),
	\\
	& \jump { \phi } = \dgal{\gamma {\bf u} \cdot {\bf H}} \jump {\beta} + \dgal{\beta} 
	\left( \dgal{\gamma} \jump { {\bf u} \cdot {\bf H} } + \vartheta \dgal{\gamma}^{-1} \left \llbracket{ \frac{|\bf u|^2}2 } \right \rrbracket 
	\right),
	\end{align*}
	with 
	\begin{align}
	\jump { {\bf u} \cdot {\bf H} }  = \sum_{j=1}^3 \left( \dgal{ H_j} \jump { u_j } + \dgal{ u_j} \jump { H_j } \right),
	\\
	\left \llbracket{ \frac{|\bf u|^2}2 } \right \rrbracket  = \sum_{j=1}^3 \dgal{ u_j} \jump { u_j }, \qquad  \left \llbracket{ \frac{|\bf H|^2}2 } \right \rrbracket  = \sum_{j=1}^3 \dgal{ H_j} \jump { H_j }.
	\end{align}
	Substituting the above expressions of jumps into \eqref{DefECflux2} gives 
	\begin{equation}\label{eqJUMP1}
	\jump{ \bf z }^\top ( {\bf M} {\bf F}_1^{\star} ) = \jump{ \bf z }^\top {\bm \varsigma},
	\end{equation}
	where 
	\begin{align*}
	& {\bf M} = \begin{pmatrix}
	( \dgal{\rho}^{ \ln } )^{-1} & {\bf 0}_3^\top & {\bf 0}_3^\top & 0
	\\
	{\bf 0}_3 & \dgal{\beta} {\bf I}_3 & \vartheta \dgal{\beta} {\bf I}_3 + \widehat {\bf H} {\bm \mu}^\top & - \dgal{\beta} \dgal{\gamma}^{-1} \widehat {\bf u} 
	\\
	{\bf 0}_3 &  {\bf O}_3 & \dgal{\beta} {\bf I}_3 + \widehat {\bf u} {\bm \mu}^\top & {\bf 0}_3
	\\
	\widehat {e} & \widehat {\bf u}^\top & \widehat {\bf H}^\top + \vartheta \widehat {\bf u}^\top & - \dgal{\gamma}
	\end{pmatrix},
	\\
	& {\bm \varsigma} = \begin{pmatrix}
	\dgal{u_1}
	\\ 
	\vartheta \dgal{ \beta u_1 } \widehat {\bf H}  - \dgal{\beta}\dgal{B_1} \left( 
	\dgal{\gamma} \widehat {\bf H} + \vartheta \dgal{\gamma}^{-1} \widehat {\bf u} \right) 
	+ ( \dgal{\rho} + \dgal{\beta} \dgal{p_m} ) {\bf e}_1
	\\
	\dgal{ \beta u_1 } \widehat {\bf H} + ( \vartheta \dgal{\beta u_1} - \dgal{\beta} \dgal{\gamma} \dgal{B_1}  ) \widehat {\bf u}
	\\
	\dgal{p_m} \dgal{u_1} - \dgal{B_1} \dgal{\gamma {\bf u} \cdot {\bf H} } 
	\end{pmatrix}.
	\end{align*}
	One can verify that the numerical flux ${\bf F}_1^{\star} ( {\bf U}_L, {\bf U}_R )$ given by 
	\eqref{ECflux1} solves the linear system ${\bf M} {\bf F}_1^{\star} = \bm \varsigma$. Thus,  
	${\bf F}_1^{\star} $ satisfies \eqref{eqJUMP1}, which is equivalent to \eqref{DefECflux} for $i=1$. 
	Hence the numerical flux  ${\bf F}_1^{\star} ( {\bf U}_L, {\bf U}_R )$ is entropy conservative in the sense of \eqref{DefECflux}. 
	
	It is worth noting that 
	\begin{equation}\label{detM}
	\det ( { \bf M } )  = \frac{ \dgal{\beta}^5 } { \dgal{\gamma} \dgal{ \rho }^{ \ln }  } 
	\left(  \dgal{\beta} + 
	\widehat {\bf u} \cdot {\bm \mu} \right) 
	\left( |\widehat { \bf u }|^2 - \dgal{\gamma}^2 \right) = \frac{ \dgal{\beta}^4 } { \dgal{\gamma} \dgal{ \rho }^{ \ln }  }    \Theta <0,
	\end{equation}
	which implies ${\bf F}_1^{\star} $ given by 
	\eqref{ECflux1} is the unique solution to the linear system ${\bf M} {\bf F}_1^{\star} = \bm \varsigma$.
	
	Finally, let us verify that the numerical flux ${\bf F}_1^{\star} ( {\bf U}_L, {\bf U}_R )$ given by 
	\eqref{ECflux1} is consistent with the flux ${\bf F}_1$. If letting $ {\bf U}_L = {\bf U}_R = {\bf U} $, then one has 
	\begin{align*}
	\widehat \Theta & = \beta^2 \gamma^2, \qquad \Theta = -\beta^2 \gamma^2,\qquad 
	\widehat{p}_{tot} = p + p_m = p_{tot}, \qquad 
	\sigma  = -\beta^2 \gamma ( \rho h \gamma^2 + |{\bf B}|^2 ) v_1,
	\\
	{\bf \Xi} &= - \beta^2 \gamma^2 v_1 ( \rho h{\gamma^2}\vec v + |\vec B|^2 \vec v  ) + \beta^2 \gamma^2 (\vec v \cdot \vec B)\vec B 
	+ \beta^2 \gamma^2 B_1 \big( \gamma^{-2} \vec B + (\vec v \cdot \vec B) \vec v \big) 
	\\
	& =  \Theta  \left( v_1 \vec m -  B_1 \big( \gamma^{-2} \vec B + (\vec v \cdot \vec B) \vec v \big) \right),
	\\
	{\bf \Pi} &= \beta^2 ( 1+|{\bf u}|^2 ) u_1 {\bf H} - \beta^2 \gamma B_1 {\bf u} = \widehat \Theta ( v_1 {\bf B} - B_1 {\bf v} ),
	\\
	\xi &= - \beta \gamma^2 ( {\bf v} \cdot {\bf B} ), \qquad  
	\dgal {\rho}^{\ln} \dgal{u_1} = \rho \gamma v_1, \qquad 
	\Theta^{-1} ( \sigma \dgal{\gamma} - \dgal{\beta} \dgal{B_1} \xi ) = m_1,
	\end{align*}
	which imply ${\bf F}_1^{\star} ( {\bf U}, {\bf U}) = {\bf F}_1 ( {\bf U} ) $. Thus, the numerical flux ${\bf F}_1^{\star} ( {\bf U}_L, {\bf U}_R )$ given by 
	\eqref{ECflux1} is consistent with the flux ${\bf F}_1$.
	
	Therefore, the numerical flux ${\bf F}_1^{\star} ( {\bf U}_L, {\bf U}_R )$ is consistent and is entropy conservative in the sense of \eqref{DefECflux}. The proof is complete. 
\end{proof}

Explicit entropy conservative fluxes ${\bf F}_i^{\star } ( {\bf U}_L, {\bf U}_R ) $ for $i=2$ and $i=3$ can be constructed similarly or obtained by simply using a symmetric transformation based on the rotational invariance of the system \eqref{ModRMHD}. For example, ${\bf F}_2^{\star} ( {\bf U}_L, {\bf U}_R ) $ is given by 
\begin{equation}\label{ECflux2}
{\bf F}_2^{\star} ( {\bf U}_L, {\bf U}_R ) = \left( \dgal {\rho}^{\ln} \dgal{u_2},~ 
\Theta^{-1} \widetilde {\bf \Xi}^\top + \widehat {p}_{tot} {\bf e}_2^\top,~ 
\widehat \Theta^{-1} \widetilde {\bf \Pi}^\top,~  \Theta^{-1} ( \widetilde \sigma \dgal{\gamma} - \dgal{\beta} \dgal{B_2}  \xi )
\right)^\top,
\end{equation}
where 
\begin{align*}
\widetilde \sigma & := 2 \dgal{\beta u_2} ( \widehat {\bf u} \cdot \widehat {\bf H} ) 
\big(  {\bm \mu} \cdot \widehat {\bf H} - \vartheta \dgal {\beta} \big)
- ( \dgal {\beta} + \widehat {\bf u} \cdot \bm \mu  ) \big( 
\dgal { u_2 }  ( \dgal {\rho}+ \dgal {\rho}^{\ln} \dgal {\beta}  \widehat e )   +  | \widehat {\bf H} |^2 
\dgal{\beta u_2}
\big), 
\\ 
\begin{split}
\widetilde {\bf \Xi} 
& := \sigma \widehat{\bf u} +  \left( \dgal{\gamma}^2 - |\widehat {\bf u}|^2 \right) \dgal{\beta u_2} 
\left( 
\big((   \widehat {\bf H} + \vartheta \widehat {\bf u}  ) \cdot {\bm \mu} \big) \widehat {\bf H}
+ \vartheta \big( \vartheta \dgal {\beta} - {\bm \mu} \cdot \widehat {\bf H} \big) \widehat {\bf u}
\right)   
\\
& \quad +  \dgal { B_2 }  \dgal {\beta}^2 \dgal {\gamma}
\left( 
2  ( \widehat {\bf u} \cdot \widehat {\bf H} )  \widehat {\bf u} 
+  \big( \dgal{\gamma}^2 - |\widehat {\bf u}|^2 \big) 
\big( \widehat{\bf H} - \vartheta \widehat {\bf u}  \big) 
\right) \widehat {\bf u}
\\
& \quad
+  \dgal {\beta} \dgal { B_2 } 
( \widehat {\bf u} \cdot \bm \mu +  \dgal {\beta}   ) \big( 
\vartheta \dgal {\gamma} - \dgal {\gamma {\bf u} \cdot {\bf H} } 
\big) \widehat {\bf u},
\end{split}
\\
\widetilde {\bf \Pi} 
& := \dgal{\beta u_2} \left( \dgal{\beta} + 
\widehat {\bf u} \cdot {\bm \mu} \right) \widehat {\bf H} 
+ \left( \dgal{\beta u_2}  ( \vartheta \dgal{\beta} - {\bm \mu} \cdot \widehat{\bf H} ) - 
\dgal{\gamma} \dgal{ B_2 } \dgal{\beta}^2  \right) \widehat {\bf u}.
\end{align*}

\begin{remark}\label{rem:Simple}
	Taking ${\bf B}_L = {\bf B}_R = {\bf 0}$, we obtain a set of explicit entropy conservative fluxes for the relativistic hydrodynamic equations with zero magnetic field: 
	\begin{equation}\label{RHDec}
	{\bf F}_i^{\star} ( {\bf U}_L, {\bf U}_R )  = \Big(   \dgal {\rho}^{\ln} \dgal{u_i},~ 
	\widehat {\rho h} \dgal{u_i} \widehat {\bf u}^\top + \frac{ \dgal{\rho} }{ \dgal{\beta} } {\bf e}_i^\top,~ {\bf 0}_3^\top,~  \widehat {\rho h} \dgal{\gamma}  \dgal{u_i} \Big)^\top,
	\end{equation}
	where $i=1,2,3$, and 
	$
	\widehat {\rho h} := \frac{  \dgal{\rho}/\dgal{\beta} + \dgal{\rho}^{\ln} \widehat {e}  } { \dgal{\gamma}^2 - |\widehat{\bf u}|^2 }. 
	$ 
	It is noticed that  
	the expressions of the entropy conservative fluxes \eqref{RHDec} are simpler than those derived in \cite{duan2019high} via a different set of parameter variables $(\rho, \beta, {\bf v})$.   
	In fact, the set of parameter variables \eqref{z} we employed 
	are carefully selected from many possible candidate sets, so as to
	render the resulting fluxes in a simple form.
\end{remark}

\section{Entropy Conservative Schemes and Entropy Stable Schemes in One Dimension}\label{sec:1Dschemes}

In this section, we construct entropy conservative schemes and entropy stable schemes for the 
1D symmetrizable RMHD equations \eqref{ModRMHD}. 
To avoid confusing subscripts, we will use $x$ to denote the 1D spatial coordinate, $\bf F$ to represent the flux vector ${\bf F}_1$, and 
$\mathcal Q$ to represent the entropy flux ${\mathcal Q}_1$ in $x_1$-direction.  

For simplicity, we consider 
a uniform mesh $x_1<x_2<\cdots < x_N$ with mesh size $x_{i+1}-x_{i}=\Delta x$. 
The midpoint values are defined as $x_{i+1/2}:=(x_i+x_{i+1})/2$ and the spatial domain is partitioned into cells $I_i=(x_{i-1/2},x_{i+1/2})$. 
A semi-discrete finite difference scheme of the 1D symmetrizable RMHD equations \eqref{ModRMHD} can be written as 
\begin{equation}\label{1Dsemischeme}
\frac{\rm d}{{\rm d} t} {\bf U}_i(t) + \frac{ \widehat{\bf F}_{i+\frac12}(t) -\widehat{\bf F}_{i-\frac12}(t)  }{\Delta x} 
+ {\bf S}( {\bf U}_i(t) ) \frac{ \widehat{ B }_{1,i+\frac12}(t) -\widehat{B}_{1,i-\frac12}(t)  }{\Delta x} = {\bf 0},
\end{equation}
where ${\bf U}_i(t) \approx {\bf U}(x_i,t) $, the numerical flux $ \widehat{\bf F}_{i+\frac12}  $ is consistent with the flux ${\bf F}({\bf U})$, and 
$$({ \widehat{\bf F}_{i+\frac12}-\widehat{\bf F}_{i-\frac12} })/{\Delta x} \approx  \left.\partial_x {\bf F} \right|_{x=x_i},\quad ({ \widehat{B}_{1,i+\frac12} -\widehat{B}_{1,i-\frac12}  })/{\Delta x} \approx \left. \partial_x B_1 \right|_{x=x_i} = \left.  \nabla \cdot {\bf B} \right|_{x=x_i} .$$ 
For notational convenience, the $t$ dependence of all quantities is suppressed below.

\subsection{Entropy Conservative Schemes}

The semi-discrete scheme \eqref{1Dsemischeme} is said to be entropy conservative if its computed solutions satisfy 
a discrete entropy equality
\begin{equation}\label{DefECscheme}
\frac{ \rm d }{ {\rm d} t} {\mathcal E} ( {\bf U}_i ) + \frac{ 1 }{\Delta x} \left( \widetilde {\mathcal Q}_{i+\frac12}  -   \widetilde {\mathcal Q}_{i-\frac12} \right) =0
\end{equation}
for some numerical entropy flux $ \widetilde {\mathcal Q}_{i+\frac12} $ consistent with the entropy flux ${\mathcal Q}$.

We introduce the following notations 
$$
\jump{ a }_{i+1/2} = a_{i+1} - a_{i},  \qquad  \dgal{a}_{i+\frac12} =  ( a_i + a_{i+1} )/2 
$$
to denote the jump and the arithmetic mean of a quantity at the interface $x_{i+1/2}$. 

\subsubsection{Second-order entropy conservative scheme}

%similar to \cite{Chandrashekar2016}

Similar to the non-relativistic case \cite{Chandrashekar2016}, a second-order accurate entropy conservative scheme is
obtained by taking $\widehat { \bf F }_{i+\frac12}  $ as the two-point entropy conservative flux and $\widehat{ B }_{1,i+\frac12} = \dgal{B_1}_{i+\frac12} $.

\begin{theorem}\label{thm:1D2nEC}
	If taking $\widehat { \bf F }_{i+\frac12}  $ as an entropy conservative numerical flux ${\bf F}_1^\star ( {\bf U}_i, {\bf U}_{i+1}  )$ satisfying \eqref{DefECflux} and $\widehat{ B }_{1,i+\frac12} = \dgal{B_1}_{i+\frac12} $, then the scheme \eqref{1Dsemischeme}, which becomes 
	\begin{equation}\label{1DEC2}
	\frac{ {\rm d } {\bf U}_i }{{\rm d} t}   = -\frac{ {\bf F}_1^\star ( {\bf U}_i, {\bf U}_{i+1}  )  - {\bf F}_1^\star ( {\bf U}_{i-1}, {\bf U}_{i}  )  }{\Delta x} 
	- {\bf S}( {\bf U}_i ) \frac{ \dgal{B_1}_{i+\frac12} - \dgal{B_1}_{i-\frac12} }{\Delta x},
	\end{equation}	
	is entropy conservative, and the corresponding numerical entropy flux is given by
	\begin{equation}\label{EC2Q}
	\widetilde {\mathcal Q}_{i+\frac12}^\star =  \dgal{ {\bf W} }_{i+\frac12} \cdot {\bf F}_1^\star ( {\bf U}_i, {\bf U}_{i+1}  ) + \dgal{ \phi }_{i+\frac12} \dgal{B_1}_{i+\frac12} - \dgal{\psi_1}_{i+\frac12},
	\end{equation}
	where ${\bf W}$, $\phi$ and $\psi_{1}$  are the entropy variables defined in \eqref{DefW}, the function defined in Lemma \ref{lem:phi}, the potential flux defined in \eqref{potentialflux}, respectively.
\end{theorem}

\begin{proof}
	Using \eqref{potentialflux}, one can easily verify that the above numerical entropy flux is consistent with the entropy flux $\mathcal Q$. 
	Note that the numerical flux ${\bf F}_1^\star ( {\bf U}_i, {\bf U}_{i+1}  )$ satisfies    
	\begin{equation}\label{ECflux11}
	\jump{ {\bf W} }_{i+\frac12} \cdot {\bf F}_1^\star ( {\bf U}_i, {\bf U}_{i+1}  )  + \jump{\phi}_{i+\frac12} \dgal{B_1}_{i+\frac12} = \jump { \psi_1 }_{i+\frac12}. 
	\end{equation}	
	Using \eqref{DefW}, \eqref{1DEC2}, \eqref{Gradphi}, \eqref{Phi_Hom} and \eqref{ECflux11} sequentially, we have 
	\begin{align*}
	&-\Delta x  \frac{ \rm d }{ {\rm d} t } {\mathcal E} ( {\bf U}_i ) = -\Delta x {\mathcal E}' ( {\bf U}_i) \frac{\rm d}{{\rm d} t} {\bf U}_i  )  = -\Delta x {\bf W}_i  \cdot \frac{\rm d}{{\rm d} t} {\bf U}_i
	\\
	& =  {\bf W}_i \cdot \big( {\bf F}_1^\star ( {\bf U}_i, {\bf U}_{i+1}  )  - {\bf F}_1^\star ( {\bf U}_{i-1}, {\bf U}_{i}  )  \big)
	+ {\bf W}_i \cdot {\bf S}( {\bf U}_i ) \big(  \dgal{B_1}_{i+\frac12} - \dgal{B_1}_{i-\frac12} \big)
	\\
	& =  {\bf W}_i \cdot \big(  {\bf F}_1^\star ( {\bf U}_i, {\bf U}_{i+1}  )  - {\bf F}_1^\star ( {\bf U}_{i-1}, {\bf U}_{i}  )  \big)
	+ \phi_i \big(  \dgal{B_1}_{i+\frac12} - \dgal{B_1}_{i-\frac12} \big)
	\\
	\begin{split}
	&= 
	\Big( \dgal{ {\bf W} }_{i+\frac12} - \frac12 \jump{ {\bf W} }_{ i+\frac12 } \Big) \cdot {\bf F}_1^\star ( {\bf U}_i, {\bf U}_{i+1}  ) 
	- 	\Big( \dgal{ {\bf W} }_{i-\frac12} + \frac12 \jump{ {\bf W } }_{ i-\frac12 } \Big) \cdot {\bf F}_1^\star ( {\bf U}_{i-1}, {\bf U}_{i}  ) 
	\\
	& \quad +
	\Big( \dgal{ \phi }_{i+\frac12} - \frac12 \jump{ \phi }_{ i+\frac12 } \Big) \dgal{B_1}_{i+\frac12} 
	- 	\Big( \dgal{ \phi }_{i-\frac12} + \frac12 \jump{ \phi }_{ i-\frac12 } \Big) \dgal{B_1}_{i-\frac12} 
	\end{split}
	\\
	\begin{split}
	& =  
	\dgal{ {\bf W} }_{i+\frac12}  \cdot {\bf F}_1^\star ( {\bf U}_i, {\bf U}_{i+1}  ) 
	- 	 \dgal{ {\bf W} }_{i-\frac12}  \cdot {\bf F}_1^\star ( {\bf U}_{i-1}, {\bf U}_{i}  ) 
	\\
	& \quad 
	+ 
	\dgal{ \phi }_{i+\frac12}  \dgal{B_1}_{i+\frac12} 
	- 	 \dgal{ \phi }_{i-\frac12}  \dgal{B_1}_{i-\frac12} 
	-\frac12 \left(  
	\jump { \psi_1 }_{i+\frac12}
	+ 	\jump { \psi_1 }_{i-\frac12}
	\right)
	\end{split}
	\\
	\begin{split}
	& =  
	\dgal{ {\bf W} }_{i+\frac12}  \cdot {\bf F}_1^\star ( {\bf U}_i, {\bf U}_{i+1}  ) 
	- 	 \dgal{ {\bf W} }_{i-\frac12}  \cdot {\bf F}_1^\star ( {\bf U}_{i-1}, {\bf U}_{i}  ) 
	\\
	& \quad 
	+ 
	\dgal{ \phi }_{i+\frac12}  \dgal{B_1}_{i+\frac12} 
	- 	 \dgal{ \phi }_{i-\frac12}  \dgal{B_1}_{i-\frac12} 
	- \left(  
	\dgal { \psi_1 }_{i+\frac12}
	+ 	\dgal { \psi_1 }_{i-\frac12}
	\right)	= \widetilde {\mathcal Q}_{i+\frac12}^\star - \widetilde {\mathcal Q}_{i-\frac12}^\star,
	\end{split}
	\end{align*}
	which implies the discrete entropy equality \eqref{DefECscheme} for the numerical entropy flux \eqref{EC2Q}. The proof is complete.
\end{proof}

We obtain an entropy conservative scheme \eqref{1DEC2} if the numerical flux ${\bf F}_1^\star ( {\bf U}_i, {\bf U}_{i+1})$ given by \eqref{ECflux1} is used. Other entropy conservative fluxes satisfying  \eqref{DefECflux} can also be used in \eqref{1DEC2} to obtain different entropy conservative schemes. Note that the 
second-order accuracy of the scheme \eqref{1DEC2} is only achieved on uniform mesh grids.

\subsubsection{Higher-order entropy conservative schemes}

The semi-discrete entropy conservative scheme \eqref{1DEC2} is only second-order accurate. 
By using the proposed entropy conservative flux \eqref{ECflux1} as building blocks, one can construct $2k$th-order accurate entropy conservative fluxes for any $k \in \mathbb N_+$; see \cite{lefloch2002fully}. These consist of linear combinations of second-order entropy conservative flux \eqref{ECflux1}. 
Specifically, a $2k$th-order accurate entropy conservative flux 
is defined as 
\begin{equation}\label{HighECflux}
\widetilde { \bf F }_{i+\frac12}^{2k,\star }:= 
\sum_{ r=1 }^k \alpha_{ k,r } \sum_{ s=0 }^{r-1}  {\bf F}_1^\star ( {\bf U}_{i-s}, {\bf U}_{i-s+r}  ),
\end{equation}
where the constants $\alpha_{k,r}$ satisfy 
\begin{equation}\label{alpha_condition}
\sum_{r=1}^k r \alpha_{k,r} = 1, \qquad \sum_{r=1}^k r^{2s-1} \alpha_{k,r} =0, \ \  s=2,\dots,k.
\end{equation}
The symmetrizable RMHD equations \eqref{ModRMHD} have a special source term, which should be treated carefully in constructing high-order accurate entropy conservative schemes. 
We find the key point is to accordingly approximate the spatial derivative $\partial_x B_1$ 
as 
%\begin{equation}\label{highB1}
$\frac1{\Delta x} \left({ \widetilde { B }_{1,i+\frac12}^{2k,\star } - 
	\widetilde { B }_{1,i-\frac12}^{2k,\star }  } \right) \approx \left. \partial_x B_1 \right|_{x = x_i},$ 
%\end{equation}
with $\widetilde { B }_{1,i+\frac12}^{2k,\star}$ defined as 
a linear combination of $\frac12 (B_{1,i-s} + B_{1,i-s+r})$ similar to \eqref{HighECflux}. Specifically, we set  
\begin{equation}\label{B11}
\widetilde { B }_{1,i+\frac12}^{2k,\star} :=  \sum_{ r=1 }^k \alpha_{ k,r } \sum_{ s=0 }^{r-1}   \bigg(\frac{B_{1,i-s} + B_{1,i-s+r}}2 \bigg).
\end{equation}

As an example, the fourth-order ($k=2$) version of $\widetilde { \bf F }_{i+\frac12}^{2k,\star }$ and $\widetilde { B }_{1,i+\frac12}^{2k,\star}$ is given by 
\begin{equation}\label{EC4flux}
\begin{cases}
\widetilde { \bf F }_{i+\frac12}^{4,\star } = \frac43 {\bf F}_1^\star ( {\bf U}_{i}, {\bf U}_{i+1}  ) -\frac16 \Big(
{\bf F}_1^\star ( {\bf U}_{i-1}, {\bf U}_{i+1}  ) + {\bf F}_1^\star ( {\bf U}_{i}, {\bf U}_{i+2}  )
\Big),
\\
\widetilde { B }_{1,i+\frac12}^{4,\star} = \frac23 \Big( B_{1,i} + B_{1,i+1} \Big) 
-\frac1{12} \Big( ( B_{1,i-1} + B_{1,i+1} )  + ( B_{1,i} + B_{1,i+2} )  \Big),
\end{cases}
\end{equation}
and the sixth-order ($k=3$) version is 
\begin{equation}\label{EC6flux}
\begin{cases}
\widetilde { \bf F }_{i+\frac12}^{6,\star } = \frac32 {\bf F}_1^\star ( {\bf U}_{i}, {\bf U}_{i+1}  ) -\frac3{10} \Big(
{\bf F}_1^\star ( {\bf U}_{i-1}, {\bf U}_{i+1}  ) + {\bf F}_1^\star ( {\bf U}_{i}, {\bf U}_{i+2}  )
\Big) 
\\
\qquad \quad + \frac{1}{30} \Big(
{\bf F}_1^\star ( {\bf U}_{i-2}, {\bf U}_{i+1}  ) + {\bf F}_1^\star ( {\bf U}_{i-1}, {\bf U}_{i+2} + {\bf F}_1^\star ( {\bf U}_{i}, {\bf U}_{i+3}  )
\Big),
\\
\widetilde { B }_{1,i+\frac12}^{6,\star}  = \frac34 \Big( B_{1,i} + B_{1,i+1} \Big) 
-\frac3{20} \Big( ( B_{1,i-1} + B_{1,i+1} )  + ( B_{1,i} + B_{1,i+2} )  \Big)
\\
\qquad \quad +\frac1{60} \Big( ( B_{1,i-2} + B_{1,i+1} )  + ( B_{1,i-1} + B_{1,i+2} ) 
+ ( B_{1,i} + B_{1,i+3} ) \Big).
\end{cases}
\end{equation}

If taking $\widehat { \bf F }_{i+\frac12}  = \widetilde { \bf F }_{i+\frac12}^{2k,\star }$  and $\widehat{ B }_{1,i+\frac12} = \widetilde { B }_{1,i+\frac12}^{2k,\star} $, then the scheme \eqref{1Dsemischeme}, which becomes 
\begin{equation}\label{1DEChigh}
\frac{ {\rm d } {\bf U}_i }{{\rm d} t}  = -\frac{ \widetilde { \bf F }_{i+\frac12}^{2k,\star }  - \widetilde { \bf F }_{i-\frac12}^{2k,\star }  }{\Delta x} 
- {\bf S}( {\bf U}_i ) \frac{ \widetilde { B }_{1,i+\frac12}^{2k,\star} - \widetilde { B }_{1,i-\frac12}^{2k,\star} }{\Delta x},
\end{equation}	
is entropy conservative and $2k$th-order accurate.

\begin{theorem}\label{thm:1DknEC}
	The scheme \eqref{1DEChigh} is entropy conservative, and the corresponding numerical entropy flux is given by
	\begin{equation}\label{ECkQ}
	\widetilde {\mathcal Q}_{i+\frac12}^{2k,\star} = \sum_{ r=1 }^k \alpha_{ k,r } \sum_{ s=0 }^{r-1}  \widetilde {\mathcal Q}  ( {\bf U}_{i-s}, {\bf U}_{i-s+r}  ),
	\end{equation}
	where the 
	constants $\alpha_{k,r}$ are defined in \eqref{alpha_condition}, and the function $\widetilde {\mathcal Q} $ is defined as
	\begin{equation}\label{DeftildeQ}
	\widetilde {\mathcal Q}  ( {\bf U}_L, {\bf U}_R  ) := \frac12 \left( {\bf W}_L + {\bf W}_R \right) \cdot {\bf F}_1^\star ( {\bf U}_L, {\bf U}_R  ) +   \frac{ \phi_L + \phi_R}{2}   \left( \frac{B_{1,L} + B_{1,R}}{2} \right) - \frac{ \psi_{1,L} + \psi_{1,R}}{2}.
	\end{equation}
\end{theorem}

\begin{proof}
	First, one can use \eqref{alpha_condition} to verify that the numerical entropy flux $\widetilde {\mathcal Q}_{i+\frac12}^{2k,\star}$ is consistent with the entropy flux $\mathcal Q$. 
	Using \eqref{DefW}, \eqref{1DEChigh}, \eqref{Gradphi} and \eqref{Phi_Hom}, we obtain  
	\begin{align*}
	-\Delta x & \frac{ \rm d }{ {\rm d} t } {\mathcal E} ( {\bf U}_i ) = -\Delta x {\mathcal E}' ( {\bf U}_i) \frac{\rm d}{{\rm d} t} {\bf U}_i   
	% = -\Delta x {\bf W}_i \cdot \frac{\rm d}{{\rm d} t} {\bf U}_i
	%\\
	%& 
	=  {\bf W}_i \cdot \big( \widetilde { \bf F }_{i+\frac12}^{2k,\star }   - \widetilde { \bf F }_{i-\frac12}^{2k,\star }  \big)
	+ {\bf W}_i \cdot {\bf S}( {\bf U}_i ) \big( \widetilde { B }_{1,i+\frac12}^{2k,\star} - \widetilde { B }_{1,i-\frac12}^{2k,\star} \big)
	\\
	& =  {\bf W}_i \cdot \big( \widetilde { \bf F }_{i+\frac12}^{2k,\star }   - \widetilde { \bf F }_{i-\frac12}^{2k,\star }  \big)
	+ \phi_i \big(  \widetilde { B }_{1,i+\frac12}^{2k,\star} - \widetilde { B }_{1,i-\frac12}^{2k,\star} \big).
	\end{align*}
	It is observed that 
	\begin{align*}
	& \widetilde { \bf F }_{i+\frac12}^{2k,\star }   - \widetilde { \bf F }_{i-\frac12}^{2k,\star } 
	= \sum_{ r=1 }^k \alpha_{ k,r } \Big( {\bf F}_1^\star ( {\bf U}_{i}, {\bf U}_{i+r}  )
	- {\bf F}_1^\star ( {\bf U}_{i-r}, {\bf U}_{i}  )
	\Big),
	\\
	& \widetilde { B }_{1,i+\frac12}^{2k,\star} - \widetilde { B }_{1,i-\frac12}^{2k,\star} = \sum_{ r=1 }^k \alpha_{ k,r } \Big( \frac12 \big( B_{1,i} + B_{1,i+r}  \big) - 
	\frac12 \big( B_{1,i-r} + B_{1,i}  \big) \Big).
	\end{align*}
	Therefore, 
	\begin{equation}\label{keye13454}
	-\Delta x  \frac{ \rm d }{ {\rm d} t } {\mathcal E} ( {\bf U}_i ) =
	\sum_{ r=1 }^k \alpha_{ k,r } \big( 
	\Pi_{i,i+r} - \Pi_{i,i-r}
	\big)
	\end{equation}
	with $\Pi_{i,i+r} = {\bf W}_i \cdot {\bf F}_1^\star ( {\bf U}_{i}, {\bf U}_{i+r}  )  + \frac{\phi_i}2    \big( B_{1,i} + B_{1,i+r}  \big) $ and $\Pi_{i,i-r} = {\bf W}_i \cdot {\bf F}_1^\star ( {\bf U}_{i-r}, {\bf U}_{i}  )  + \frac{\phi_i}2    \big( B_{1,i} + B_{1,i-r}  \big) $. Note that 
	\begin{align} \nonumber
	\begin{split}
	\Pi_{i,i+r}
	& = \left( \frac{ {\bf W}_i + {\bf W}_{i+r} } 2  \right) \cdot {\bf F}_1^\star ( {\bf U}_{i}, {\bf U}_{i+r}  ) + \left( \frac{ \phi_i + \phi_{i+r} } 2\right) \left( \frac{ B_{1,i} + B_{1,i+r}  }2 \right)
	\\ \nonumber
	& \quad -\frac12 \left( \big( {\bf W}_{i+r} - {\bf W}_i \big) \cdot {\bf F}_1^\star ( {\bf U}_{i}, {\bf U}_{i+r}  ) + (\phi_{i+r} - \phi_i) \frac{ B_{1,i} + B_{1,i+r}  }2 \right)
	\end{split}
	\\ 
	& = \widetilde {\mathcal Q} ( {\bf U}_{i}, {\bf U}_{i+r}  ) + \frac{ \psi_{1,i} + \psi_{1,i+r} }2 - \frac{ \psi_{1,i+r} - \psi_{1,i}  }2 = \widetilde {\mathcal Q} ( {\bf U}_{i}, {\bf U}_{i+r}  ) + \psi_{1,i}, \label{PIipr}
	\end{align}
	where the property \eqref{DefECflux} of ${\bf F}_1^\star$ has been used in the penultimate equality sign. Similarly, %one can show that 
	\begin{equation}\label{Piimr}
	\Pi_{i,i-r} = \widetilde {\mathcal Q} ( {\bf U}_{i-r}, {\bf U}_{i}  ) + \psi_{1,i}.
	\end{equation}
	Plugging \eqref{PIipr} and \eqref{Piimr} into \eqref{keye13454} gives 
	\begin{equation*} 
	-\Delta x  \frac{ \rm d }{ {\rm d} t } {\mathcal E} ( {\bf U}_i ) =
	\sum_{ r=1 }^k \alpha_{ k,r } \big( 
	\widetilde {\mathcal Q} ( {\bf U}_{i}, {\bf U}_{i+r}  ) - \widetilde {\mathcal Q} ( {\bf U}_{i-r}, {\bf U}_{i}  )
	\big) = \widetilde {\mathcal Q}_{i+\frac12}^{2k,\star} - \widetilde {\mathcal Q}_{i-\frac12}^{2k,\star},
	\end{equation*}
	which implies the discrete entropy equality \eqref{DefECscheme} for the numerical entropy flux \eqref{ECkQ}. The proof is complete. 
\end{proof}

\subsection{Entropy Stable Schemes}

Entropy is conserved only if the solutions of the RMHD equations \eqref{ModRMHD} are smooth. Entropy conservative schemes preserve the entropy and can work well in smooth regions. However, for solutions containing discontinuity where entropy is dissipated, 
entropy conservative schemes may produce oscillations; see, e.g., the numerical examples in \cite{TadmorZhong2006,FMT09,winters2016affordable} and Example \ref{example1DRiemanns} of the present paper. % 
Consequently, some numerical dissipative mechanism should be added to ensure entropy stability. 

The 1D semi-discrete scheme \eqref{1Dsemischeme} is said to be entropy stable if its computed solutions satisfy 
a discrete entropy inequality
\begin{equation}\label{DefESscheme}
\frac{ \rm d }{ {\rm d} t} {\mathcal E} ( {\bf U}_i ) + \frac{ 1 }{\Delta x} \left( \widehat {\mathcal Q}_{i+\frac12}  -   \widehat {\mathcal Q}_{i-\frac12} \right) \le 0
\end{equation}
for some numerical entropy flux $ \widehat {\mathcal Q}_{i+\frac12} $ consistent with the entropy flux ${\mathcal Q}$.  

\subsubsection{First-order entropy stable scheme}

Let us add a numerical dissipation term to the entropy conservative flux ${\bf F}_1^\star$ and define 
\begin{equation}\label{ES1}
\hat{\bf F}_{i+\frac12} = {\bf F}_1^\star ( {\bf U}_i, {\bf U}_{i+1} ) - \frac12 {\bf D}_{i+\frac12} \jump{ {\bf W} }_{i+\frac12},
\end{equation}
where ${\bf F}_1^\star ( {\bf U}_i, {\bf U}_{i+1}  )$ is an entropy conservative numerical flux, for example, the one given in \eqref{ECflux1}, and ${\bf D}_{i+\frac12}$ is any symmetric positive definite matrix. 

\begin{theorem}
	The scheme \eqref{1Dsemischeme} with $\widehat{ B }_{1,i+\frac12} = \dgal{B_1}_{i+\frac12} $ and numerical flux \eqref{ES1} is entropy stable, and the corresponding numerical entropy flux is given by 
	\begin{equation}\label{ES1Q}
	\widehat {\mathcal Q}_{i+\frac12} =  \widetilde {\mathcal Q}_{i+\frac12}^\star - \frac12 \dgal{ {\bf W} }_{i+\frac12}^\top {\bf D}_{i+\frac12} 
	\jump{ {\bf W} }_{i+\frac12},
	\end{equation}
	where $\widetilde {\mathcal Q}_{i+\frac12}^\star$ is defined by \eqref{EC2Q}. 
\end{theorem}

\begin{proof}
	First, one can easily verify that the numerical entropy flux \eqref{ES1Q} is consistent with the entropy flux $\mathcal Q$.  
	Substituting  $\widehat{ B }_{1,i+\frac12} = \dgal{B_1}_{i+\frac12} $ and numerical flux \eqref{ES1} into the scheme \eqref{1Dsemischeme} and then following the proof of Theorem \ref{thm:1D2nEC}, we obtain 
	\begin{equation*}
	\begin{aligned}
	-\Delta x  \frac{ \rm d }{ {\rm d} t } {\mathcal E} ( {\bf U}_i ) & = \widetilde {\mathcal Q}_{i+\frac12}^\star - \widetilde {\mathcal Q}_{i-\frac12}^\star - \frac12 {\bf W}_i^\top \left( 
	{\bf D}_{i+\frac12} \jump{ {\bf W} }_{i+\frac12} - {\bf D}_{i-\frac12} \jump{ {\bf W} }_{i-\frac12}
	\right)
	\\
	& = \widehat{\mathcal Q}_{i+\frac12} - \widehat {\mathcal Q}_{i-\frac12} 
	+ \frac14 \left( \jump{{\bf W}}_{i+\frac12}^\top  {\bf D}_{i+\frac12} \jump{ {\bf W} }_{i+\frac12}  + \jump{{\bf W}}_{i-\frac12}^\top  {\bf D}_{i-\frac12} \jump{ {\bf W} }_{i-\frac12} \right).
	\end{aligned}
	\end{equation*}
	Therefore,  
	\begin{multline*}
	\frac{ \rm d }{ {\rm d} t } {\mathcal E} ( {\bf U}_i ) + \frac1{\Delta x} 
	\left(  \widehat{\mathcal Q}_{i+\frac12} - \widehat {\mathcal Q}_{i-\frac12}  \right)
	= -\frac1{4\Delta x} \left( \jump{{\bf W}}_{i+\frac12}^\top  {\bf D}_{i+\frac12} \jump{ {\bf W} }_{i+\frac12}  + \jump{{\bf W}}_{i-\frac12}^\top  {\bf D}_{i-\frac12} \jump{ {\bf W} }_{i-\frac12} \right) \le 0,
	\end{multline*}
	which implies the discrete entropy inequality \eqref{DefESscheme} for the numerical entropy flux 
	\eqref{ES1Q}. The proof is complete. 
\end{proof}

In the computations, we evaluate the dissipation matrix 
${\bf D}_{i+\frac12}$ as follows.  
Let ${\bf U}_{i+\frac12}^{\tt ave}$ be an average state at $x=x_{i+\frac12}$ with the corresponding primitive variables %defined by 
\begin{equation}\label{AVEstate}
\rho_{i+\frac12}^{\tt ave} := \dgal{\rho }_{i+\frac12}^{\rm ln}, ~
{\bf v}_{i+\frac12}^{\tt ave} := \dgal{ {\bf v} }_{i+\frac12}, ~
{\bf B}_{i+\frac12}^{\tt ave} := \dgal{ {\bf B} }_{i+\frac12}, ~
{p}_{i+\frac12}^{\tt ave} :=  \rho_{i+\frac12}^{\tt ave} / \dgal{\beta }_{i+\frac12}^{\rm ln},
\end{equation}
where $\beta = \rho / p$, and $\dgal{\cdot }_{i+\frac12}^{\rm ln}$ denotes the logarithmic mean defined in \eqref{logmean}. 
Then the dissipation matrix ${\bf D}_{i+\frac12}$ is evaluated at the above average state by  
$
{\bf D}_{i+\frac12} = {\mathbf R}_{i+\frac12} |{\bf \Lambda}_{i+\frac12}| {\mathbf R}_{i+\frac12}^\top,
$
where $|{\bf \Lambda}|$ is a diagonal matrix to be specified later, 
${\mathbf R}$ is the matrix formed by the scaled (right) eigenvectors of the Jacobian matrix ${\bf A}_1 ({\bf U}) := {\bf F}'_1 ( {\bf U} ) + {\bf S} ({\bf U}) B_1'( {\bf U} )$, and it satisfies 
\begin{equation}\label{scalingR}
{\bf A}_1 = {\bf R} {\bf \Lambda} {\bf R}^{-1}, \qquad \frac{\partial {\bf U}}{\partial {\bf W}} =  {\mathbf R} {\bf R}^\top.
\end{equation}
Let $\{\lambda_\ell\}_{1\le \ell \le 8}$ be the eight eigenvalues of the Jacobian matrix ${\bf A}_1$. Then, the diagonal matrix $|{\bf \Lambda}|$ can be chosen as 
\begin{equation}\label{eq:Roe}
|{\bf \Lambda}| = {\rm diag} \{ |\lambda_1|, \cdots,|\lambda_8| \},
\end{equation}
which gives the Roe-type dissipation term in \eqref{ES1}, or taken as  
\begin{equation}\label{eq:Rusanov}
|{\bf \Lambda}| = \Big( \max_{1\le \ell \le 8} \{ |\lambda_\ell|\} \Big) {\bf I}_8,
\end{equation}
which gives the Rusanov (also called generalized Lax-Friedrichs) type dissipation term in \eqref{ES1}. 
We remark that the eigenvector scaling theorem \cite{Barth1998} ensures that there exist scaled eigenvalues of ${\bf A}_1$ satisfying \eqref{scalingR}. For the computations of the eigenvalues and eigenvectors of the Jacobian matrix of the RMHD equations, see \cite{anton2010relativistic}.

\subsubsection{High-order entropy stable schemes}

The entropy stable scheme \eqref{1Dsemischeme} with $\widehat{ B }_{1,i+\frac12} = \dgal{B_1}_{i+\frac12} $ and numerical flux \eqref{ES1} is only first order accurate in space, due to the presence of ${\mathcal O}(\Delta x)$ jump $\jump{ {\bf W} }_{i+\frac12}$ in the dissipation term.  
Towards achieving higher-order entropy stable schemes, we should use high-order dissipation operators with more accurate estimate of jump at cell interface. 
In this paper, we consider two approaches to construct high-order dissipation operators: the ENO based approach \cite{fjordholm2012arbitrarily} 
and WENO based approach \cite{biswas2018low}. 

In the ENO based approach, we define high-order entropy stable fluxes as 
\begin{equation}\label{ECfluxENO}
\widehat{\bf F}_{i+\frac12} = \widetilde {\bf F}_{i+\frac12}^{2k,*} - \frac12  {\mathbf R}_{i+\frac12} |{\bf \Lambda}_{i+\frac12}| 
\llangle {\bm \omega  } \rrangle_{i+\frac12}^{{\tt ENO}},
\end{equation}
where $\widetilde { \bf F }_{i+1/2}^{2k,\star }$ is the $2k$th-order entropy conservative flux defined in \eqref{HighECflux}, 
$ \llangle {\bm \omega  } \rrangle_{i+1/2}^{{\tt ENO}} := {\bm \omega  }_{i+1/2}^+ - 
{\bm \omega  }_{i+1/2}^-$ with ${\bm \omega  }_{i+1/2}^-$ and ${\bm \omega  }_{i+1/2}^+$ denoting, respectively, the left and right limiting values of the scaled entropy variables $ \bm \omega := {\bf R}^\top_{i+1/2} {\bf W} $ at interface $x_{i+1/2}$, obtained by $2k$th-order ENO reconstruction. 
The sign preserving property \cite{fjordholm2013eno} of ENO reconstruction implies that 
\begin{equation}\label{ENOsign}
{\rm sign} \big( \llangle {\bm \omega  } \rrangle_{i+\frac12}^{{\tt ENO}} \big) = {\rm sign} \big( \jump{\bm \omega}_{i+\frac12} \big).
\end{equation}
We refer the readers to \cite[Eq.~(3.12)]{fjordholm2012arbitrarily} for a more precise interpretation of the equality \eqref{ENOsign}. 
In the WENO based approach \cite{biswas2018low}, high-order accurate entropy stable fluxes can be defined as
\begin{equation}\label{ECfluxWENO}
\widehat{\bf F}_{i+\frac12} = \widetilde {\bf F}_{i+\frac12}^{2k,*} - \frac12  {\mathbf R}_{i+\frac12} |{\bf \Lambda}_{i+\frac12}| 
\llangle {\bm \omega  } \rrangle_{i+\frac12}^{{\tt WENO}},
\end{equation}
where the $\ell$th component of the vector $\llangle {\bm \omega  } \rrangle_{i+1/2}^{{\tt WENO}}$ is computed by 
\begin{equation}\label{keyddffd}
\llangle { \omega_\ell  } \rrangle_{i+\frac12}^{{\tt WENO}} = \theta_{\ell, i+\frac12} ( {\omega }_{\ell,i+\frac12}^+ - {\omega  }_{\ell,i+\frac12}^- ),  \quad \theta_{\ell,i+\frac12}:=\begin{cases}
1, \  & ( {\omega }_{\ell,i+\frac12}^+ - {\omega  }_{\ell,i+\frac12}^- ) \jump{\omega_\ell}_{i+\frac12} > 0,\\
0, \ & {\rm otherwise,}
\end{cases}
\end{equation}
with ${\omega }_{\ell,i+1/2}^- $ and $ {\omega  }_{\ell,i+1/2}^+$ denoting, respectively, the left and right limiting values of $\omega_\ell$ at interface $x_{i+1/2}$ by using $(2k-1)$th-order WENO reconstruction. Although the standard WENO reconstruction may not satisfy the sign stability, the  
use of switch operator $\theta_{\ell, i+\frac12}$ in \eqref{keyddffd}, proposed in \cite{biswas2018low}, ensures that %a ``weak'' sign property
\begin{equation}\label{WENOsign}
{\rm sign} \big( \llangle {\bm \omega  } \rrangle_{i+\frac12}^{{\tt WENO}} \big) = {\rm sign} \big( \jump{\bm \omega}_{i+\frac12} \big).
\end{equation}

\begin{theorem}\label{thm:1DknES}
	The scheme \eqref{1Dsemischeme}, with 
	$\widehat{ B }_{1,i+\frac12} = \widetilde { B }_{1,i+\frac12}^{2k,\star} $ and the ENO-based numerical flux \eqref{ECfluxENO} or the WENO-based numerical flux \eqref{ECfluxWENO}, is entropy stable, and the corresponding numerical entropy flux is given by 
	\begin{equation}\label{ES2kQ}
	\widehat {\mathcal Q}_{i+\frac12} =  \widetilde {\mathcal Q}_{i+\frac12}^{2k,\star} - \frac12 \dgal{ {\bf W} }_{i+\frac12}^\top  {\mathbf R}_{i+\frac12} |{\bf \Lambda}_{i+\frac12}| 
	\llangle {\bm \omega  } \rrangle_{i+\frac12},
	\end{equation}
	where $\widetilde {\mathcal Q}_{i+1/2}^{2k,\star}$ is defined in \eqref{ECkQ}, and $\llangle {\bm \omega  } \rrangle_{i+1/2}$ is taken as $\llangle {\bm \omega  } \rrangle_{i+1/2}^{\tt ENO}$ or $\llangle {\bm \omega  } \rrangle_{i+1/2}^{\tt WENO}$ accordingly. 
\end{theorem}

\begin{proof}
	First, it is evident that the numerical entropy flux \eqref{ES2kQ} is consistent with the entropy flux $\mathcal Q$.  Substituting  $\widehat{ B }_{1,i+1/2} = \widetilde { B }_{1,i+1/2}^{2k,\star} $ and numerical flux \eqref{ECfluxENO} or \eqref{ECfluxWENO} into the scheme \eqref{1Dsemischeme}, and then following the proof of Theorem \ref{thm:1DknEC}, we obtain 
	\begin{align*}
	-\Delta x & \frac{ \rm d }{ {\rm d} t } {\mathcal E} ( {\bf U}_i )  = \widetilde {\mathcal Q}_{i+\frac12}^{2k,\star} - \widetilde {\mathcal Q}_{i-\frac12}^{2k,\star} - \frac12 {\bf W}_i^\top \left( 
	{\mathbf R}_{i+\frac12} |{\bf \Lambda}_{i+\frac12}| 
	\llangle {\bm \omega  } \rrangle_{i+\frac12} - {\mathbf R}_{i-\frac12} |{\bf \Lambda}_{i-\frac12}| 
	\llangle {\bm \omega  } \rrangle_{i-\frac12}
	\right)
	\\
	& = \widehat{\mathcal Q}_{i+\frac12} - \widehat {\mathcal Q}_{i-\frac12} 
	+ \frac14 \left( \jump{{\bf W}}_{i+\frac12}^\top  {\mathbf R}_{i+\frac12} |{\bf \Lambda}_{i+\frac12}| 
	\llangle {\bm \omega  } \rrangle_{i+\frac12}  + \jump{{\bf W}}_{i-\frac12}^\top  {\mathbf R}_{i-\frac12} |{\bf \Lambda}_{i-\frac12}| 
	\llangle {\bm \omega  } \rrangle_{i-\frac12} \right)
	\\
	& = \widehat{\mathcal Q}_{i+\frac12} - \widehat {\mathcal Q}_{i-\frac12} 
	+ \frac14 \left( \jump{{\bm \omega}}_{i+\frac12}^\top  |{\bf \Lambda}_{i+\frac12}| 
	\llangle {\bm \omega  } \rrangle_{i+\frac12}  + \jump{{\bm \omega}}_{i-\frac12}^\top   |{\bf \Lambda}_{i-\frac12}| 
	\llangle {\bm \omega  } \rrangle_{i-\frac12} \right).
	\end{align*}
	It follows that 
	$$
	\frac{\rm d }{ {\rm d} t } {\mathcal E} ( {\bf U}_i ) + \frac1{\Delta x} 
	\left(  \widehat{\mathcal Q}_{i+\frac12} - \widehat {\mathcal Q}_{i-\frac12}  \right)
	= -\frac1{4\Delta x} \left( \jump{{\bm \omega}}_{i+\frac12}^\top  |{\bf \Lambda}_{i+\frac12}| 
	\llangle {\bm \omega  } \rrangle_{i+\frac12}  + \jump{{\bm \omega}}_{i-\frac12}^\top   |{\bf \Lambda}_{i-\frac12}| 
	\llangle {\bm \omega  } \rrangle_{i-\frac12} \right) \le 0, 
	$$
	where the last inequality is obtained by using \eqref{ENOsign} or \eqref{WENOsign} accordingly. 
	Therefore, the computed solutions of the scheme satisfy a discrete entropy inequality \eqref{DefESscheme} for the numerical entropy flux \eqref{ES2kQ}. 
\end{proof}

\begin{remark}
	The entropy stability of the scheme in Theorem \ref{thm:1DknES} 
	is established only at the semi-discrete level. With explicit time discretization by, for example, a
	Runge-Kutta method, we cannot prove the entropy stability of the resulting fully discrete schemes. 
	The entropy stability of fully discrete schemes will be only demonstrated by numerical experiments in Sect.~\ref{sec:examples}.
\end{remark}

\section{Entropy Conservative Schemes and Entropy Stable Schemes in Two Dimensions}\label{sec:2Dschemes}

The 1D entropy conservative schemes and entropy stable schemes developed in Sect.~\ref{sec:1Dschemes} can be easily extended to the multidimensional cases on rectangular meshes. This section presents the extension for 2D RMHD equations \eqref{ModRMHD} with $d=2$. 
To avoid confusing subscripts, we will use $(x,y)$ to denote the 2D spatial coordinates.  

Let us consider a uniform 2D Cartesian mesh consisting of grid points $(x_i,y_j)=(i\Delta x, j \Delta y)$ for $i,j \in \mathbb Z$, where both spatial step-sizes $\Delta x$ and $\Delta y$ are given positive constants. A semi-discrete finite difference scheme for 2D modified RMHD equations \eqref{ModRMHD} can be written as 
\begin{equation}\label{2Dsemischeme}
\begin{aligned}
& \frac{\rm d}{{\rm d} t} {\bf U}_{ij}(t)  + \frac{ \widehat{\bf F}_{1,i+\frac12,j}(t) -\widehat{\bf F}_{1,i-\frac12,j}(t)  }{\Delta x} 
+ \frac{ \widehat{\bf F}_{2,i,j+\frac12}(t) -\widehat{\bf F}_{2,i,j-\frac12}(t)  }{\Delta y} 
\\
& 
+ {\bf S}( {\bf U}_{ij}(t) ) \left( \frac{ \widehat{ B }_{1,i+\frac12,j}(t) -\widehat{B}_{1,i-\frac12,j}(t)  }{\Delta x} 
+ \frac{ \widehat{ B }_{2,i,j+\frac12}(t) -\widehat{B}_{2,i,j-\frac12}(t)  }{\Delta y} 
\right) = {\bf 0},
\end{aligned}
\end{equation}
where ${\bf U}_{ij}(t) \approx {\bf U}(x_{ij},t) $, and $ \widehat{\bf F}_{1,i+1/2,j} $ (resp.\ $ \widehat{\bf F}_{2,i,j+1/2}  $) is numerical flux consistent with ${\bf F}_1$ (resp.\ ${\bf F}_2$). For convenience, the $t$ dependence of all quantities is suppressed below.

\subsection{Entropy Conservative Schemes}

The semi-discrete scheme \eqref{2Dsemischeme} is said to be entropy conservative if its computed solutions satisfy 
a discrete entropy equality
\begin{equation}\label{DefECscheme2D}
\frac{ \rm d }{ {\rm d} t} {\mathcal E} ( {\bf U}_{ij} ) + \frac{ 1 }{\Delta x} \left( \widetilde {\mathcal Q}_{1,i+\frac12,j}  -   \widetilde {\mathcal Q}_{1,i-\frac12,j} \right)
+ \frac{ 1 }{\Delta y} \left( \widetilde {\mathcal Q}_{2,i,j+\frac12}  -   \widetilde {\mathcal Q}_{2,i,j-\frac12} \right) =0
\end{equation}
for some numerical entropy fluxes $ \widetilde {\mathcal Q}_{1,i+\frac12,j} $ and 
$ \widetilde {\mathcal Q}_{2,i,j+\frac12} $ consistent with the entropy flux ${\mathcal Q}_1$ and ${\mathcal Q}_2$, respectively.

Analogously to the one-dimensional case, for any $k \in \mathbb N_+$ we define 
\begin{equation}\label{HighECfluxXY}
\begin{aligned}
& \widetilde { \bf F }_{1,i+\frac12,j}^{2k,\star }:= 
\sum_{ r=1 }^k \alpha_{ k,r } \sum_{ s=0 }^{r-1}  {\bf F}_1^\star ( {\bf U}_{i-s,j}, {\bf U}_{i-s+r,j}  ),~~~~
\widetilde { \bf F }_{2,i,j+\frac12}^{2k,\star }:= 
\sum_{ r=1 }^k \alpha_{ k,r } \sum_{ s=0 }^{r-1}  {\bf F}_2^\star ( {\bf U}_{i,j-s}, {\bf U}_{i,j-s+r}  ),
\\
& \widetilde { B }_{1,i+\frac12,j}^{2k,\star} :=  \sum_{ r=1 }^k \alpha_{ k,r } \sum_{ s=0 }^{r-1}  \bigg( \frac{B_{1,i-s,j} + B_{1,i-s+r,j} }2  \bigg),~~~~
\widetilde { B }_{2,i,j+\frac12}^{2k,\star} :=  \sum_{ r=1 }^k \alpha_{ k,r } \sum_{ s=0 }^{r-1}   \bigg( \frac{B_{2,i,j-s} + B_{2,i,j-s+r}}2 \bigg),
\end{aligned}
\end{equation}
where the constants $\alpha_{k,r}$ is given by \eqref{alpha_condition}.

\begin{theorem}
	For any $k \in \mathbb N_+$, the scheme \eqref{2Dsemischeme} with 
	$$
	\widehat { \bf F }_{1,i+\frac12,j} =  \widetilde { \bf F }_{1,i+\frac12,j}^{2k,\star }, \ \ \
	\widehat { \bf F }_{2,i,j+\frac12} =  \widetilde { \bf F }_{2,i,j+\frac12}^{2k,\star }, \ \ \
	\widehat { B }_{1,i+\frac12,j} = \widetilde { B }_{1,i+\frac12,j}^{2k,\star}, \ \ \
	\widehat { B }_{2,i,j+\frac12} = \widetilde { B }_{2,i,j+\frac12}^{2k,\star}
	$$
	is a $2k$th-order accurate entropy conservative scheme with 
	the %corresponding 
	numerical entropy fluxes %given by
	\begin{equation}\label{ECkQ2D}
	\widetilde {\mathcal Q}_{1,i+\frac12,j}^{2k,\star} = \sum_{ r=1 }^k \alpha_{ k,r } \sum_{ s=0 }^{r-1}  \widetilde {\mathcal Q}_1  ( {\bf U}_{i-s,j}, {\bf U}_{i-s+r,j}  ), \quad 
	\widetilde {\mathcal Q}_{2,i,j+\frac12}^{2k,\star} = \sum_{ r=1 }^k \alpha_{ k,r } \sum_{ s=0 }^{r-1}  \widetilde {\mathcal Q}_2  ( {\bf U}_{i,j-s}, {\bf U}_{i,j-s+r}  ),
	\end{equation}
	where the function $\widetilde {\mathcal Q}_\ell $ is defined by
	\begin{align*}
	\widetilde {\mathcal Q}_\ell ( {\bf U}_L, {\bf U}_R  ) := \frac12 \left( {\bf W}_L + {\bf W}_R \right) \cdot {\bf F}_\ell^\star ( {\bf U}_L, {\bf U}_R  ) +   \frac{ \phi_L + \phi_R}{2}   \left( \frac{B_{\ell,L} + B_{\ell,R}}{2} \right) - \frac{ \psi_{\ell,L} + \psi_{\ell,R}}{2}.
	\end{align*}
\end{theorem}

	The proof is similar to those of Theorems \ref{thm:1D2nEC}--\ref{thm:1DknEC} and is omitted here.

\subsection{Entropy Stable Schemes}

The 2D semi-discrete scheme \eqref{2Dsemischeme} is said to be entropy stable if its computed solutions satisfy a discrete entropy inequality
\begin{equation}\label{DefESscheme2D}
\frac{ \rm d }{ {\rm d} t} {\mathcal E} ( {\bf U}_{ij} ) + \frac{ 1 }{\Delta x} \left( \widehat {\mathcal Q}_{1,i+\frac12,j}  -   \widehat {\mathcal Q}_{1,i-\frac12,j} \right)
+ \frac{ 1 }{\Delta y} \left( \widehat {\mathcal Q}_{2,i,j+\frac12}  -   \widehat {\mathcal Q}_{2,i,j-\frac12} \right) \le 0
\end{equation}
for some numerical entropy fluxes $ \widehat {\mathcal Q}_{1,i+\frac12,j} $ and 
$ \widehat {\mathcal Q}_{2,i,j+\frac12} $ consistent with the entropy flux ${\mathcal Q}_1$ and ${\mathcal Q}_2$, respectively.  

Analogously to the one-dimensional case, we define 
\begin{equation}\label{HighESfluxXY}
\begin{aligned}
& \widehat { \bf F }_{1,i+\frac12,j} =  \widetilde { \bf F }_{1,i+\frac12,j}^{2k,\star } - \frac12  {\mathbf R}_{1,i+\frac12,j} |{\bf \Lambda}_{1,i+\frac12,j}| 
\llangle {\bm \omega  } \rrangle_{i+\frac12,j},
\\
& \widehat { \bf F }_{2,i,j+\frac12} =  \widetilde { \bf F }_{2,i,j+\frac12}^{2k,\star } - \frac12  {\mathbf R}_{2,i,j+\frac12} |{\bf \Lambda}_{2,i,j+\frac12}| 
\llangle {\bm \omega  } \rrangle_{i,j+\frac12},
\end{aligned}
\end{equation}
where ${\mathbf R}_{1,i+\frac12,j}$ (resp.\ ${\mathbf R}_{2,i,j+\frac12}$) is the matrix formed by the scaled right eigenvectors of the Jacobian matrix ${\bf A}_1 ({\bf U}_{i+\frac12,j}^{\tt ave}) := {\bf F}'_1 ( {\bf U}_{i+\frac12,j}^{\tt ave} ) + {\bf S} ({\bf U}_{i+\frac12,j}^{\tt ave}) B_1'( {\bf U}_{i+\frac12,j}^{\tt ave} )$ (resp.\ ${\bf A}_2 ({\bf U}_{i,j+\frac12}^{\tt ave}) := {\bf F}'_2 ( {\bf U}_{i,j+\frac12}^{\tt ave} ) + {\bf S} ({\bf U}_{i,j+\frac12}^{\tt ave}) B_2'( {\bf U}_{i,j+\frac12}^{\tt ave} )$); the diagonal matrix $|{\bf \Lambda}_{1,i+\frac12,j}|$ (resp.\ $|{\bf \Lambda}_{2,i,j+\frac12}|$) is defined as \eqref{eq:Roe} or \eqref{eq:Rusanov} with the eigenvalues of ${\bf A}_1 ({\bf U}_{i+\frac12,j}^{\tt ave})$ (resp.\ ${\bf A}_2 ({\bf U}_{i,j+\frac12}^{\tt ave})$). Here ${\bf U}_{i+\frac12,j}^{\tt ave}$ and ${\bf U}_{i,j+\frac12}^{\tt ave}$ denote the average states at the corresponding interfaces, analogously to the 1D case defined in \eqref{AVEstate}. The ``high-order accurate'' jumps $\llangle {\bm \omega  } \rrangle_{i+1/2,j}$ and $\llangle {\bm \omega  } \rrangle_{i,j+1/2}$ in \eqref{HighESfluxXY} are computed by ENO reconstruction or WENO reconstruction using switch operator, which can be performed precisely as in the 1D case, dimension by dimension.

\begin{theorem}
	The scheme \eqref{2Dsemischeme}, with $\widehat { B }_{1,i+1/2,j} = \widetilde { B }_{1,i+1/2,j}^{2k,\star},$ 
	$\widehat { B }_{2,i,j+1/2} = \widetilde { B }_{2,i,j+1/2}^{2k,\star}$ and 
	numerical fluxes \eqref{HighESfluxXY}, is entropy stable, and the numerical entropy flux is %given by 
	\begin{align*}
	&\widehat {\mathcal Q}_{1,i+\frac12,j} =  \widetilde {\mathcal Q}_{1,i+\frac12,j}^{2k,\star} - \frac14 ( {\bf W}_{i,j} + {\bf W}_{i+1,j}   )^\top {\mathbf R}_{1,i+\frac12,j} |{\bf \Lambda}_{1,i+\frac12,j}| 
	\llangle {\bm \omega  } \rrangle_{i+\frac12,j},
	\\
	&\widehat {\mathcal Q}_{2,i,j+\frac12} =  \widetilde {\mathcal Q}_{2,i,j+\frac12}^{2k,\star} - \frac14  ( {\bf W}_{i,j} + {\bf W}_{i,j+1}   )^\top  {\mathbf R}_{2,i,j+\frac12} |{\bf \Lambda}_{2,i,j+\frac12}| 
	\llangle {\bm \omega  } \rrangle_{i,j+\frac12},
	\end{align*}
	where $\widetilde {\mathcal Q}_{1,i+1/2,j}^{2k,\star}$ and $\widetilde {\mathcal Q}_{2,i,j+1/2}^{2k,\star}$ are defined in \eqref{ECkQ2D}. 
\end{theorem}

	The proof is similar to that of Theorem \ref{thm:1DknES} and thus is omitted here.

\section{Numerical Experiments}\label{sec:examples}

In this section, we conduct numerical experiments on several 1D and 2D benchmark RMHD problems, 
to demonstrate the performance of the proposed high-order accurate entropy stable schemes and entropy conservative schemes. 
For convenience, we abbreviate the forth-order and sixth-order accurate ($k=2,3$ respectively) entropy conservative schemes as {\tt EC4} and {\tt EC6}, respectively. 
The 1D and 2D entropy stable schemes with 1D numerical flux \eqref{ECfluxENO} or 2D numerical flux \eqref{HighESfluxXY},  using $k=2$ and fourth-order accurate ENO reconstruction,   are abbreviated as {\tt ES4}. 
The 1D and 2D entropy stable schemes with 1D numerical flux \eqref{ECfluxWENO} or 2D numerical flux \eqref{HighESfluxXY},  using $k=3$ and fifth-order accurate WENO reconstruction,   are abbreviated as {\tt ES5}.   
Unless otherwise stated, 
all these semi-discrete schemes are equipped with 
a fourth-order accurate explicit total-variation-diminishing (TVD) Runge-Kutta time discretization to obtain fully discrete schemes, and
we use a CFL number of $0.4$ and the ideal equation of state $p=(\Gamma-1)\rho e$ with $\Gamma=5/3$. 
In all the tests, the Rusanov-type dissipation operator is employed in the entropy stable schemes.

%%%%%%%%%%%%%%%%%%%%%%% Example 1 %%%%%%%%%%%%%%%%%%%%%%%%%%

\begin{expl}[Smooth problem]\label{example1D:accuracy}\rm
	This test is used to check the accuracy of our schemes. Consider a 1D smooth problem which 
	describes Alfv\'en waves propagating periodically within the domain $[0,1]$ and has the exact solution
	$$
	{\bf V}(x,t)=( 1,~0,~v_2(x,t),~v_3(x,t),~1,~\sigma v_2(x,t),~\sigma v_3(x,t),~0.01 )^\top, \quad (x,t)\in [0,1] \times {\mathbb{R}}^+,
	$$
	where the vector $\bf V$ denotes the primitive variables as defined in \eqref{priV}, $\sigma=\sqrt{1+\rho h \gamma^2}$, $v_2(x,t) = 0.2 \sin( 2 \pi ( x + t/\sigma )  )$, and 
	$v_3(x,t)=0.2\cos( 2 \pi ( x + t/\sigma ) )$.

	In our computations, the domain $[0,1]$ is divided into $N$ uniform cells, and periodic boundary conditions are specified. 
	The time step-size is taken as $\Delta t = 0.4 \Delta x^{\frac64} $ and $\Delta t = 0.4 \Delta x^{\frac54} $ for 
	$\tt EC6$ and $\tt ES 5$, respectively, in order to make the error in spatial discretization dominant in the present case. 
	Tables \ref{tab:acc1} and \ref{tab:acc2} 
	list the numerical errors at $t = 0.5$ in the numerical velocity component $v_2$ and the corresponding
	convergence rates for the $\tt EC4$, $\tt ES4$, $\tt EC6$ and $\tt ES 5$ schemes at different grid resolutions. The convergence behaviors for $v_3$, $B_2$ and $B_3$ are similar and omitted. 
	We clearly observe that the expected convergence orders of the schemes are achieved accordingly.
		Figure \ref{fig:ex1} displays the evolution of discrete total entropy $\sum_{i} {\mathcal E} ( {\bf U}_i (t)  ) \Delta x$, which approximates $\int_0^1 {\mathcal E} ( {\bf U}(x,t) ) dx $ that 
	remains conservative 
	for this smooth problem. For the entropy stable schemes, the discrete total entropy 
	is not constant and slowly decays due to
	the numerical dissipation, but we observe convergence with grid refinement, while for the
	entropy conservative schemes, it is nearly constant with time as expected. 
	
	\begin{table}[htbp]
		\centering
		\caption{\small Example \ref{example1D:accuracy}: 
			$l^1$-errors and $l^2$-errors in $v_2$ at $t=0.5$, and corresponding convergence rates for
			the {\tt EC4} and {\tt ES4} schemes  at
			different grid resolutions.
		}\label{tab:acc1}
		\begin{tabular}{c||c|c|c|c||c|c|c|c}
			\hline
			\multirow{2}{12pt}{$N$}
			&\multicolumn{4}{c||}{ {\tt EC4} }
			&\multicolumn{4}{c}{ {\tt ES4} } 
			\\
			\cline{2-9}
			& $l^1$-error& order & $l^2$-error & order & $l^1$-error& order & $l^2$-error & order \\
			\hline
			8& 3.14e-3& --          &3.53e-3& --      & 4.23e-3 & --   & 4.62e-3  &-- \\
			16& 2.10e-4& 3.91   & 2.33e-4&    3.92    & 2.34e-4 & 4.17 & 2.60e-4 & 4.15\\
			32& 1.33e-5& 3.98   & 1.48e-5&  3.98      & 1.38e-5 & 4.09 & 1.54e-5 & 4.08 \\
			64  & 8.34e-7 &  3.99    & 9.27e-7&  3.99 & 8.46e-7 & 4.03 & 9.40e-7& 4.03 \\
			128 & 5.22e-8 &  4.00 & 5.80e-8&    4.00  & 5.25e-8 & 4.01 & 5.83e-8 & 4.01 \\
			256 &  3.26e-9 &  4.00  &3.62e-9&  4.00   & 3.27e-9 & 4.00 & 3.63e-9 & 4.00 \\
			512 &  2.04e-10 &  4.00  &2.26e-10&  4.00 & 2.04e-10 & 4.00 & 2.27e-10 & 4.00 \\
			\hline
		\end{tabular}
	\end{table}

	\begin{table}[htbp]
		\centering
		\caption{\small Same as Table \ref{tab:acc1} except for 
			the {\tt EC6} and {\tt ES5} schemes.
		}\label{tab:acc2}
		\begin{tabular}{c||c|c|c|c||c|c|c|c}
			\hline
			\multirow{2}{12pt}{$N$}
			&\multicolumn{4}{c||}{ {\tt EC6} }
			&\multicolumn{4}{c}{ {\tt ES5} } 
			\\
			\cline{2-9}
			& $l^1$-error& order & $l^2$-error & order & $l^1$-error& order & $l^2$-error & order \\
			\hline
			8&  3.85e-4 & --          & 4.36e-4 & --      & 5.64e-3 & --   & 6.16e-3  &-- \\
			16&  7.11e-6 & 5.76   & 7.96e-6 &    5.78    & 2.62e-4 & 4.43 & 2.94e-4 & 4.39\\
			32& 1.18e-7 & 5.91   & 1.31e-7 &  5.93     & 8.59e-6 & 4.93 & 9.53e-6 & 4.95 \\
			64  & 1.86e-9 &  5.98    & 2.07e-9 &  5.98  & 2.68e-7 & 5.00 & 2.97e-7 & 5.00 \\
			128 & 2.92e-11 &  6.00 & 3.24e-11 &    5.99  & 8.35e-9 & 5.00 & 9.27e-9 & 5.00 \\
			256 &  4.81e-13 &  5.92  &5.34e-13 &  5.92   & 2.61e-10 & 5.00 & 2.89e-10 & 5.00 \\
			\hline
		\end{tabular}
	\end{table}

		\begin{figure}[htbp]
		\centering
		\includegraphics[width=0.48\textwidth]{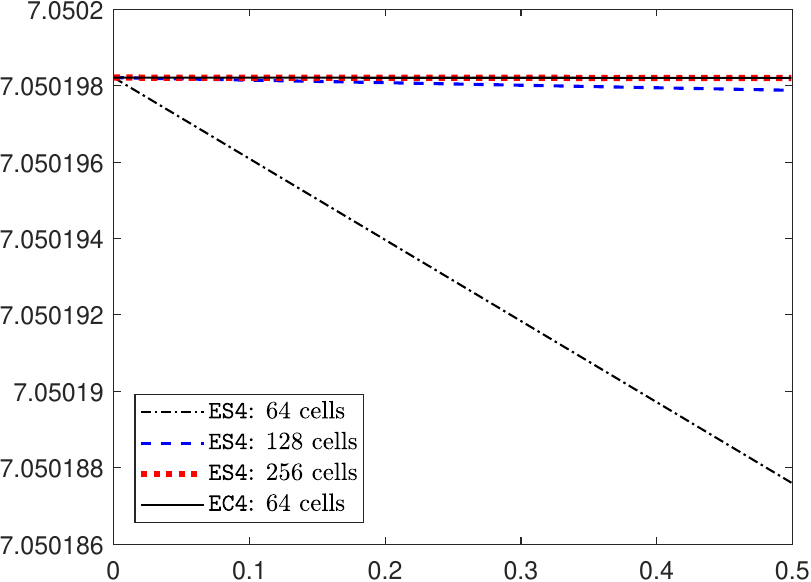}
		\includegraphics[width=0.48\textwidth]{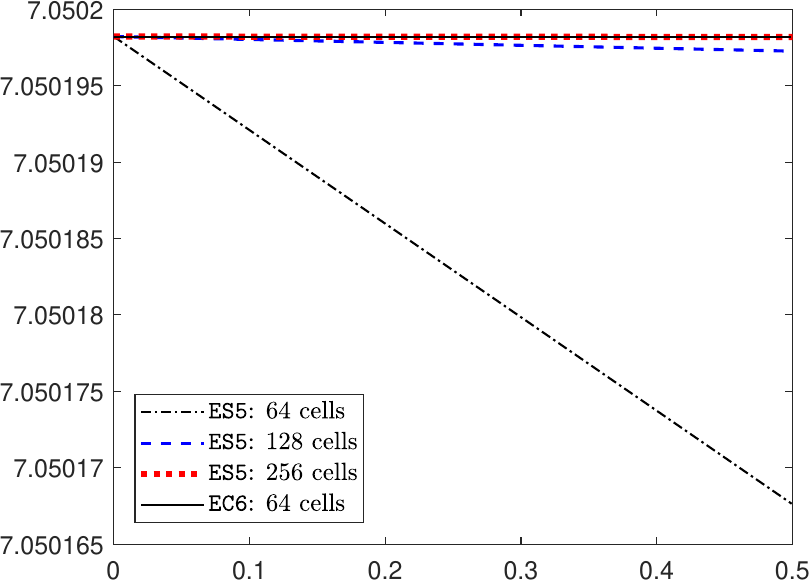}
		\caption{\small
			Example \ref{example1D:accuracy}: Evolution of discrete total entropy. 
		}\label{fig:ex1} 
	\end{figure}
	
\end{expl}

\begin{expl}[1D Riemann problems]\label{example1DRiemanns}\rm
	This example 
	verifies the
	performance of the proposed entropy conservative and entropy stable schemes in 
	resolving 1D RMHD wave configurations, by simulating four 1D Riemann problems (RPs) \cite{Balsara2001,komissarov1999godunov}. 
	%whose exact solutions were provided in \cite{giacomazzo_rezzolla_2006}. 
	The initial data of each RP comprise two different constant states separated by the initial discontinuity at $x=0$, see Table \ref{tab:riemann1d}. 
	Their exact solutions were provided in \cite{giacomazzo_rezzolla_2006}.

	\begin{table}[htbp]
		\centering %\vspace{0.1cm}
		\caption{\small Initial data of the four 1D RPs in Example \ref{example1DRiemanns}.
		}
		\begin{tabular}{c|c|cccccccc}
			\hline
			\multicolumn{2}{c|}{} & $\rho $   & $ v_1 $ &  $v_2 $   &  $v_3 $  &  $ B_1 $  & $ B_2 $  &  $ B_3 $  &  $ p $  \\
			\hline
			\multirow{2}{32pt}{RP \uppercase\expandafter{\romannumeral1} }
			&
			left state  & 1.08 & 0.4  & 0.3
			& 0.2 & 2 & 0.3 & 0.3 & 0.95 \\
			\cline{2-10}
			& right state & 1 & -0.45  & -0.2
			& 0.2 & 2.0 & -0.7 & 0.5 & 1  \\
			\hline
			\multirow{2}{32pt}{RP \uppercase\expandafter{\romannumeral2} }
			&
			left state  & 1 & 0  & 0
			& 0 & 5 & 6 & 6 & 30 \\
			\cline{2-10}
			& right state & 1 & 0  & 0
			& 0 & 5 & 0.7 & 0.7 & 1  \\
			\hline
			\multirow{2}{32pt}{RP \uppercase\expandafter{\romannumeral3}}
			&
			left state  & 1 & 0  & 0.3
			& 0.4 & 1 & 6 & 2 & 5 \\
			\cline{2-10}
			& right state & 0.9 & 0  & 0
			& 0 & 1 & 5 & 2 & 5.3 \\
			\hline
			\multirow{2}{32pt}{RP \uppercase\expandafter{\romannumeral4}}
			&
			left state  & 1 & 0  & 0
			& 0 & 0 & 2 & 0 & 10 \\
			\cline{2-10}
			& right state & 0.1 & 0  & 0
			& 0 & 0 & 0 & 0 & 1 \\
			\hline			
		\end{tabular}\label{tab:riemann1d}
	\end{table}
	
Figures~\ref{fig:RP1}, \ref{fig:RP2} and \ref{fig:RP3} give the numerical results %at $t=0.55$
 obtained by using {\tt ES5} with $1000$ uniform cells %within the domain $[-0.5, 0.5]$. 
 for RPs \uppercase\expandafter{\romannumeral1}, \uppercase\expandafter{\romannumeral2} and  \uppercase\expandafter{\romannumeral3}, respectively. 
We see 
that the wave structures including discontinuities are well resolved by 
{\tt ES5} and that the
numerical solutions are in good agreement with the exact ones.

		\begin{figure}[htbp]
		\centering
		{\includegraphics[width=0.48\textwidth]{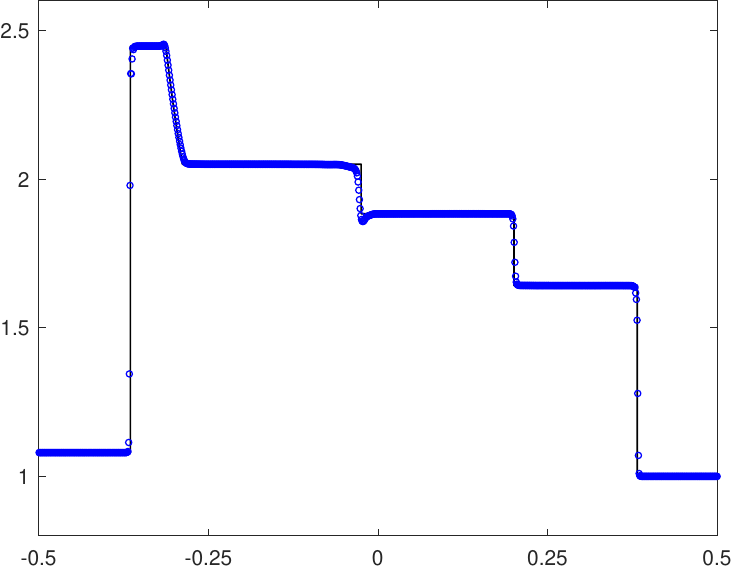}}
		{\includegraphics[width=0.48\textwidth]{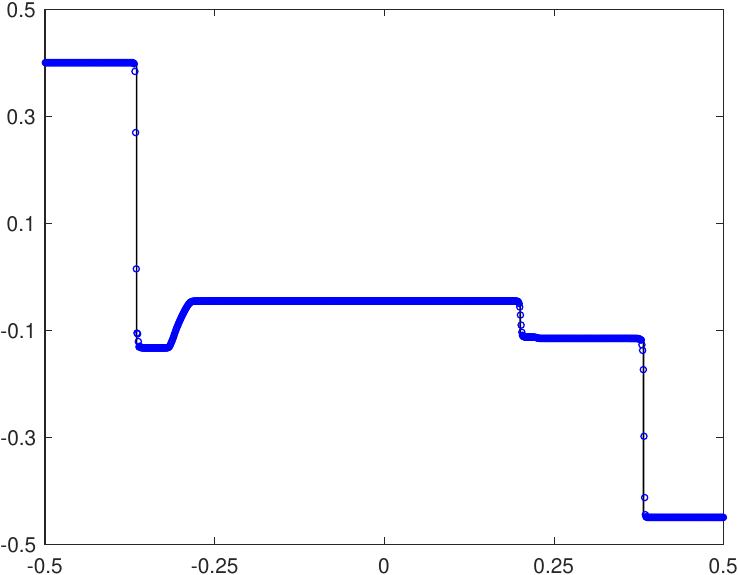}}
		{\includegraphics[width=0.48\textwidth]{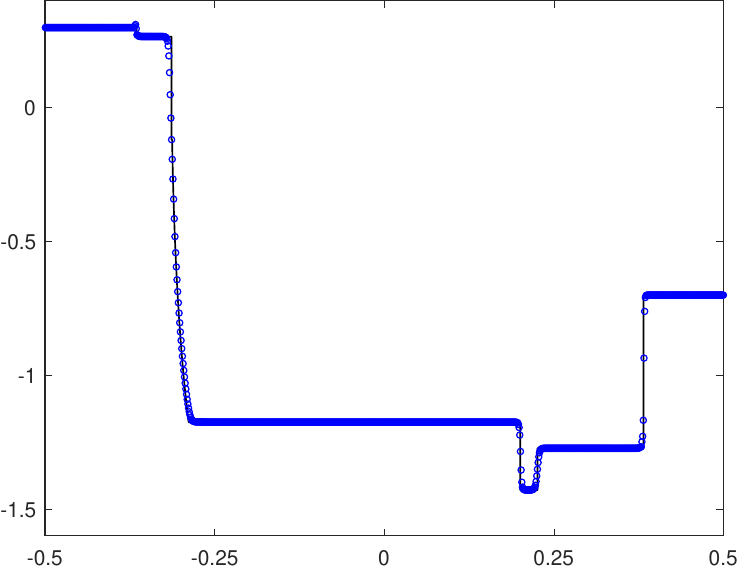}}
		{\includegraphics[width=0.48\textwidth]{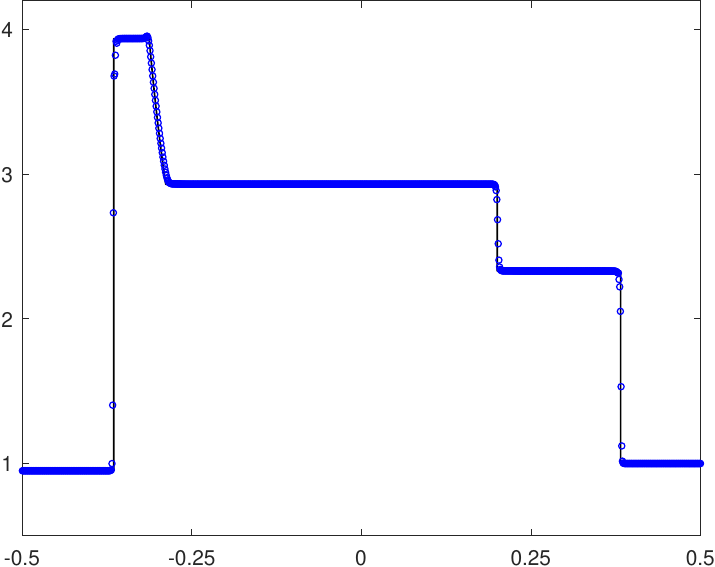}}
		\caption{\small The first Riemann problem in Example \ref{example1DRiemanns}: The density $\rho$ (top-left),
			velocity $v_1$ (top-right), magnetic field component $B_2$ (bottom-left), and pressure $p$ (bottom-right) at $t=0.55$. 
			The symbols ``${\color{blue}\circ}$'' denote the numerical solutions 
			computed by {\tt ES5}, while 
			the solid lines denote the exact solutions.}
		\label{fig:RP1}
	\end{figure}

		\begin{figure}[htbp]
		\centering
		{\includegraphics[width=0.48\textwidth]{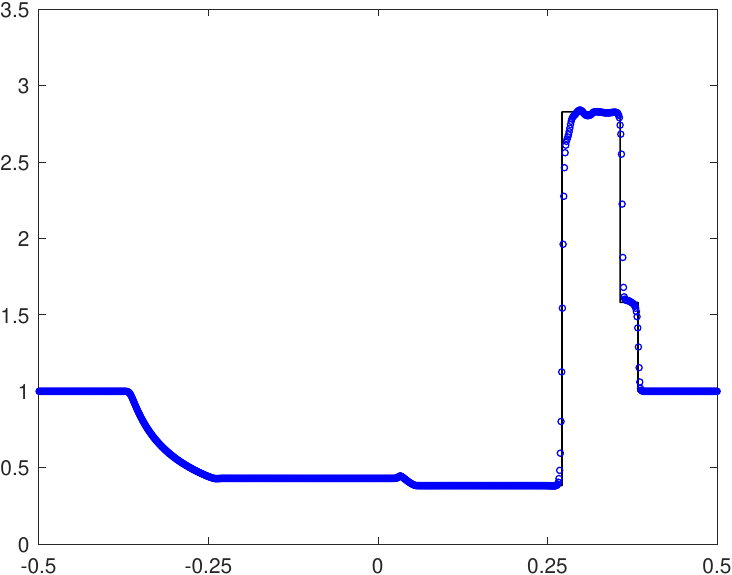}}
		{\includegraphics[width=0.48\textwidth]{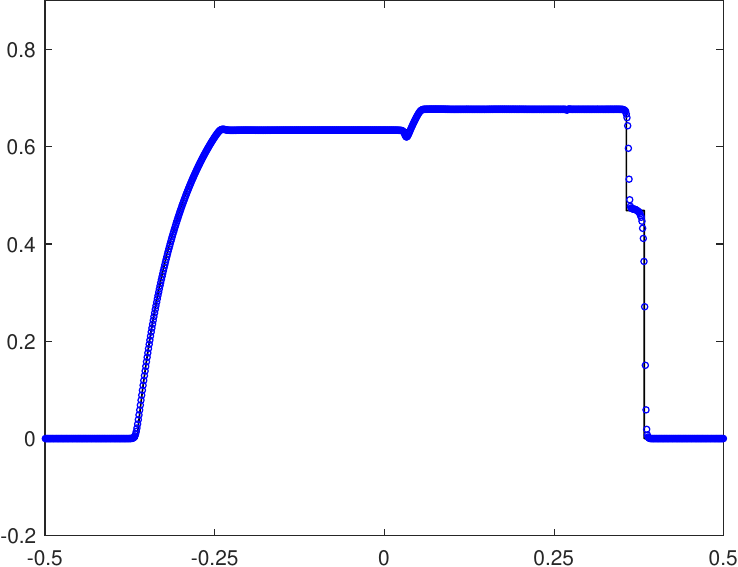}}
		{\includegraphics[width=0.48\textwidth]{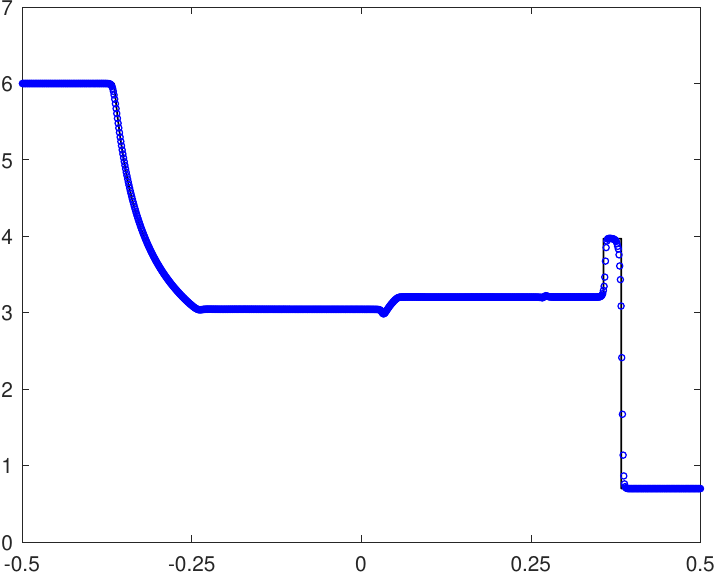}}
		{\includegraphics[width=0.48\textwidth]{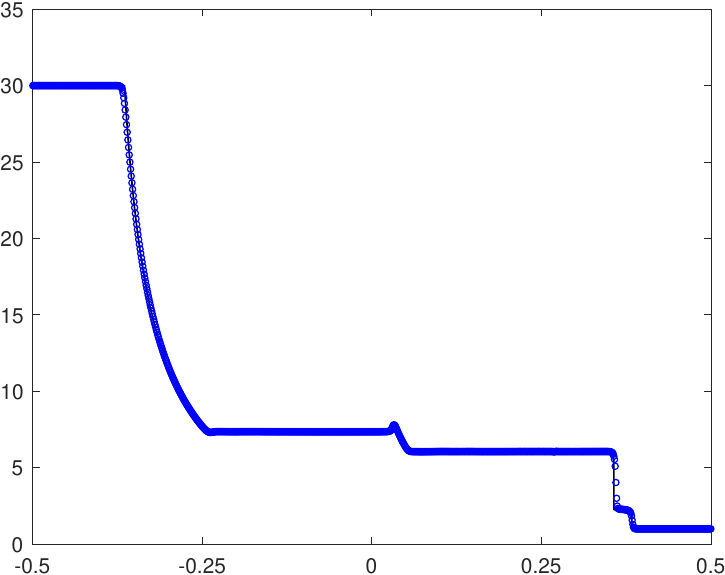}}
		\caption{\small Same as Figure~\ref{fig:RP1} except for the second Riemann problem at $t=0.4$.}
		\label{fig:RP2}
	\end{figure}

	\begin{figure}[htbp]
		\centering
		{\includegraphics[width=0.48\textwidth]{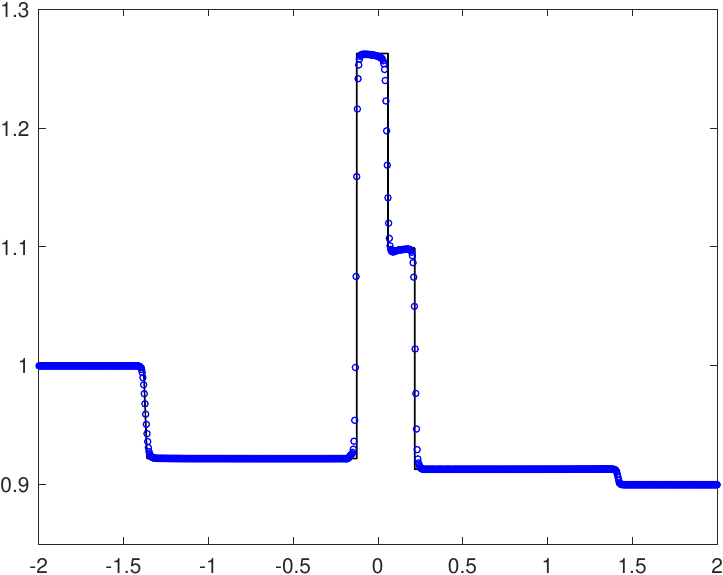}}
		{\includegraphics[width=0.48\textwidth]{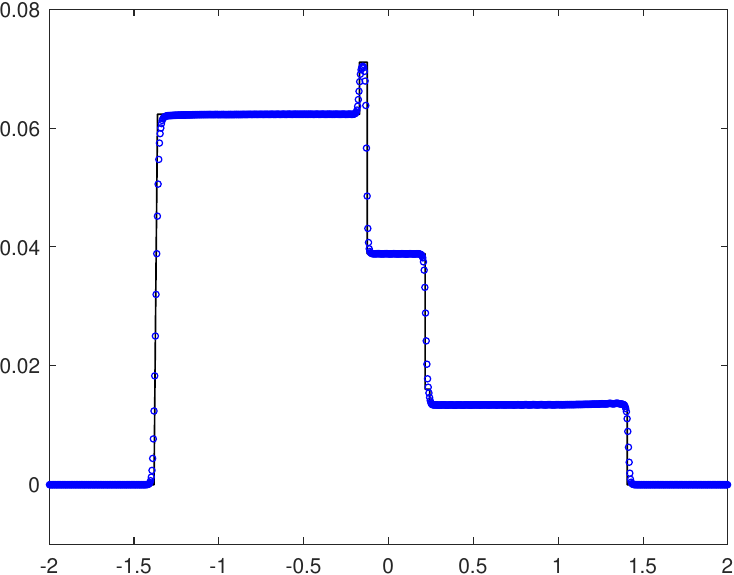}}
		{\includegraphics[width=0.48\textwidth]{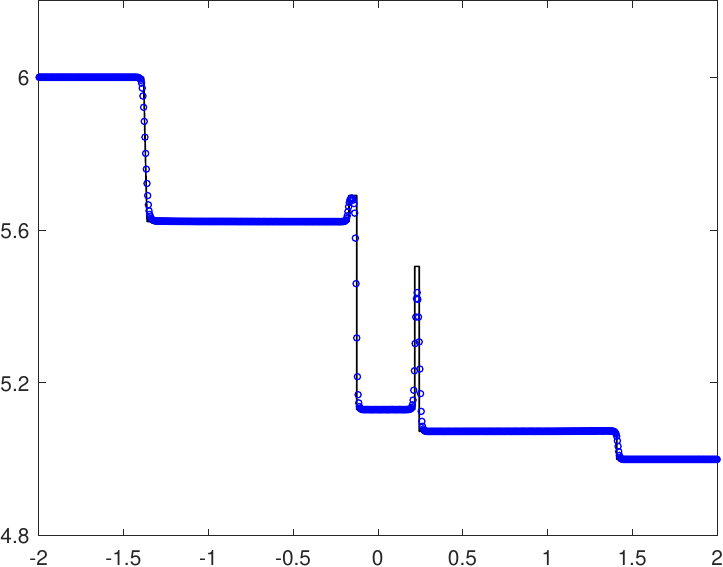}}
		{\includegraphics[width=0.48\textwidth]{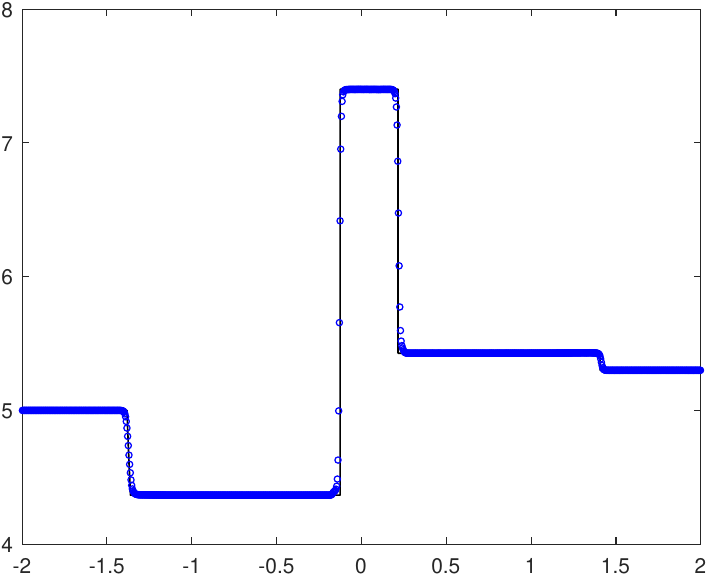}}
		\caption{\small Same as Figure~\ref{fig:RP1} except for the third Riemann problem at $t=1.5$.}
		\label{fig:RP3}
	\end{figure}

The RP \uppercase\expandafter{\romannumeral4}  is a variant of the RMHD shock-tube 2 proposed by Komissarov \cite{komissarov1999godunov} with $\Gamma = \frac 43$. 
The numerical solutions at $t=1$, computed respectively by {\tt EC6} and {\tt ES5} on $1000$ uniform girds, 
are displayed in  
Figure~\ref{fig:RP4}. 
We see that {\tt EC6} produces high-frequency oscillations in its numerical solution, because it enforces 
entropy conservation so that the entropy 
dissipation at the shock does not take place. 
Oscillations are not produced in the numerical solution of {\tt ES5}, owing to its WENO dissipative mechanism.  
In order to carefully verify the entropy conservative/stable property of the proposed schemes, we 
investigate the behavior of the discrete total entropy ${\mathcal E}_{\rm tot} (t):=\sum_{i} {\mathcal E} ( {\bf U}_i (t)  ) \Delta x$. 
Figure \ref{fig:RP4_entropy1} shows the evolution of ${\mathcal E}_{\rm tot} (t)$ computed by 
{\tt EC6} and {\tt ES5}, respectively, with fourth-order TVD explicit Runge-Kutta time discretization and different CFL numbers (correspond to different time step-sizes). 
For {\tt ES5}, we see that the entropy dissipates at almost the same rates for different time step-sizes, as expected. 
Although the semi-discrete scheme {\tt EC6} conserves the entropy in theory, 
the results of the fully discrete scheme show small-magnitude entropy dissipation, 
which is solely introduced by time discretization. 
This is evidenced by the observation that the dissipation magnitude for {\tt EC6} becomes smaller when smaller CFL numbers are used. 
To further verify that the entropy conservative property converges in time with the temporal order, we 
show the logarithmic plots of the decrement of the total entropy in Figure \ref{fig:RP4_entropy2}. 
Again, Figure \ref{fig:RP4_entropy2}(a) clearly shows that the decrement magnitude becomes smaller for smaller time step-sizes. 
Moreover, Figure \ref{fig:RP4_entropy2}(b) indicates that the entropy dissipates at a rate of order ${\mathcal O}(\Delta t^5)$, ranging from $\sim 10^{-3}$ with ${\rm CFL}=0.4$ to $\sim 10^{-8}$ with ${\rm CFL}=0.05$. 
This demonstrates the convergence of the entropy conservative property and confirms that the semi-discrete scheme {\tt EC6} does conserve entropy.

\begin{figure}[htbp]
	\centering
	\subfigure[Entropy conservative scheme {\tt EC6}]{\includegraphics[width=0.48\textwidth]{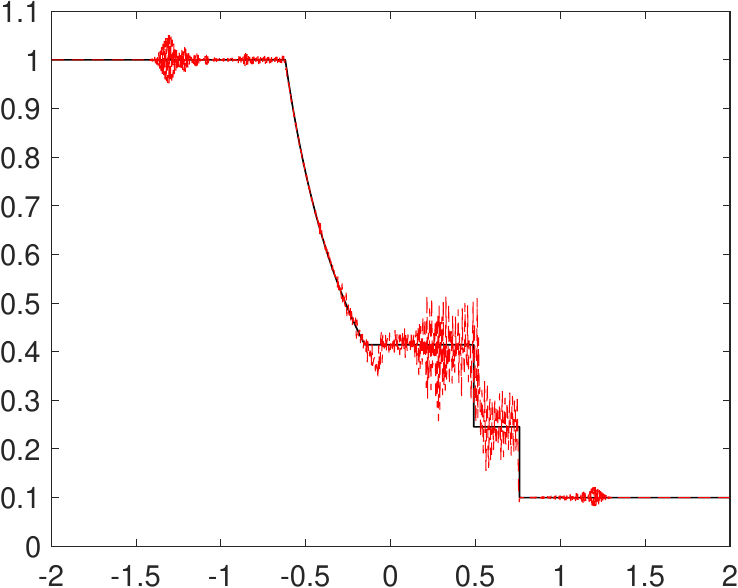}}
	\subfigure[Entropy stable scheme {\tt ES5}]{\includegraphics[width=0.48\textwidth]{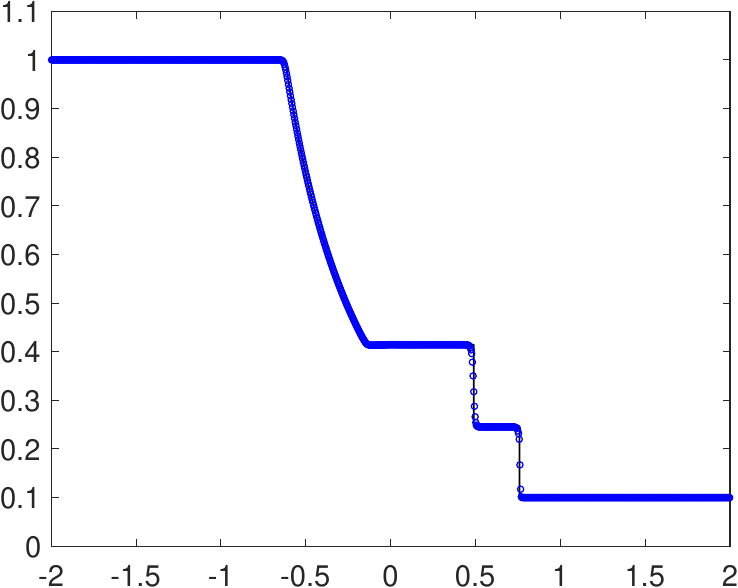}}
	\caption{\small 
		The density $\rho$ at $t=1$ for RP \uppercase\expandafter{\romannumeral4} in Example \ref{example1DRiemanns}. 
		The dash line and the symbols ``${\color{blue}\circ}$'' denote the numerical solutions 
		 computed by {\tt EC6} and {\tt ES5}, respectively, while 
		the solid lines denote the exact solutions obtained by the RP solver in \cite{giacomazzo_rezzolla_2006}.}
	\label{fig:RP4}
\end{figure}

\begin{figure}[htbp]
	\centering
	{\includegraphics[width=0.48\textwidth]{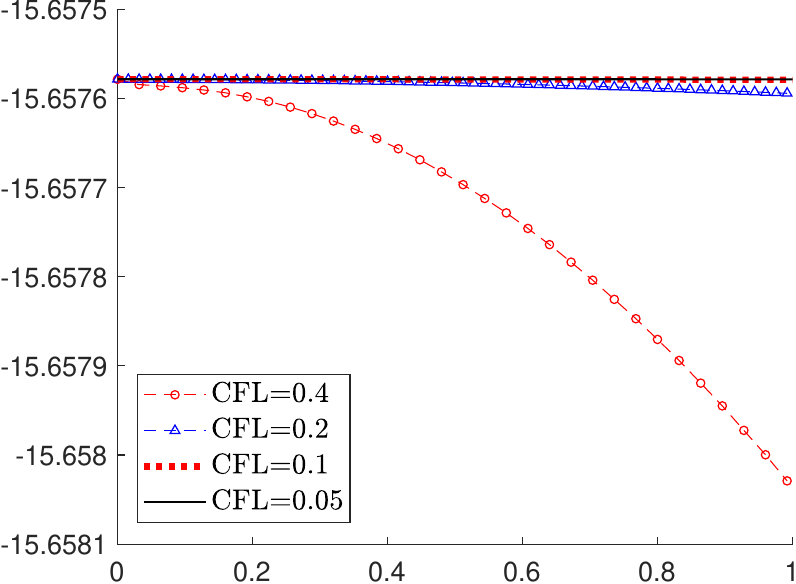}}
	{\includegraphics[width=0.48\textwidth]{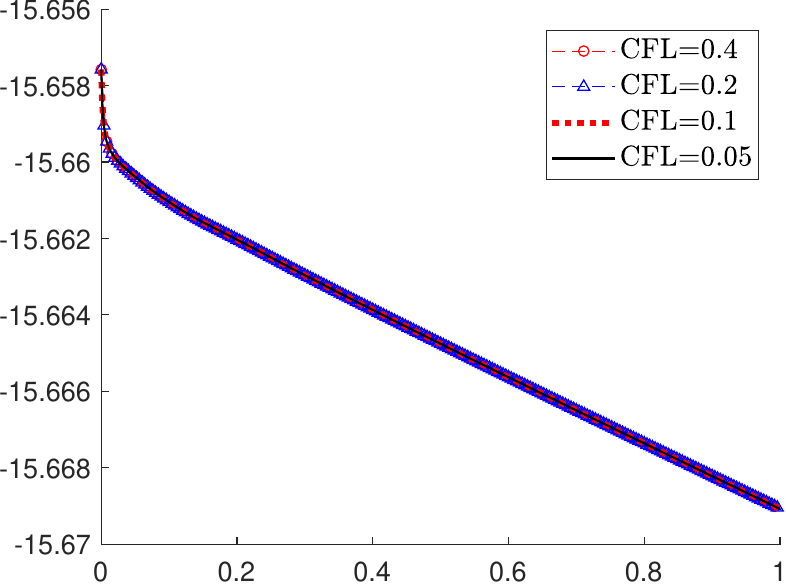}}
	\caption{\small 
		Time evolution of discrete total entropy for {\tt EC6} (left) and {\tt ES5} (right) with  fourth-order TVD Runge-Kutta time discretization and different CFL numbers.}
	\label{fig:RP4_entropy1}
\end{figure}

\begin{figure}[htbp]
	\centering
	\subfigure[Evolution of ${\mathcal E}_{\rm tot} (t) - {\mathcal E}_{\rm tot} (0)$    ]{\includegraphics[width=0.48\textwidth]{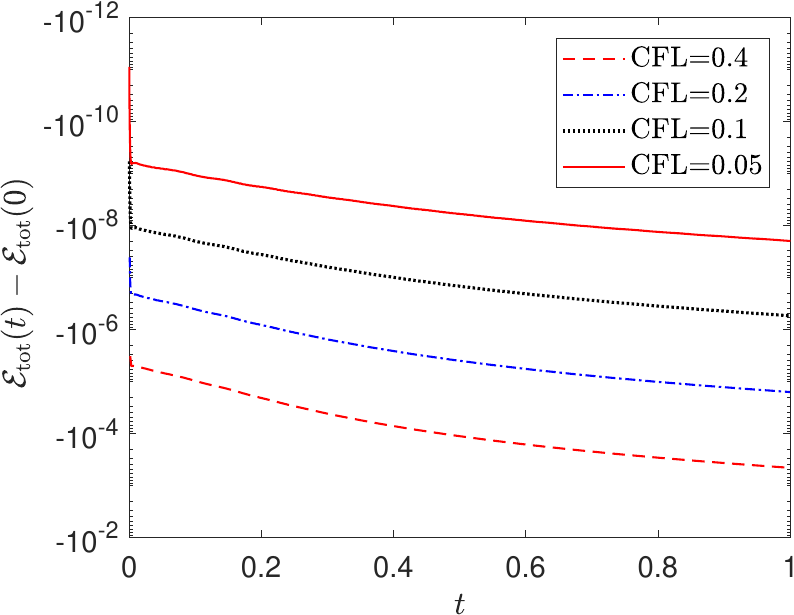}}~~~
	\subfigure[Decrement $|{\mathcal E}_{\rm tot} (1) - {\mathcal E}_{\rm tot} (0)|$ at $t=1$ ]{\includegraphics[width=0.48\textwidth]{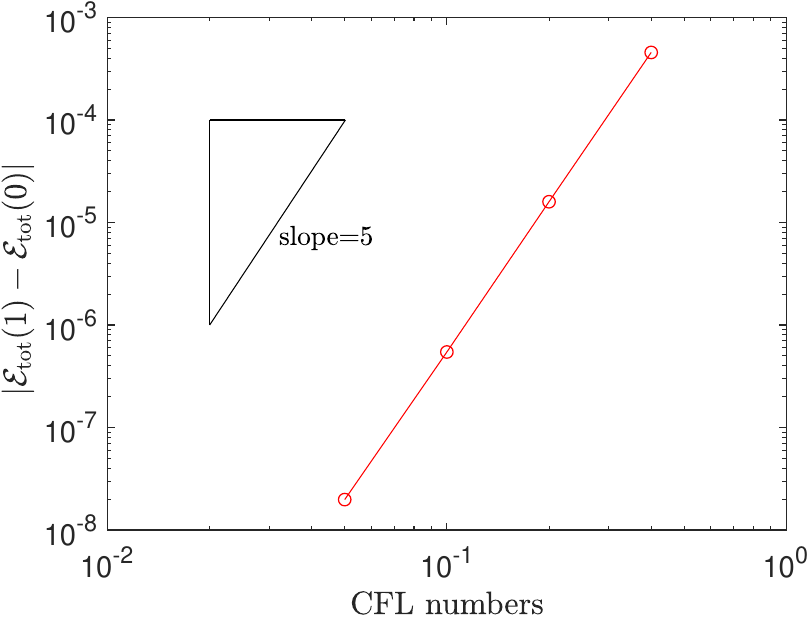}}%Entropy1DRPEC6_CFL0p04
	\caption{\small 
		The decrement of total entropy for {\tt EC6} with fourth-order TVD Runge-Kutta time discretization and different CFL numbers.}
	\label{fig:RP4_entropy2}
\end{figure}

\end{expl}

%%%%%%%%%%%%%%%%%%%%%%% Example 5: BL %%%%%%%%%%%%%%%%%%%%%%%%%%

\begin{expl}[Blast problem]\label{example2DBL}\rm
	Blast problem is a benchmark test for RMHD numerical schemes. 
	Our setup is the same as in \cite{MignoneHLLCRMHD,Zanotti2015}. Initially, the computational domain $[-6,6]^2$ is filled with a homogeneous gas at rest with adiabatic index $\Gamma=4/3$.
	The explosion zone ($r<0.8$) has a density of $10^{-2}$ and a pressure of $1$, while the ambient medium ($r>1$) has
	a low density of $10^{-4}$ and a low pressure of $5\times 10^{-4}$, where $r=\sqrt{x^2+y^2}$.
	A linear taper is applied to the density and pressure for $r\in[0.8,1]$. The magnetic field is initialized in the $x$-direction as $0.1$.

	Our numerical results at $t=4$, obtained by using {\tt ES5} on the mesh of $400\times 400$ uniform cells, are shown in Figure~\ref{fig:BL}. 
	We observe that the wave pattern of the configuration is composed by two main waves,
	an external fast   and a reverse shock waves. The former is almost circular, while the latter is somewhat elliptic.
	The magnetic field is essentially confined between them, while the inner region is almost devoid of magnetization. 
	Our numerical results agree quite well with those reported in \cite{Zanotti2015,WuTangM3AS}. To validate the entropy stability of {\tt ES5}, 
	we compute the total
	entropy $\sum_{i,j} {\mathcal E} ( {\bf U}_{ij} (t) ) \Delta x \Delta y $, 
	which should decrease with time if the scheme
	is entropy stable. 
	The evolution of total entropy is displayed in the left %figure 
	of Figure~\ref{fig:Et} obtained by using {\tt ES5} at different gird resolutions. 
	We clearly see a
	monotonic decay which indicates that the fully discrete scheme {\tt ES5} is entropy stable.

		\begin{figure}[htbp]
		\centering
		{\includegraphics[width=0.48\textwidth]{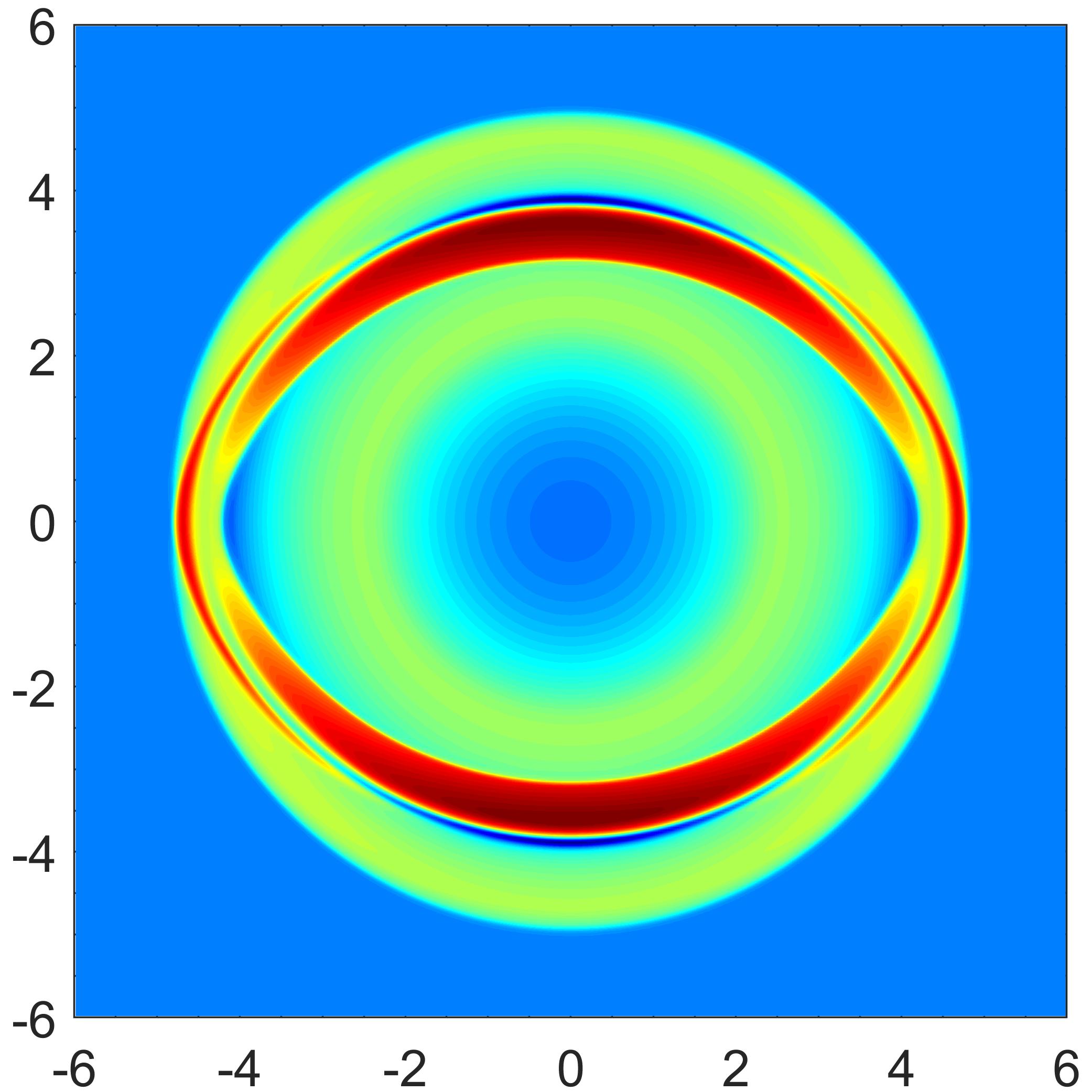}}
		{\includegraphics[width=0.48\textwidth]{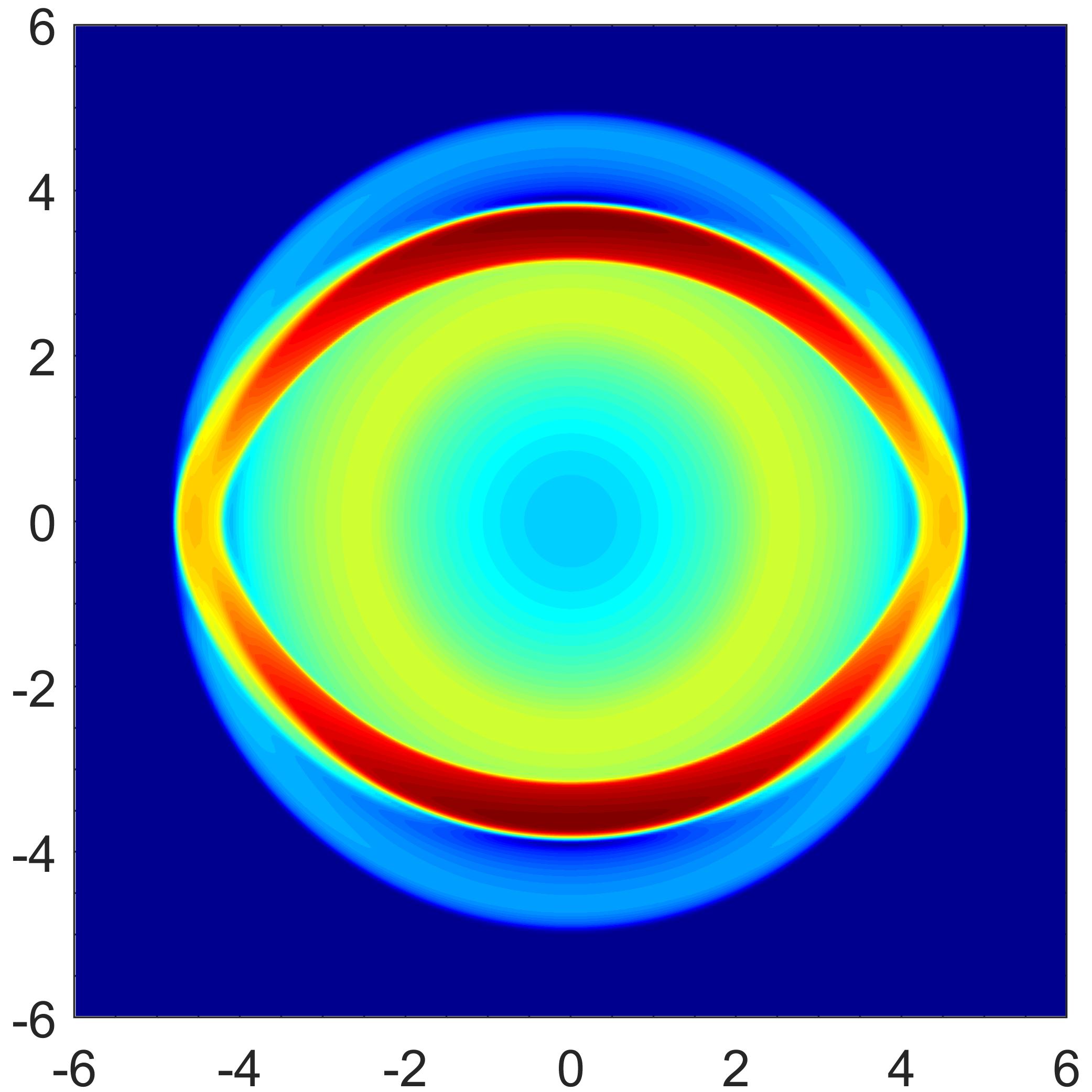}}
		{\includegraphics[width=0.48\textwidth]{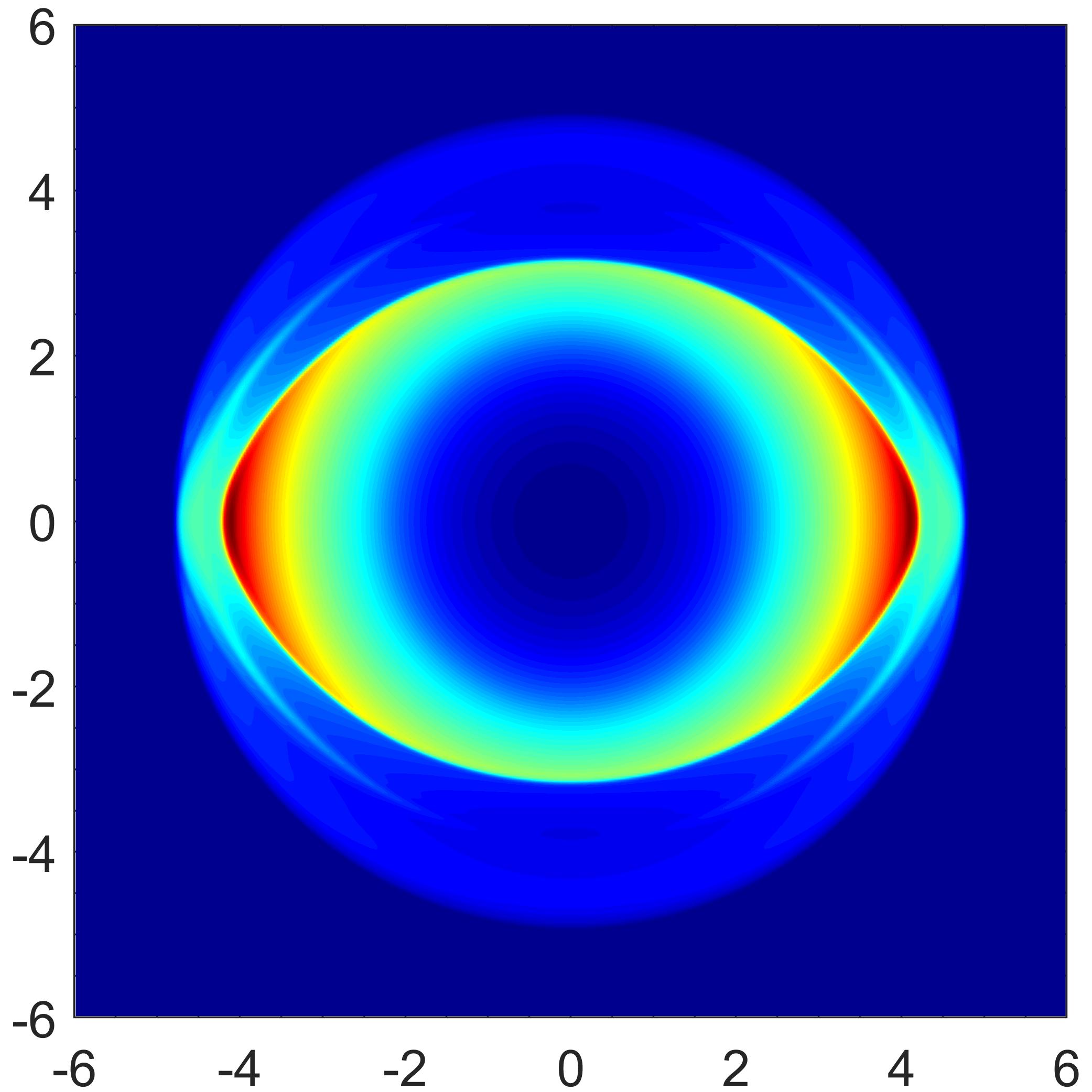}}
		{\includegraphics[width=0.48\textwidth]{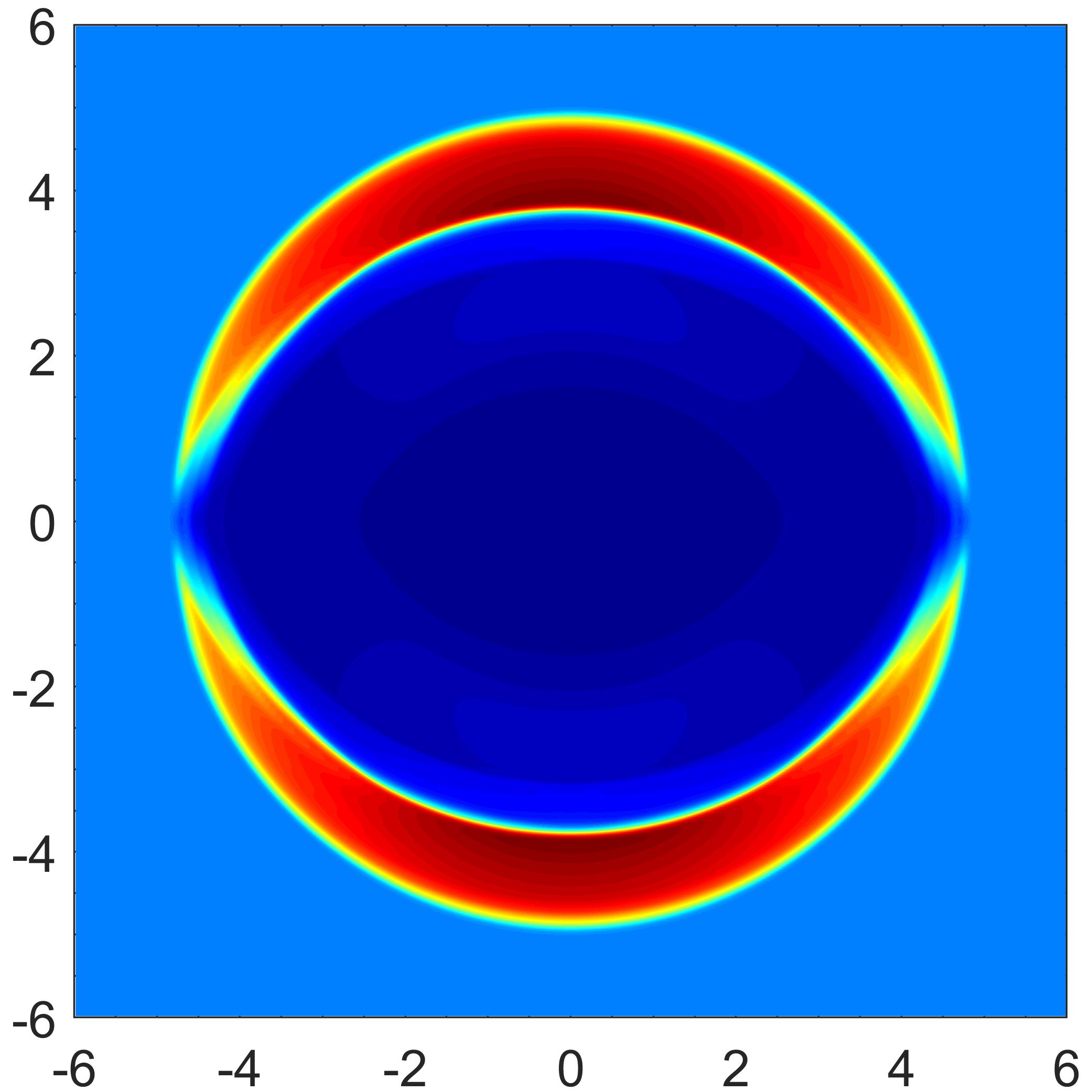}}
		\caption{\small Example \ref{example2DBL}: the schlieren images of %rest-mass 
			density logarithm (top-left), gas pressure (top-right), Lorentz factor (bottom-left) and magnetic field strength
			(bottom-right) at $t = 4$.}
		\label{fig:BL}
	\end{figure}

\end{expl}

%%%%%%%%%%%%%%%%%%%%%%% Example 6: OT %%%%%%%%%%%%%%%%%%%%%%%%%%

\begin{expl}[Orszag-Tang problem]\label{example2DOT}\rm
	This test simulates the relativistic version \cite{Host:2008} of the classical Orszag-Tang problem \cite{orszag1979small}. 
	Our setup is the same as in \cite{Host:2008}. Initially, the computational domain $[0,2\pi]^2$ is filled 
	with 
	hot gas.
	We set the adiabatic index $\Gamma = 4/3$, 
	the initial pressure $p=10$, and the rest-mass density $\rho = 1$. 
	The initial velocity field of the fluid is given by 
	$
	{\bf v}(x,y,0)=   ( - A \sin(y), A\sin(x), 0   )^\top,
	$
	where the parameter $A=0.99/\sqrt{2}$ so that the maximum velocity is $0.99$ (the corresponding Lorentz factor is about $7$). The magnetic field is initialized at ${\bf B}(x,y,0)=(-\sin y, \sin(2x), 0)^\top$. Periodic conditions are specified at all the boundaries. 
	Although the solution of this problem is smooth initially,
	complicated wave structures are formed as the time increases, and turbulence
	behavior will be produced eventually.

	Figure~\ref{fig:OT} gives the numerical results obtained by using {\tt ES5} on $600\times 600$ uniform grids. 
	In comparison with the results in \cite{Host:2008}, 
	the complicated flow structures are well captured by {\tt ES5} with high resolution. 
	To validate the entropy stability, the evolution of total entropy is shown in the right %figure
	 of Figure~\ref{fig:Et} obtained by using {\tt ES5} at different gird resolutions. We observe that the total entropy 
	remains constant 
	at initial times because initially the solution is smooth, and it starts to decrease at $t\approx 2$ when discontinuities start to form, as expected.

		\begin{figure}[htbp]
		\centering
		{\includegraphics[width=0.48\textwidth]{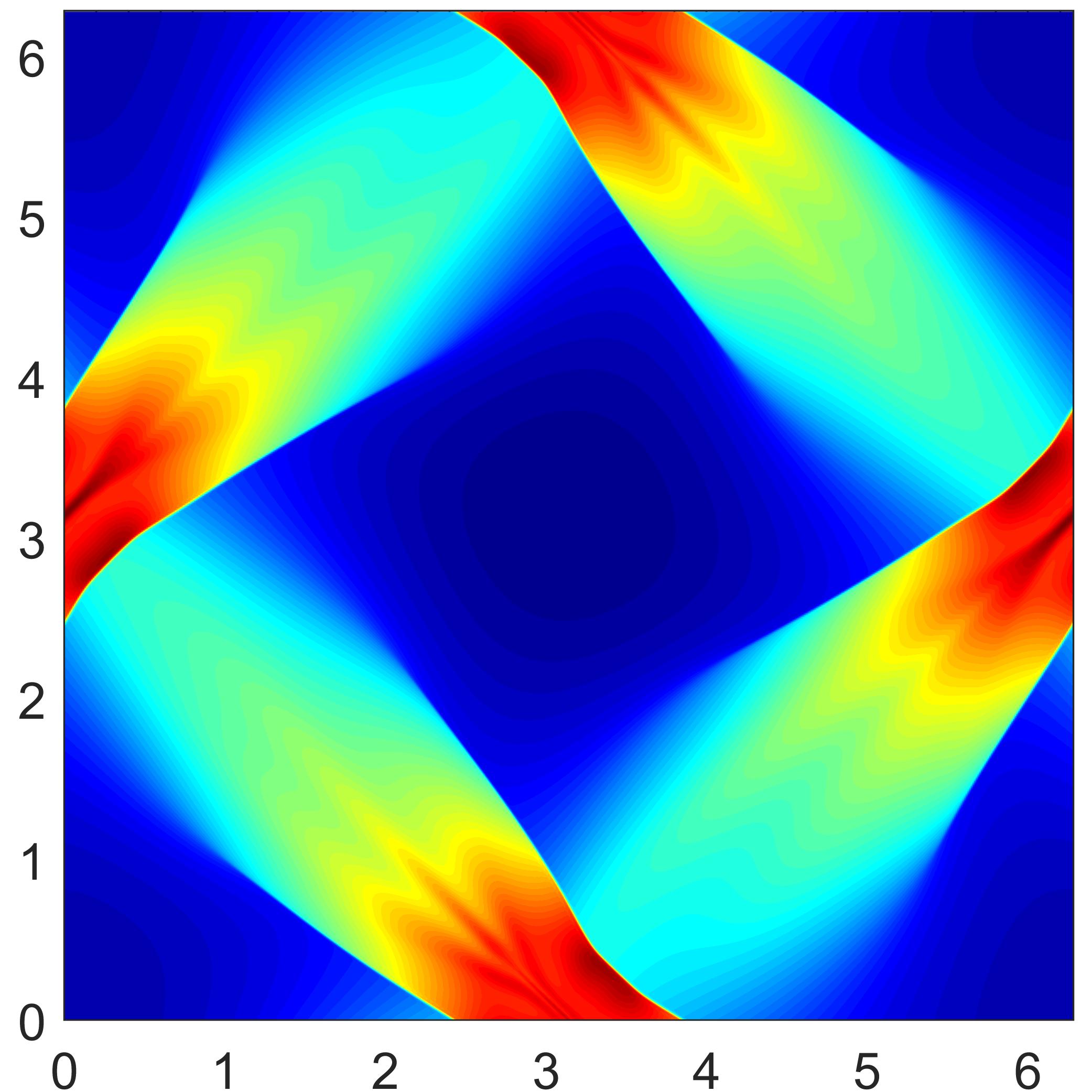}}
		{\includegraphics[width=0.48\textwidth]{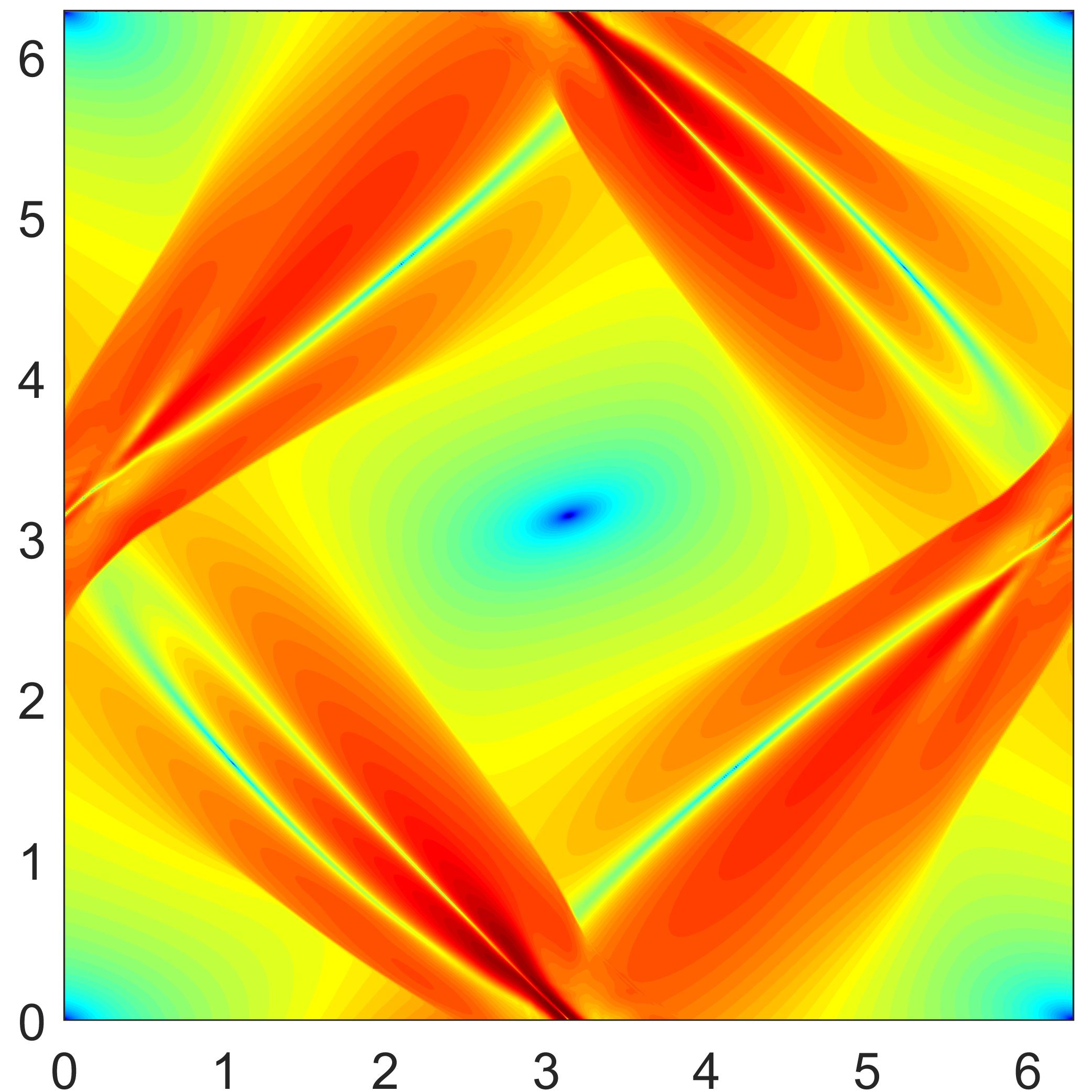}}
		{\includegraphics[width=0.48\textwidth]{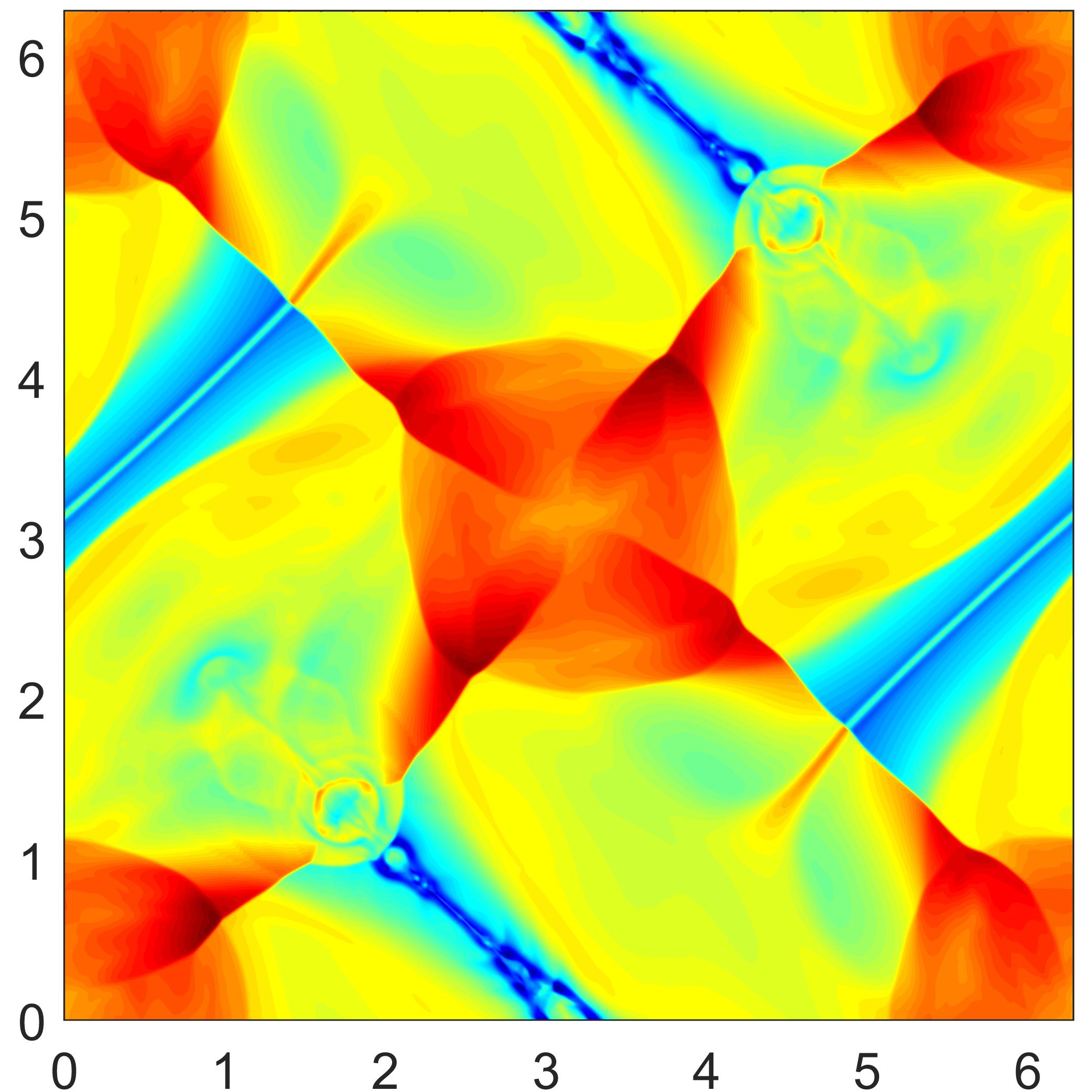}}
		{\includegraphics[width=0.48\textwidth]{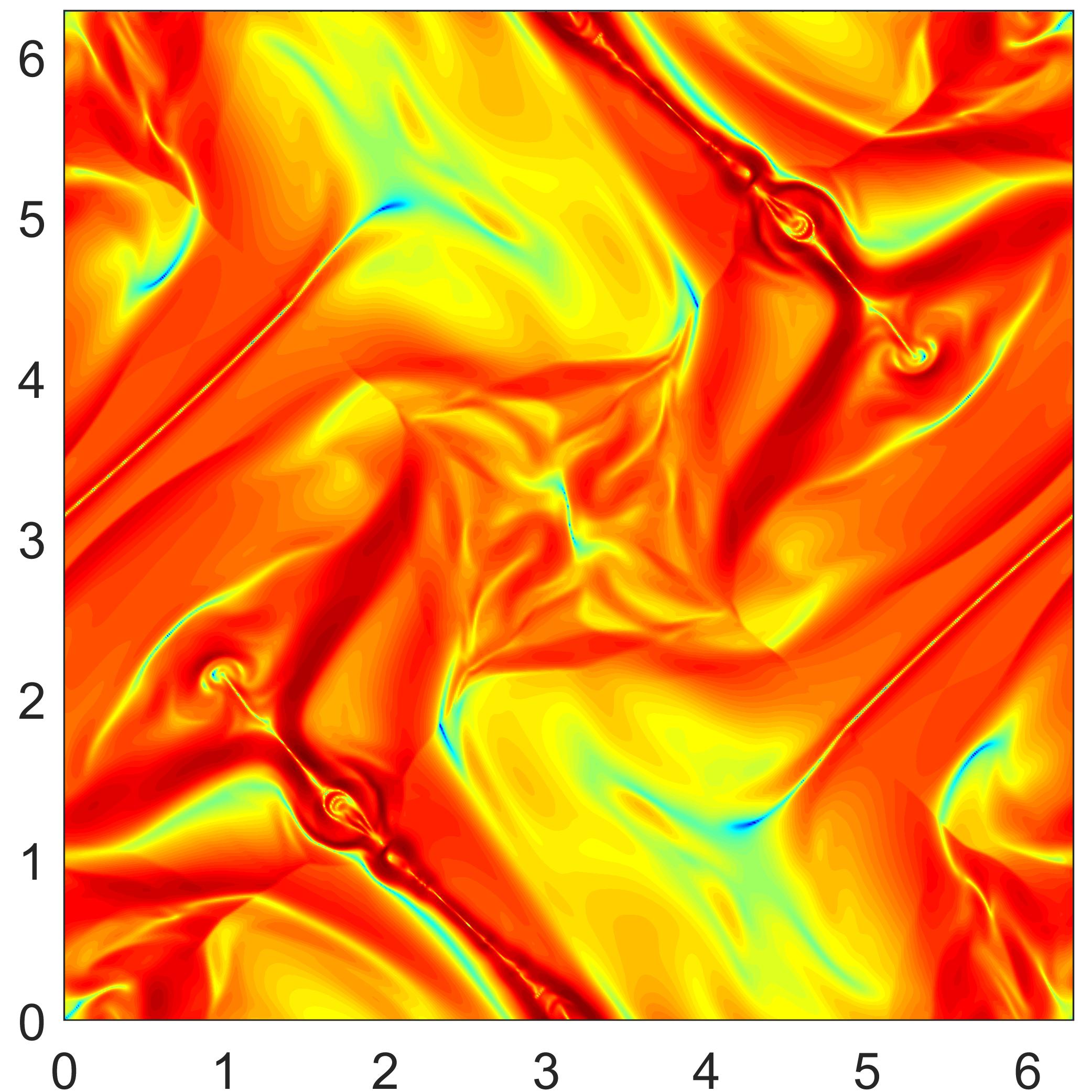}}
		\caption{\small Example \ref{example2DOT}: the schlieren images of rest-mass density logarithm (left) and magnetic pressure logarithm (right). Top: $t=2.818127$; 
			%$t = 4.6$; 
			bottom: $t=6.8558$.} %,
		%obtained by using {\tt ES5} with $600\times 600$ grids
		\label{fig:OT}
	\end{figure}

	\begin{figure}[htbp]
		\centering
		\includegraphics[width=0.48\textwidth]{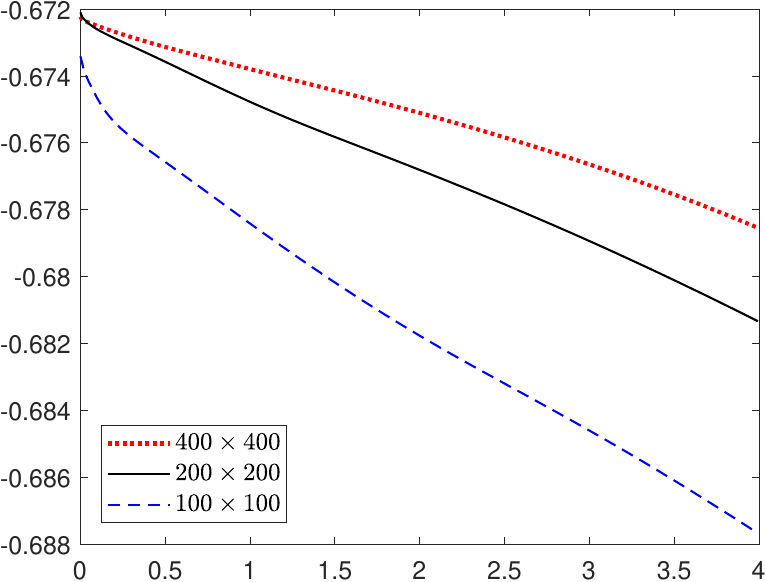}
		\includegraphics[width=0.48\textwidth]{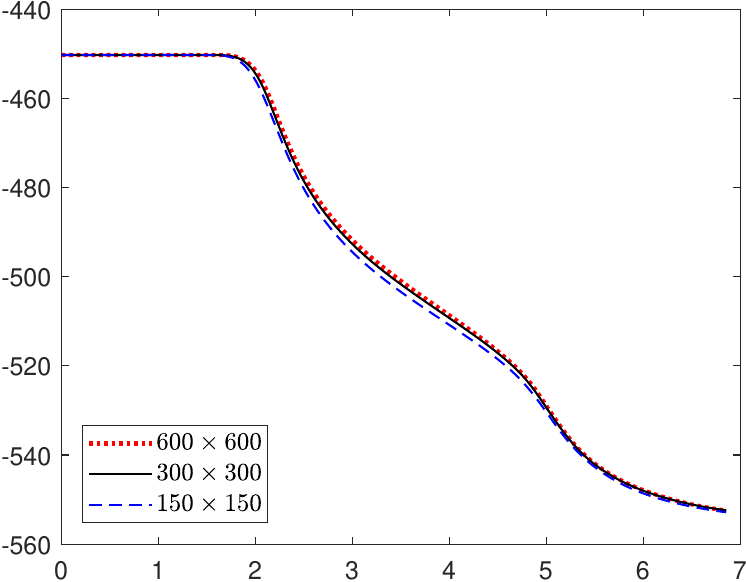}
		\caption{\small
			Evolution of total entropy for the blast problem (left) and the Orszag-Tang problem (right), obtained by using ${\tt ES 5}$ at different grid resolutions. 
		}\label{fig:Et} 
	\end{figure}

\end{expl}

%%%%%%%%%%%%%%%%%%%%%%% Example 7: SC %%%%%%%%%%%%%%%%%%%%%%%%%%

\begin{expl}[Shock cloud interaction problem]\label{example2DSC}{\rm
		This problem describes the disruption of a high density cloud by a strong shock wave. Our setup is the same as in \cite{HeTang2012RMHD}.
		%Different from the setup in \cite{MignoneHLLCRMHD}, the magnetic field is not orthogonal to the slab plane so that the
		%magnetic divergence-free treatment has to be imposed.
		The computational domain is $[-0.2,1.2]\times [0,1]$, with the left boundary specified as
		inflow condition and the others as outflow conditions. Initially, a shock wave moves to the right from $x=0.05$,
		with the left and right states  
		\begin{align*}
		& {\vec V}_L = (3.86859,0.68,0,0,0,0.84981,-0.84981,1.25115)^\top,
		\\
		& {\vec V}_R = (1,0,0,0,0,0.16106,0.16106,0.05)^\top,
		\end{align*}
		respectively. 
		There exists a rest circular cloud centred at the point $(0.25,0.5)$ with radius 0.15.  The cloud has the same states to the surrounding fluid except for a higher density of 30. Figure~\ref{fig:SC} displays the schlieren images of rest-mass density logarithm and magnetic pressure logarithm at $t=1.2$ obtained by using {\tt ES5} with $560\times 400$ uniform cells. One can see that the discontinuities are captured with high resolution, and the results agree well with those in \cite{HeTang2012RMHD,ZhaoTang2017,WuTangM3AS}.

	\begin{figure}[htbp]
		\centering
		{\includegraphics[width=0.49\textwidth]{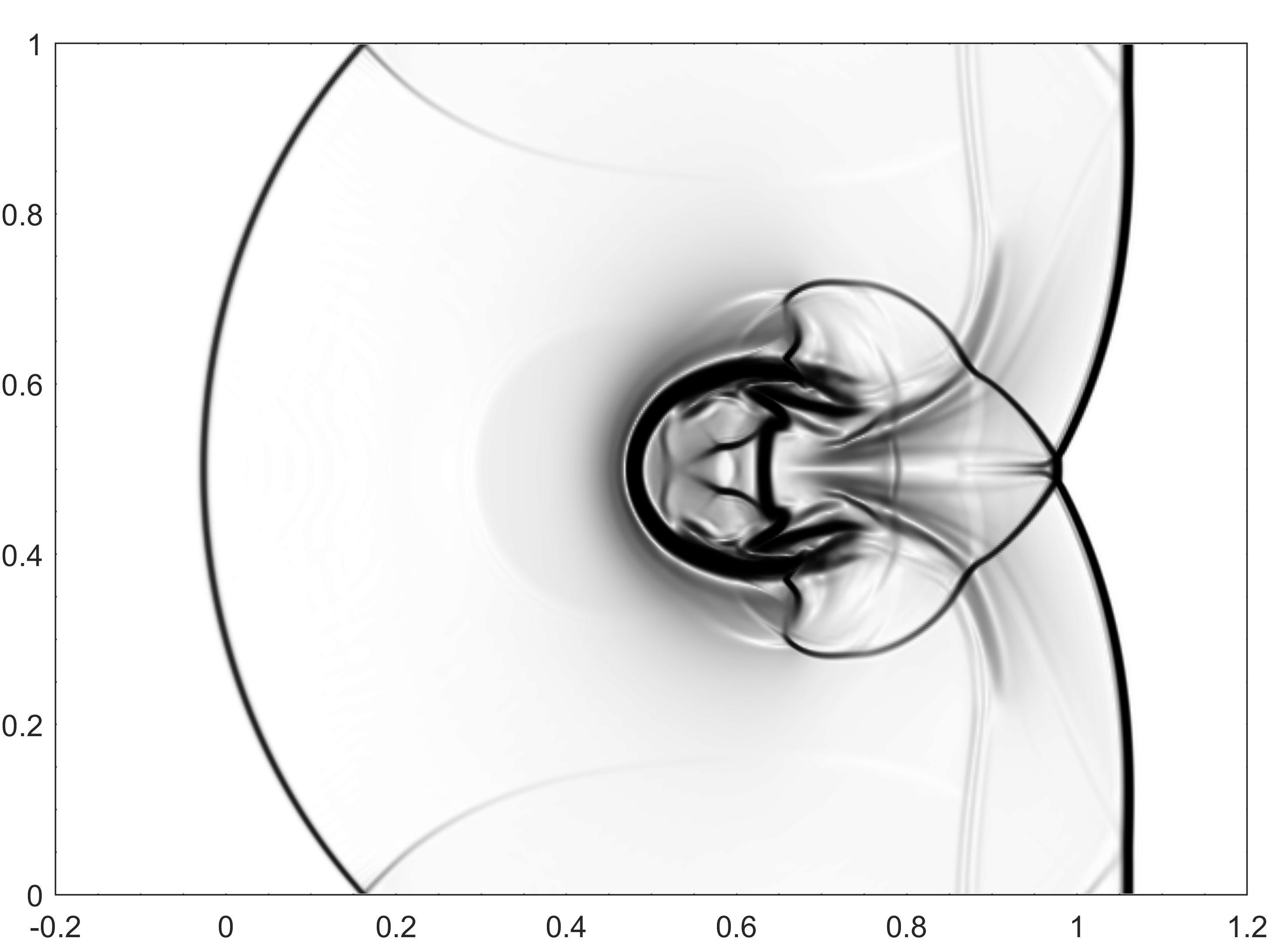}}
		{\includegraphics[width=0.49\textwidth]{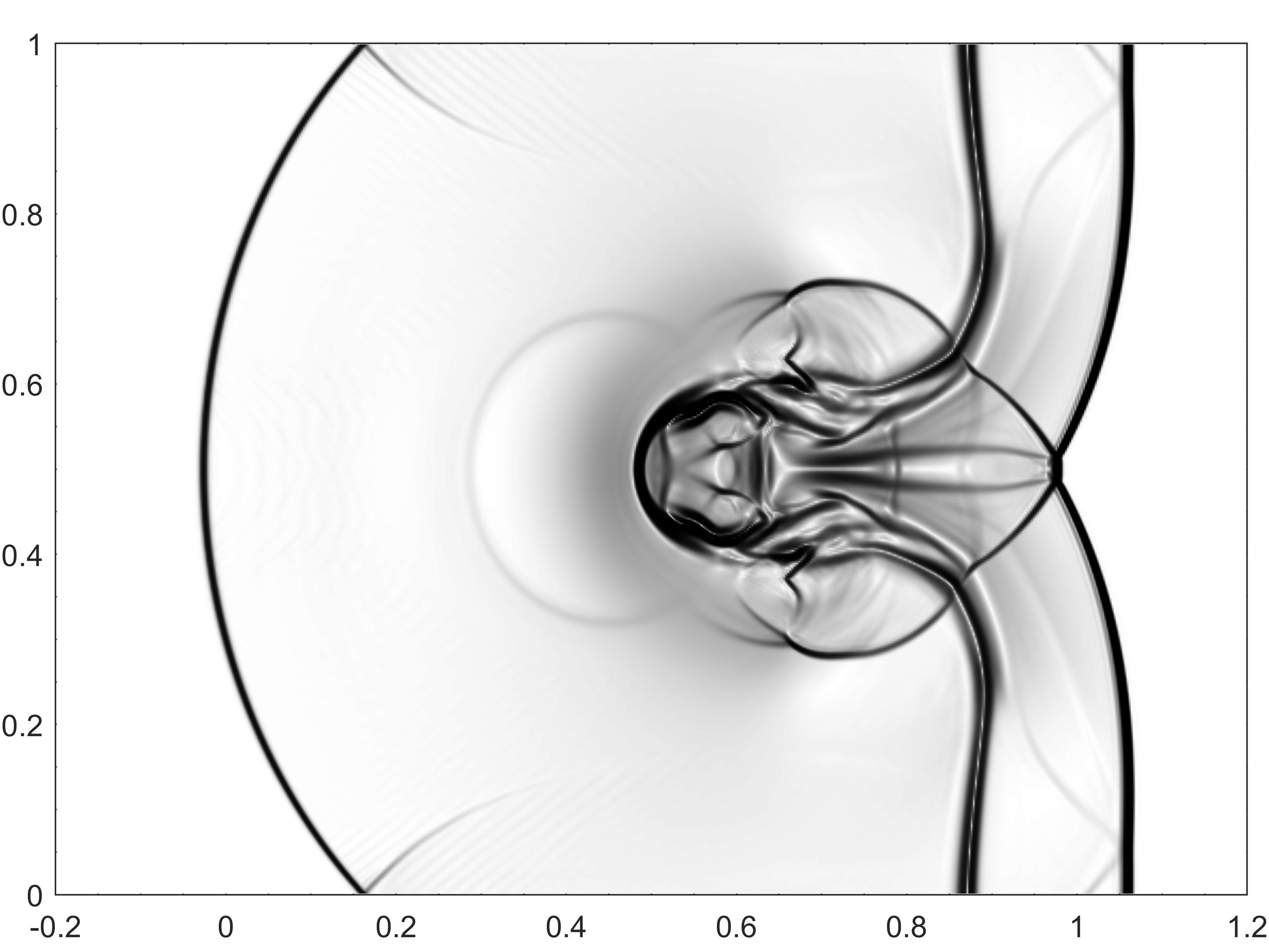}}
		\caption{\small Example \ref{example2DSC}: the schlieren images of rest-mass density logarithm (left) and magnetic
			pressure logarithm (right) at time $t = 1.2$.}
		\label{fig:SC}
	\end{figure}

	}
\end{expl}

\section{Conclusions}\label{sec:conclusion}
In this paper, we presented rigorous entropy analysis and developed high-order accurate entropy stable schemes for the RMHD equations. 
We proved that the conservative RMHD equations \eqref{eq:RMHD} are not symmetrizable and thus do not admit an entropy pair. To address this issue, we first proposed a symmetrizable RMHD system \eqref{ModRMHD} by building the divergence-free condition \eqref{eq:DivB} into 
the RMHD equations through an additional source term. 
Based on the symmetrizable RMHD system, high-order accurate entropy stable finite difference schemes are developed on Cartesian meshes. 
These schemes are built on 
affordable explicit entropy conservative fluxes that are technically derived through carefully selected parameter variables, 
a special high-order discretization of the source term in the symmetrizable RMHD system, and 
suitable high-order dissipative operators based on (weighted) essentially non-oscillatory reconstruction to ensure the entropy stability.  
The accuracy and robustness of the proposed entropy stable schemes were demonstrated by benchmark numerical RMHD examples. 
{\em The symmetrizable RMHD system is 
	a relativistic analogue to 
	the non-relativistic Godunov--Powell system \cite{Godunov1972,Powell1994}. Its discovery would also enable one to generalize a large class of related numerical methods, including some other entropy stable schemes (e.g.~\cite{LiuShuZhang2017}) and Powell's eight-wave methods, to RMHDs.}

\bibliographystyle{siamplain}
\bibliography{references}

\begin{thebibliography}{10}

\bibitem{anile1987mathematical}
{\sc A.~Anile and S.~Pennisi}, {\em On the mathematical structure of test
  relativistic magnetofluiddynamics}, in Annales de l'IHP Physique
  th{\'e}orique, vol.~46, 1987, pp.~27--44.

\bibitem{anton2010relativistic}
{\sc L.~Ant{\'o}n, J.~A. Miralles, J.~M. Mart{\'\i}, J.~M. Ib{\'a}{\~n}ez,
  M.~A. Aloy, and P.~Mimica}, {\em Relativistic magnetohydrodynamics:
  renormalized eigenvectors and full wave decomposition {Riemann} solver},
  Astrophys. J. Suppl. Ser., 188 (2010), p.~1.

\bibitem{Balsara2001}
{\sc D.~Balsara}, {\em Total variation diminishing scheme for relativistic
  magnetohydrodynamics}, Astrophys. J. Suppl. Ser., 132 (2001), p.~83.

\bibitem{Balsara2004}
{\sc D.~S. Balsara}, {\em Second-order-accurate schemes for
  magnetohydrodynamics with divergence-free reconstruction}, Astrophys. J.
  Suppl. Ser., 151 (2004), pp.~149--184.

\bibitem{Barth1998}
{\sc T.~Barth}, {\em Numerical methods for gasdynamic systems on unstructured
  meshes}, in An introduction to recent developments in theory and numerics for
  conservation laws, Springer, 1999, pp.~195--285.

\bibitem{Barth2006}
{\sc T.~Barth}, {\em On the role of involutions in the discontinuous {Galerkin}
  discretization of {Maxwell} and magnetohydrodynamic systems}, in Compatible
  spatial discretizations, Springer, 2006, pp.~69--88.

\bibitem{biswas2018low}
{\sc B.~Biswas and R.~K. Dubey}, {\em Low dissipative entropy stable schemes
  using third order {WENO} and {TVD} reconstructions}, Adv. Comput. Math., 44
  (2018), pp.~1153--1181.

\bibitem{bouchut1996muscl}
{\sc F.~Bouchut, C.~Bourdarias, and B.~Perthame}, {\em A {MUSCL} method
  satisfying all the numerical entropy inequalities}, Math. Comp., 65 (1996),
  pp.~1439--1461.

\bibitem{Brackbill1980}
{\sc J.~U. Brackbill and D.~C. Barnes}, {\em The effect of nonzero {$\nabla
  \cdot {\bf {B}}$} on the numerical solution of the magnetodydrodynamic
  equations}, J. Comput. Phys., 35 (1980), pp.~426--430.

\bibitem{carpenter2014entropy}
{\sc M.~H. Carpenter, T.~C. Fisher, E.~J. Nielsen, and S.~H. Frankel}, {\em
  Entropy stable spectral collocation schemes for the {Navier--Stokes}
  equations: Discontinuous interfaces}, SIAM J. Sci. Comput., 36 (2014),
  pp.~B835--B867.

\bibitem{chandrashekar2013kinetic}
{\sc P.~Chandrashekar}, {\em Kinetic energy preserving and entropy stable
  finite volume schemes for compressible {Euler} and {Navier-Stokes}
  equations}, Commun. Comput. Phys., 14 (2013), pp.~1252--1286.

\bibitem{Chandrashekar2016}
{\sc P.~Chandrashekar and C.~Klingenberg}, {\em Entropy stable finite volume
  scheme for ideal compressible {MHD} on {2-D} {Cartesian} meshes}, SIAM J.
  Numer. Anal., 54 (2016), pp.~1313--1340.

\bibitem{chen2017entropy}
{\sc T.~Chen and C.-W. Shu}, {\em Entropy stable high order discontinuous
  {Galerkin} methods with suitable quadrature rules for hyperbolic conservation
  laws}, J. Comput. Phys., 345 (2017), pp.~427--461.

\bibitem{crandall1980monotone}
{\sc M.~G. Crandall and A.~Majda}, {\em Monotone difference approximations for
  scalar conservation laws}, Math. Comp., 34 (1980), pp.~1--21.

\bibitem{dafermos2000hyperbolic}
{\sc C.~M. Dafermos}, {\em Hyperbolic Conservation Laws in Continuum Physics},
  vol.~325, Springer, 4th~ed., 2010.

\bibitem{Dedner2002}
{\sc A.~Dedner, F.~Kemm, D.~Kr{\"o}ner, C.-D. Munz, T.~Schnitzer, and
  M.~Wesenberg}, {\em Hyperbolic divergence cleaning for the {MHD} equations},
  J. Comput. Phys., 175 (2002), pp.~645--673.

\bibitem{derigs2016novel}
{\sc D.~Derigs, A.~R. Winters, G.~J. Gassner, and S.~Walch}, {\em A novel
  high-order, entropy stable, {3D AMR MHD} solver with guaranteed positive
  pressure}, J. Comput. Phys., 317 (2016), pp.~223--256.

\bibitem{DERIGS2018420}
{\sc D.~Derigs, A.~R. Winters, G.~J. Gassner, S.~Walch, and M.~Bohm}, {\em
  Ideal {GLM-MHD}: {About} the entropy consistent nine-wave magnetic field
  divergence diminishing ideal magnetohydrodynamics equations}, J. Comput.
  Phys., 364 (2018), pp.~420--467.

\bibitem{duan2019high}
{\sc J.~Duan and H.~Tang}, {\em High-order accurate entropy stable finite
  difference schemes for one-and two-dimensional special relativistic
  hydrodynamics}, arXiv:1905.06092,  (2019).

\bibitem{Evans1988}
{\sc C.~R. Evans and J.~F. Hawley}, {\em Simulation of magnetohydrodynamic
  flows: a constrained transport method}, Astrophys. J., 332 (1988),
  pp.~659--677.

\bibitem{fisher2013high}
{\sc T.~C. Fisher and M.~H. Carpenter}, {\em High-order entropy stable finite
  difference schemes for nonlinear conservation laws: Finite domains}, J.
  Comput. Phys., 252 (2013), pp.~518--557.

\bibitem{FMT09}
{\sc U.~S. Fjordholm, S.~Mishra, and E.~Tadmor}, {\em {Energy preserving and
  energy stable schemes for the shallow water equations}}, in {Foundations of
  Computational Mathematics, Hong Kong 2008}, vol.~363 of {London Math. Soc.
  Lecture Notes Series}, 2009, pp.~93--139.

\bibitem{fjordholm2012arbitrarily}
{\sc U.~S. Fjordholm, S.~Mishra, and E.~Tadmor}, {\em Arbitrarily high-order
  accurate entropy stable essentially nonoscillatory schemes for systems of
  conservation laws}, SIAM J. Numer. Anal., 50 (2012), pp.~544--573.

\bibitem{fjordholm2013eno}
{\sc U.~S. Fjordholm, S.~Mishra, and E.~Tadmor}, {\em {ENO} reconstruction and
  {ENO} interpolation are stable}, Found. Comput. Math., 13 (2013),
  pp.~139--159.

\bibitem{fjordholm2016sign}
{\sc U.~S. Fjordholm and D.~Ray}, {\em A sign preserving {WENO} reconstruction
  method}, J. Sci. Comput., 68 (2016), pp.~42--63.

\bibitem{font2008}
{\sc J.~A. Font}, {\em Numerical hydrodynamics and magnetohydrodynamics in
  general relativity}, Living Rev. Relativ., 11 (2008), p.~7.

\bibitem{freistuhler2013symmetrizations}
{\sc H.~Freist{\"u}hler and Y.~Trakhinin}, {\em Symmetrizations of {RMHD}
  equations and stability of relativistic current--vortex sheets}, Classical
  and Quantum Gravity, 30 (2013), p.~085012.

\bibitem{gassner2013skew}
{\sc G.~J. Gassner}, {\em A skew-symmetric discontinuous {Galerkin} spectral
  element discretization and its relation to {SBP-SAT} finite difference
  methods}, SIAM J. Sci. Comput., 35 (2013), pp.~A1233--A1253.

\bibitem{gassner2016well}
{\sc G.~J. Gassner, A.~R. Winters, and D.~A. Kopriva}, {\em A well balanced and
  entropy conservative discontinuous {Galerkin} spectral element method for the
  shallow water equations}, Appl. Math. Comput., 272 (2016), pp.~291--308.

\bibitem{godlewski2013numerical}
{\sc E.~Godlewski and P.-A. Raviart}, {\em Numerical approximation of
  hyperbolic systems of conservation laws}, vol.~118, Springer Science \&
  Business Media, 2013.

\bibitem{Godunov1972}
{\sc S.~K. Godunov}, {\em Symmetric form of the equations of
  magnetohydrodynamics}, Numerical Methods for Mechanics of Continuum Medium, 1
  (1972), pp.~26--34.

\bibitem{guercilena2017entropy}
{\sc F.~Guercilena, D.~Radice, and L.~Rezzolla}, {\em Entropy-limited
  hydrodynamics: a novel approach to relativistic hydrodynamics}, Comput.
  Astrophys. Cosmol., 4 (2017), p.~3.

\bibitem{giacomazzo_rezzolla_2006}
{\sc F.~Guercilena and L.~Rezzolla}, {\em The exact solution of the {Riemann}
  problem in relativistic magnetohydrodynamics}, J. Fluid Mech., 562 (2006),
  pp.~223--259.

\bibitem{harten1976finite}
{\sc A.~Harten, J.~M. Hyman, P.~D. Lax, and B.~Keyfitz}, {\em On
  finite-difference approximations and entropy conditions for shocks}, Commun.
  Pure Appl. Math., 29 (1976), pp.~297--322.

\bibitem{HeTang2012RMHD}
{\sc P.~He and H.~Tang}, {\em An adaptive moving mesh method for
  two-dimensional relativistic magnetohydrodynamics}, Comput. Fluids, 60
  (2012), pp.~1--20.

\bibitem{hesthaven2019entropy}
{\sc J.~S. Hesthaven and F.~M{\"o}nkeberg}, {\em Entropy stable essentially
  nonoscillatory methods based on {RBF} reconstruction}, ESAIM. Math. Model
  Numer. Anal., 53 (2019), pp.~925--958.

\bibitem{hiltebrand2014entropy}
{\sc A.~Hiltebrand and S.~Mishra}, {\em Entropy stable shock capturing
  space--time discontinuous galerkin schemes for systems of conservation laws},
  Numer. Math., 126 (2014), pp.~103--151.

\bibitem{ismail2009affordable}
{\sc F.~Ismail and P.~L. Roe}, {\em Affordable, entropy-consistent {Euler} flux
  functions {II: Entropy} production at shocks}, J. Comput. Phys., 228 (2009),
  pp.~5410--5436.

\bibitem{komissarov1999godunov}
{\sc S.~S. Komissarov}, {\em A {Godunov-type} scheme for relativistic
  magnetohydrodynamics}, Mon. Not. R. Astron. Soc., 303 (1999), pp.~343--366.

\bibitem{lefloch2002fully}
{\sc P.~G. Lefloch, J.-M. Mercier, and C.~Rohde}, {\em Fully discrete, entropy
  conservative schemes of arbitrary order}, SIAM J. Numer. Anal., 40 (2002),
  pp.~1968--1992.

\bibitem{Li2005}
{\sc F.~Li and C.-W. Shu}, {\em Locally divergence-free discontinuous
  {Galerkin} methods for {MHD} equations}, J. Sci. Comput., 22 (2005),
  pp.~413--442.

\bibitem{Li2011}
{\sc F.~Li, L.~Xu, and S.~Yakovlev}, {\em Central discontinuous {Galerkin}
  methods for ideal {MHD} equations with the exactly divergence-free magnetic
  field}, J. Comput. Phys., 230 (2011), pp.~4828--4847.

\bibitem{LiuShuZhang2017}
{\sc Y.~Liu, C.-W. Shu, and M.~Zhang}, {\em Entropy stable high order
  discontinuous {Galerkin} methods for ideal compressible {MHD} on structured
  meshes}, J. Comput. Phys., 354 (2018), pp.~163--178.

\bibitem{Marti2015}
{\sc J.~M. Mart{\'\i} and E.~M{\"u}ller}, {\em Grid-based methods in
  relativistic hydrodynamics and magnetohydrodynamics}, Living Rev. Comput.
  Astrophys., 1 (2015), p.~3.

\bibitem{MignoneHLLCRMHD}
{\sc A.~Mignone and G.~Bodo}, {\em An {HLLC} riemann solver for relativistic
  flows--{II}. magnetohydrodynamics}, Mon. Not. R. Astron. Soc., 368 (2006),
  pp.~1040--1054.

\bibitem{orszag1979small}
{\sc S.~A. Orszag and C.-M. Tang}, {\em Small-scale structure of
  two-dimensional magnetohydrodynamic turbulence}, J. Fluid Mech., 90 (1979),
  pp.~129--143.

\bibitem{osher1984riemann}
{\sc S.~Osher}, {\em Riemann solvers, the entropy condition, and difference},
  SIAM J. Numer. Anal., 21 (1984), pp.~217--235.

\bibitem{osher1988convergence}
{\sc S.~Osher and E.~Tadmor}, {\em On the convergence of difference
  approximations to scalar conservation laws}, Math. Comp., 50 (1988),
  pp.~19--51.

\bibitem{Powell1994}
{\sc K.~G. Powell}, {\em An approximate {Riemann} solver for
  magnetohydrodynamics (that works in more than one dimension)}, Tech. Report
  ICASE Report No. 94-24, NASA Langley, VA, 1994.

\bibitem{Powell1995}
{\sc K.~G. Powell, P.~Roe, R.~Myong, and T.~Gombosi}, {\em An upwind scheme for
  magnetohydrodynamics}, in 12th Computational Fluid Dynamics Conference, 1995,
  p.~1704.

\bibitem{ranocha2018comparison}
{\sc H.~Ranocha}, {\em Comparison of some entropy conservative numerical fluxes
  for the euler equations}, J. Sci. Comput., 76 (2018), pp.~216--242.

\bibitem{ruggeri1981convex}
{\sc T.~Ruggeri and A.~Strumia}, {\em Convex covariant entropy density,
  symmetric conservative form, and shock waves in relativistic
  magnetohydrodynamics}, J. Math. Phys., 22 (1981), pp.~1824--1827.

\bibitem{tadmor1987numerical}
{\sc E.~Tadmor}, {\em The numerical viscosity of entropy stable schemes for
  systems of conservation laws. {I}}, Math. Comp., 49 (1987), pp.~91--103.

\bibitem{tadmor2003entropy}
{\sc E.~Tadmor}, {\em Entropy stability theory for difference approximations of
  nonlinear conservation laws and related time-dependent problems}, Acta
  Numer., 12 (2003), pp.~451--512.

\bibitem{TadmorZhong2006}
{\sc E.~Tadmor and W.~Zhong}, {\em Entropy stable approximations of
  {Navier--Stokes} equations with no artificial numerical viscosity}, J.
  Hyperbolic Differ. Equ., 03 (2006), pp.~529--559.

\bibitem{Toth2000}
{\sc G.~T{\'o}th}, {\em The {$\nabla \cdot {\bf {B}} = 0$} constraint in
  shock-capturing magnetohydrodynamics codes}, J. Comput. Phys., 161 (2000),
  pp.~605--652.

\bibitem{trakhinin2001stability}
{\sc Y.~L. Trakhinin}, {\em On stability of shock waves in relativistic
  magnetohydrodynamics}, Quarterly of Applied Mathematics, 59 (2001),
  pp.~25--45.

\bibitem{Host:2008}
{\sc B.~van~der Holst, R.~Keppens, and Z.~Meliani}, {\em A multidimensional
  grid-adaptive relativistic magnetofluid code}, Comput. Phys. Commun., 179
  (2008), pp.~617--627.

\bibitem{van1991maxwell}
{\sc M.~H. van Putten}, {\em Maxwell's equations in divergence form for general
  media with applications to {MHD}}, Commun. Math. Phys., 141 (1991),
  pp.~63--77.

\bibitem{van1995two}
{\sc M.~H. van Putten}, {\em A two-dimensional numerical implementation of
  magnetohydrodynamics in divergence form}, SIAM J. Numer. Anal., 32 (1995),
  pp.~1504--1518.

\bibitem{van1996knots}
{\sc M.~H. van Putten}, {\em Knots in simulations of magnetized relativistic
  jets}, Astrophys. J. Lett., 467 (1996), p.~L57.

\bibitem{winters2016affordable}
{\sc A.~R. Winters and G.~J. Gassner}, {\em Affordable, entropy conserving and
  entropy stable flux functions for the ideal {MHD} equations}, J. Comput.
  Phys., 304 (2016), pp.~72--108.

\bibitem{Wu2017a}
{\sc K.~Wu}, {\em Positivity-preserving analysis of numerical schemes for ideal
  magnetohydrodynamics}, SIAM J. Numer. Anal., 56 (2018), pp.~2124--2147.

\bibitem{WuShu2018}
{\sc K.~Wu and C.-W. Shu}, {\em A provably positive discontinuous {Galerkin}
  method for multidimensional ideal magnetohydrodynamics}, SIAM J. Sci.
  Comput., 40 (2018), pp.~B1302--B1329.

\bibitem{WuShu2019}
{\sc K.~Wu and C.-W. Shu}, {\em Provably positive high-order schemes for ideal
  magnetohydrodynamics: analysis on general meshes}, Numer. Math., 142 (2019),
  pp.~995--1047.

\bibitem{WuTangM3AS}
{\sc K.~Wu and H.~Tang}, {\em Admissible states and
  physical-constraints-preserving schemes for relativistic magnetohydrodynamic
  equations}, Math. Models Methods Appl. Sci., 27 (2017), pp.~1871--1928.

\bibitem{Zanotti2015}
{\sc O.~Zanotti, F.~Fambri, and M.~Dumbser}, {\em Solving the relativistic
  magnetohydrodynamics equations with {ADER} discontinuous {Galerkin} methods,
  a posteriori subcell limiting and adaptive mesh refinement}, Mon. Not. R.
  Astron. Soc., 452 (2015), pp.~3010--3029.

\bibitem{ZhaoTang2017}
{\sc J.~Zhao and H.~Tang}, {\em {Runge-Kutta} discontinuous {Galerkin} methods
  for the special relativistic magnetohydrodynamics}, J. Comput. Phys., 343
  (2017), pp.~33--72.

\end{thebibliography}

\end{document}